\documentclass[a4paper]{amsart}

\usepackage[utf8]{inputenc}
\usepackage[T1]{fontenc}
\usepackage[english]{babel}

\usepackage{amsfonts}
\usepackage{amsmath}
\usepackage[alphabetic]{amsrefs}
\usepackage{amssymb}
\usepackage{amstext}
\usepackage{amsthm}

\usepackage{geometry}
\usepackage{fullpage}

\usepackage[shortlabels]{enumitem}
\usepackage{graphicx}
\usepackage{hyperref}
\usepackage{listings}
\usepackage{mathrsfs}
\usepackage{mathtools}
\usepackage{multicol}
\usepackage{shuffle}

\usepackage{tikz}
\usepackage{tikz-cd}
\usetikzlibrary{arrows,backgrounds,chains,decorations.pathreplacing,patterns,shapes, calc, positioning}

\usepackage{xcolor}
\definecolor{purpleish}{RGB}{102,0,102}
\definecolor{pinkish}{RGB}{255,153,204}

\newcommand{\area}{\mathsf{area}}
\newcommand{\dinv}{\mathsf{dinv}}

\newcommand{\inv}{\mathsf{inv}}

\newcommand{\bounce}{\mathsf{bounce}}

\newcommand{\maj}{\mathsf{maj}}
\newcommand{\ides}{\mathsf{ides}}
\newcommand{\rdiagword}{\mathsf{rdiagword}}
\newcommand{\shift}{\mathsf{shift}}
\newcommand{\diagword}{\mathsf{diagword}}
\newcommand{\Yconsec}{\mathsf{Yconsec}}
\newcommand{\touch}{\mathsf{touch}}

\newcommand{\dcomp}{\mathsf{dcomp}}

\newcommand{\D}{\mathsf{D}} % Dyck paths (new)
\newcommand{\W}{\mathsf{W}} % Labellings
\newcommand{\ED}{\mathsf{ED}} % Partial paths starting east
\newcommand{\SQ}{\mathsf{SQ}} % square paths ending east
\newcommand{\LSQ}{\mathsf{LSQ}} % labelled square paths
\newcommand{\LD}{\mathsf{LD}} % labelled Dyck paths
\newcommand{\Pref}{\mathsf{Pref}} % Preference functions
\newcommand{\Park}{\mathsf{Park}} % Preference functions

\newcommand{\qbinom}[2]{\genfrac{[}{]}{0pt}{}{#1}{#2}}

\makeatletter

\pgfkeys{
	/tikz/sharp angle/.code={%
		\pgfsetarrowoptions{sharp >}{#1}%
		\pgfsetarrowoptions{sharp <}{-#1}%
	},
	/tikz/sharp > angle/.code={%
		\pgfsetarrowoptions{sharp >}{#1}%
	},
	/tikz/sharp < angle/.code={%
		\pgfsetarrowoptions{sharp <}{#1}%
	},
	/tikz/sharp protrude/.code=\csname if#1\endcsname\qrr@tikz@sharp@z@-0.05\p@\else\qrr@tikz@sharp@z@\z@\fi,
	/tikz/sharp protrude/.default=true
}

\newdimen\qrr@tikz@sharp@z@
\qrr@tikz@sharp@z@\z@
\pgfarrowsdeclare{sharp >}{sharp >}{%
	\edef\pgf@marshal{\noexpand\pgfutil@in@{and}{\pgfgetarrowoptions{sharp >}}}%
	\pgf@marshal
	\ifpgfutil@in@
	\edef\pgf@tempa{\pgfgetarrowoptions{sharp >}}
	\expandafter\qrr@tikz@sharp@parse\pgf@tempa\@qrr@tikz@sharp@parse
	\else
	\qrr@tikz@sharp@parse\pgfgetarrowoptions{sharp >}and-\pgfgetarrowoptions{sharp >}\@qrr@tikz@sharp@parse
	\fi
	\pgfmathparse{max(\pgf@tempa,\pgf@tempb,0)}%
	\let\qrr@tikz@sharp@max\pgfmathresult
	\pgfmathsetlength\pgf@xa{.5*\pgflinewidth * tan(\qrr@tikz@sharp@max)}%
	\pgfarrowsleftextend{+\pgf@xa}%
	\pgfarrowsrightextend{+\pgf@xa}%
}{%
	\edef\pgf@marshal{\noexpand\pgfutil@in@{and}{\pgfgetarrowoptions{sharp >}}}%
	\pgf@marshal
	\ifpgfutil@in@
	\edef\pgf@tempa{\pgfgetarrowoptions{sharp >}}
	\expandafter\qrr@tikz@sharp@parse\pgf@tempa\@qrr@tikz@sharp@parse
	\else
	\qrr@tikz@sharp@parse\pgfgetarrowoptions{sharp >}and-\pgfgetarrowoptions{sharp >}\@qrr@tikz@sharp@parse
	\fi
	\pgfmathsetlength\pgf@ya{.5*\pgflinewidth * tan(max(\pgf@tempa,\pgf@tempb,0))}%
	\pgfmathsetlength\pgf@xa{-.5*\pgflinewidth * tan(\pgf@tempa)}%
	\pgfmathsetlength\pgf@xb{-.5*\pgflinewidth * tan(\pgf@tempb)}%
	\advance\pgf@xa\pgf@ya
	\advance\pgf@xb\pgf@ya
	\ifdim\pgf@xa>\pgf@xb
	\pgftransformyscale{-1}%
	\pgf@xc\pgf@xb
	\pgf@xb\pgf@xa
	\pgf@xa\pgf@xc
	\fi
	\pgfpathmoveto{\pgfqpoint{\qrr@tikz@sharp@z@}{.5\pgflinewidth}}%
	\pgfpathlineto{\pgfqpoint{\pgf@xa}{.5\pgflinewidth}}%
	\pgfpathlineto{\pgfqpoint{\pgf@ya}{+0pt}}%
	\pgfpathlineto{\pgfqpoint{\pgf@xb}{-.5\pgflinewidth}}%
	\pgfpathlineto{\pgfqpoint{\qrr@tikz@sharp@z@}{-.5\pgflinewidth}}%
	\pgfusepathqfill
}
\pgfarrowsdeclare{sharp <}{sharp <}{%
	\edef\pgf@marshal{\noexpand\pgfutil@in@{and}{\pgfgetarrowoptions{sharp <}}}%
	\pgf@marshal
	\ifpgfutil@in@
	\edef\pgf@tempa{\pgfgetarrowoptions{sharp <}}
	\expandafter\qrr@tikz@sharp@parse\pgf@tempa\@qrr@tikz@sharp@parse
	\else
	\expandafter\qrr@tikz@sharp@parse\pgfgetarrowoptions{sharp <}and-\pgfgetarrowoptions{sharp <}\@qrr@tikz@sharp@parse
	\fi
	\pgfmathparse{max(\pgf@tempa,\pgf@tempb,0)}%
	\let\qrr@tikz@sharp@max\pgfmathresult
	\pgfmathsetlength\pgf@xa{.5*\pgflinewidth * tan(\qrr@tikz@sharp@max)}%
	\pgfarrowsleftextend{+\pgf@xa}%
	\pgfarrowsrightextend{+\pgf@xa}%
}{%
	\edef\pgf@marshal{\noexpand\pgfutil@in@{and}{\pgfgetarrowoptions{sharp <}}}%
	\pgf@marshal
	\ifpgfutil@in@
	\edef\pgf@tempa{\pgfgetarrowoptions{sharp <}}
	\expandafter\qrr@tikz@sharp@parse\pgf@tempa\@qrr@tikz@sharp@parse
	\else
	\expandafter\qrr@tikz@sharp@parse\pgfgetarrowoptions{sharp <}and-\pgfgetarrowoptions{sharp <}\@qrr@tikz@sharp@parse
	\fi
	\pgfmathsetlength\pgf@ya{.5*\pgflinewidth * tan(max(\pgf@tempa,\pgf@tempb,0))}%
	\pgfmathsetlength\pgf@xa{-.5*\pgflinewidth * tan(\pgf@tempa)}%
	\pgfmathsetlength\pgf@xb{-.5*\pgflinewidth * tan(\pgf@tempb)}%
	\advance\pgf@xa\pgf@ya
	\advance\pgf@xb\pgf@ya
	\ifdim\pgf@xa>\pgf@xb
	\pgftransformyscale{-1}%
	\pgf@xc\pgf@xb
	\pgf@xb\pgf@xa
	\pgf@xa\pgf@xc
	\fi
	\pgfpathmoveto{\pgfqpoint{\qrr@tikz@sharp@z@}{.5\pgflinewidth}}%
	\pgfpathlineto{\pgfqpoint{\pgf@xa}{.5\pgflinewidth}}%
	\pgfpathlineto{\pgfqpoint{\pgf@ya}{+0pt}}%
	\pgfpathlineto{\pgfqpoint{\pgf@xb}{-.5\pgflinewidth}}%
	\pgfpathlineto{\pgfqpoint{\qrr@tikz@sharp@z@}{-.5\pgflinewidth}}%
	\pgfusepathqfill
}
\def\qrr@tikz@sharp@parse#1and#2\@qrr@tikz@sharp@parse{\def\pgf@tempa{#1}\def\pgf@tempb{#2}}

\makeatother

\newcommand\multiset[2]%
{\mathchoice{\left(\kern-0.4em{\binom{#1}{#2}}\kern-0.4em\right)}
	{\bigl(\kern-0.2em{\binom{#1}{#2}}\kern-0.2em\bigr)}
	{\bigl(\kern-0.2em{\binom{#1}{#2}}\kern-0.2em\bigr)}
	{\bigl(\kern-0.2em{\binom{#1}{#2}}\kern-0.2em\bigr)}}

\let\existstemp\exists \renewcommand*{\exists}{\mathop \existstemp}
\let\foralltemp\forall \renewcommand*{\forall}{\mathop \foralltemp}

\def\quotient#1#2{\raise1ex\hbox{$#1$}\Big/\lower1ex\hbox{$#2$}}

\newcommand{\<}{\langle}
\renewcommand{\>}{\rangle}

\newtheorem{theorem}{Theorem}[section]
\newtheorem{lemma}[theorem]{Lemma}
\newtheorem{proposition}[theorem]{Proposition}
\newtheorem{corollary}[theorem]{Corollary}
\newtheorem{conjecture}[theorem]{Conjecture}

\theoremstyle{definition}
\newtheorem{definition}[theorem]{Definition}

\newtheorem{example}[theorem]{Example}

\theoremstyle{remark}
\newtheorem{remark}[theorem]{Remark}

\title{Theta operators, refined Delta conjectures, and coinvariants}

\author{Michele D'Adderio}
\address{Universit\'e Libre de Bruxelles (ULB)\\D\'epartement de Math\'ematique\\ Boulevard du Triomphe, B-1050 Bruxelles\\ Belgium}\email{mdadderi@ulb.ac.be}
\author{Alessandro Iraci}
\address{Universit\`a di Pisa and Universit\'e Libre de Bruxelles (ULB)\\Dipartimento di Matematica\\ Largo Bruno Pontecorvo 5, 56127 Pisa\\ Italia}\email{iraci@student.dm.unipi.it}

\author{Anna Vanden Wyngaerd}
\address{Universit\'e Libre de Bruxelles (ULB)\\D\'epartement de Math\'ematique\\ Boulevard du Triomphe, B-1050 Bruxelles\\ Belgium}\email{anvdwyng@ulb.ac.be}

\begin{document}
	
\begin{abstract}
	We introduce the family of  Theta operators $\Theta_f$ indexed by symmetric functions $f$ that allow us to conjecture a compositional refinement of the Delta conjecture of Haglund, Remmel and Wilson \cite{Haglund-Remmel-Wilson-2015} for $\Delta_{e_{n-k-1}}'e_n$. We show that the $4$-variable Catalan theorem of Zabrocki \cite{Zabrocki-4Catalan-2016} is precisely the Schr\"{o}der case of our compositional Delta conjecture, and we show how to relate this conjecture to the Dyck path algebra introduced by Carlsson and Mellit in \cite{Carlsson-Mellit-ShuffleConj-2015}, extending one of their results. 
	
	Again using the Theta operators, we conjecture a touching refinement of the generalized Delta conjecture for $\Delta_{h_m}\Delta_{e_{n-k-1}}'e_n$, and prove the case $k=0$, which was also conjectured in \cite{Haglund-Remmel-Wilson-2015}, extending the shuffle theorem of Carlsson and Mellit to a \emph{generalized shuffle theorem} for $\Delta_{h_m}\nabla e_n$. Moreover we show how this implies the case $k=0$ of our generalized Delta square conjecture for $\frac{[n-k]_t}{[n]_t}\Delta_{h_m}\Delta_{e_{n-k}}\omega(p_n)$, extending the square theorem of Sergel \cite{Leven-2016} to a \emph{generalized square theorem} for $\Delta_{h_m}\nabla \omega(p_n)$.
	
	Still the Theta operators will provide a conjectural formula for the Frobenius characteristic of super-diagonal coinvariants with two sets of Grassmanian variables, extending the one of Zabrocki in \cite{Zabrocki_Delta_Module} for the case with one set of such variables. We propose a combinatorial interpretation of this last formula at $q=1$, leaving open the problem of finding a $\dinv$ statistic that gives the whole symmetric function.
\end{abstract}
	
\maketitle
\tableofcontents

\section{Introduction}
%\section{Introduction}

In the 90's Garsia and Haiman introduced the $\mathfrak{S}_n$-module of \emph{diagonal harmonics}, i.e.\ the coinvariants of the diagonal action of $\mathfrak{S}_n$ on polynomials in two sets of $n$ variables, and they conjectured that its Frobenius characteristic was given by $\nabla e_n$, where $\nabla$ is the \emph{nabla} operator on symmetric functions introduced in \cite{Bergeron-Garsia-Haiman-Tesler-Positivity-1999}. In 2002 Haiman proved this conjecture (see \cite{Haiman-Vanishing-2002}). Later the authors of \cite{HHLRU-2005} formulated the so called \emph{shuffle conjecture}, i.e.\ they predicted a combinatorial formula for $\nabla e_n$ in terms of labelled Dyck paths. Several years later in \cite{Haglund-Morse-Zabrocki-2012} Haglund, Morse and Zabrocki conjectured a \emph{compositional} refinement of the shuffle conjecture, which specified also the points where the Dyck paths touches the main diagonal. Recently Carlsson and Mellit in \cite{Carlsson-Mellit-ShuffleConj-2015} proved precisely this refinement, thanks to the introduction of what they called the \emph{Dyck path algebra}. See \cite{Willigenburg_History_Shuffle} for more on this story.  

\medskip

In \cite{Haglund-Remmel-Wilson-2015} Haglund, Remmel and Wilson formulated the \emph{Delta conjecture}, which can be stated as
\[\Delta_{e_{n-k-1}}'e_n=\sum_{P\in \LD(0,n)^{\ast k}} q^{\dinv(P)} t^{\area(P)}x^P, \]
where the sum is over labelled Dyck paths of size $n$ with positive labels and $k$ decorated rises. It turns out that for $k=0$ this formula reduces to the shuffle conjecture. Recently this conjecture attracted quite a bit of interest: see \cite[Section~2]{TheBible} for the state of the art on this problem.

Even more recently Zabrocki in \cite{Zabrocki_Delta_Module} conjectured that this formula gives the Frobenius characteristic of the submodule of the \emph{super-diagonal coinvariants} of degree $k$ in the Grassmannian variables (see Section~\ref{sec:superdiag} for the missing definitions). It turns out that the submodule of degree $0$ in the Grassmannian variables is precisely the module of diagonal harmonics.
So the whole framework of ``diagonal harmonics $+$ shuffle conjecture'' got generalized to this new setting.

\medskip

In this work we add a new piece to the puzzle, by extending the compositional shuffle conjecture to a \emph{compositional Delta conjecture}. In order to do so, we introduce a new family of \emph{Theta operators} $\Theta_f$ on symmetric functions, indexed by symmetric functions $f$. In fact with this article we want to make the case for the Theta operators.

Here are the highlights of the present paper:
\begin{itemize}
	\item We state a \emph{compositional Delta conjecture}, which will read as follows: for a composition $\alpha\vDash n-k$
	\[\Theta_{e_k}\nabla C_\alpha=\mathop{\sum_{P\in \LD(0,n)^{\ast k}}}_{\dcomp(P)=\alpha}q^{\dinv(P)}t^{\area(P)} x^P,\]
	where $\nabla C_\alpha$ are the symmetric functions appearing in the compositional shuffle conjecture, and $\dcomp(P)$ is the composition given by the distances between the consecutive points where the Dyck path touches the diagonal, ignoring the rows containing a decorated rise.
	
	\item We show that Zabrocki's $4$-variable Catalan theorem \cite{Zabrocki-4Catalan-2016} is actually the Schr\"{o}der case of our compositional Delta conjecture.
	
	\item We show that the combinatorial side satisfies a recursion that can be described by the Dyck path algebra in the same way as Carlsson and Mellit proved it for the compositional shuffle conjecture. In this case our theorem will read as
	\[d_-^{\ell(\alpha)} M_\alpha^{*k} =\mathop{\sum_{P\in \LD(0,n)^{\ast k}}}_{\dcomp(P)=\alpha}q^{\dinv(P)}t^{\area(P)} x^P,\]
	where $M_\alpha^{*k}$ is defined recursively, and it coincides with $N_\alpha$ in \cite{Carlsson-Mellit-ShuffleConj-2015} for $k=0$. This reduces our refinement of the Delta conjecture to an identity of operators.
	
	\item We state a \emph{touching generalized Delta conjecture}, which refines the generalized Delta conjecture in \cite{Haglund-Remmel-Wilson-2015} for $\Delta_{h_m}\Delta_{e_{n-k-1}}'e_n$ to 
	\begin{equation*}
	\Delta_{h_m}\Theta_{e_k}\nabla E_{n-k,r}=\mathop{\sum_{P\in \LD(m,n)^{\ast k}}}_{\touch(P)=r}q^{\dinv(P)}t^{\area(P)} x^P,
	\end{equation*}
	where $\nabla E_{n-k,r}$ are the symmetric functions already appearing in the shuffle conjecture and the sum of the righthand side is over labelled Dyck paths with nonnegative labels of size $m+n$, with $m$ zero labels, $k$ decorated rises and touching the diagonal $r$ times. Furthermore we state a \emph{touching generalized Delta square conjecture}, which refines our generalized Delta square conjecture \cite{DAdderio-Iraci-VandenWyngaerd-Delta-Square} for $[n-k]_t/[n]_t\Delta_{h_m}\Delta_{e_{n-k}}\omega(p_n)$ to
	\begin{equation*}
	\frac{[n]_q}{[r]_q} \Delta_{h_m} \Theta_{e_k}\nabla E_{n-k,r}=\mathop{\sum_{P\in \mathsf{LSQ}(m,n)^{\ast k}}}_{\touch(P)=r}q^{\dinv(P)}t^{\area(P)} x^P,
	\end{equation*} where the sum is over labelled square paths ending east with nonnegative labels, $k$ decorated rises and touching the diagonal $r$ times. 
	
	\item We prove the case $k=0$ of our touching generalized Delta conjecture, which was already conjectured in \cite[Conjecture~7.5]{Haglund-Remmel-Wilson-2015}. This extends the shuffle theorem of Carlsson and Mellit \cite{Carlsson-Mellit-ShuffleConj-2015} to a \emph{generalized shuffle theorem} for $\Delta_{h_m}\nabla e_n$.
	
	\item We prove the case $k=0$ of our touching generalized Delta square conjecture. This extends the square theorem of Sergel \cite{Leven-2016} to a \emph{generalized square theorem} for $\Delta_{h_m}\nabla \omega(p_n)$.
	
	\item We extend Zabrocki's conjecture \cite{Zabrocki_Delta_Module} to the module $M_n^{(2)}$ of super-diagonal coinvariants of the diagonal action of $\mathfrak{S}_n$ on polynomials in two sets of $n$ commutative variables and two sets of $n$ Grassmanian variables: 
	\begin{equation*}
	\mathcal{F}_{q,t,\underline{z}}(M_n^{(2)}) = \mathop{\sum_{i,j\geq 0}}_{1\leq i+j<n} z_1^iz_2^j\Theta_{e_i}\Theta_{e_j}\nabla e_{n-(i+j)}.
	\end{equation*}
	
	\item We conjecture a combinatorial interpretation for the formula in the previous item at $q=1$:
	\begin{equation*}
	\left.\Theta_{e_r}\Theta_{e_k} \nabla e_{n-r-k}\right|_{q=1}  = \sum_{P\in \LD(0,n)^{\ast k, \circ r}}t^{\area(P)}x^P
	\end{equation*}
	where $\LD(0,n)^{\ast k, \circ r}$ is the set of labelled Dyck paths with positive labels of size $n$ with $k$ decorated rises and $r$ decorated contractible valleys. We leave open the outstanding problem of finding a $\dinv$ statistic that gives the whole symmetric function $\Theta_{e_r}\Theta_{e_k} \nabla e_{n-r-k}$.
\end{itemize}

The rest of this paper is organized in the following way. In Section~2 we introduce the basic ingredients of symmetric functions, that will be needed in Section~3 to introduce the Theta operators and to state the basic theorems that led us to their definition. In Section~4 we recall the combinatorial definitions needed in Section~5 to state our refined Delta conjectures. In Section~6 we prove our results about the compositional Delta conjecture, while in Section~7 we prove our results about the touching generalized Delta conjectures. In Section~8 we state our conjectures about the Frobenius characteristics of super-diagonal coinvariants, while in Section~9 we state our combinatorial interpretation of those at $q=1$. In Section~10 we state more conjectures about the Theta operators. Finally in Section~11 we give the details of the technical proofs of symmetric function theory that we left out in the previous sections.

\section*{Acknowledgements}

We are happy to thank Mike Zabrocki for providing us useful references, programs and information to check our conjecture on the Frobenius characteristic of super-diagonal coinvariants.

\section{Symmetric functions: basics}
%\section{Symmetric functions: basics}

In this section we limit ourselves to introduce the necessary notation to state our main theorems and conjectures. We refer to Section~\ref{sec:SF_tools} for more on symmetric functions.

The main references that we will use for symmetric functions
are \cite{Macdonald-Book-1995}, \cite{Stanley-Book-1999} and \cite{Haglund-Book-2008}. 

\medskip

The standard bases of the symmetric functions that will appear in our
calculations are the complete $\{h_{\lambda}\}_{\lambda}$, elementary $\{e_{\lambda}\}_{\lambda}$, power $\{p_{\lambda}\}_{\lambda}$ and Schur $\{s_{\lambda}\}_{\lambda}$ bases.

\medskip

\emph{We will use the usual convention that $e_0=h_0=1$ and $e_k=h_k=0$ for $k<0$.}

\medskip

The ring $\Lambda$ of symmetric functions can be thought of as the polynomial ring in the power
sum generators $p_1, p_2, p_3,\dots$. This ring has a grading $\Lambda=\bigoplus_{n\geq 0}\Lambda^{(n)}$ given by assigning degree $i$ to $p_i$ for all $i\geq 1$. As we are working with Macdonald symmetric functions
involving two parameters $q$ and $t$, we will consider this polynomial ring over the field $\mathbb{Q}(q,t)$.
We will make extensive use of the \emph{plethystic notation}.

With this notation we will be able to add and subtract alphabets, which will be represented as sums of monomials $X = x_1 + x_2 + x_3+\cdots $. Then, given a symmetric function $f$, and thinking of it as an element of $\Lambda$, we denote by $f[X]$ the expression $f$ with $p_k$ replaced by $x_{1}^{k}+x_{2}^{k}+x_{3}^{k}+\cdots$, for all $k$. More generally, given any expression $Q(z_1,z_2,\dots)$, we define the plethystic substitution $f[Q(z_1,z_2,\dots)]$ to be $f$ with $p_k$ replaced by $Q(z_1^k,z_2^k,\dots)$.

We denote by $\<\, , \>$ the \emph{Hall scalar product} on symmetric functions, which can be defined by saying that the Schur functions form an orthonormal basis. We denote by $\omega$ the fundamental algebraic involution which sends $e_k$ to $h_k$, $s_{\lambda}$ to $s_{\lambda'}$ and $p_k$ to $(-1)^{k-1}p_k$.

With the symbol ``$\perp$'' we denote the operation of taking the adjoint of an operator with respect to the Hall scalar product, i.e.
\begin{equation}
\langle f^\perp g,h\rangle=\langle g,fh\rangle\quad \text{ for all }f,g,h\in \Lambda.
\end{equation}

For a partition $\mu\vdash n$, we denote by
\begin{equation}
\widetilde{H}_{\mu} \coloneqq \widetilde{H}_{\mu}[X]=\widetilde{H}_{\mu}[X;q,t]=\sum_{\lambda\vdash n}\widetilde{K}_{\lambda \mu}(q,t)s_{\lambda}
\end{equation}
the \emph{(modified) Macdonald polynomials}, where 
\begin{equation}
\widetilde{K}_{\lambda \mu} \coloneqq \widetilde{K}_{\lambda \mu}(q,t)=K_{\lambda \mu}(q,1/t)t^{n(\mu)}\quad \text{ with }\quad n(\mu)=\sum_{i\geq 1}\mu_i(i-1)
\end{equation}
are the \emph{(modified) Kostka coefficients} (see \cite[Chapter~2]{Haglund-Book-2008} for more details). 

The set $\{\widetilde{H}_{\mu}[X;q,t]\}_{\mu}$ is a basis of the ring of symmetric functions $\Lambda$ with coefficients in $\mathbb{Q}(q,t)$. This is a modification of the basis introduced by Macdonald \cite{Macdonald-Book-1995}, and they are the Frobenius characteristic of the so called Garsia-Haiman modules (see \cite{Garsia-Haiman-PNAS-1993}).

If we identify the partition $\mu$ with its Ferrers diagram, i.e. with the collection of cells $\{(i,j)\mid 1\leq i\leq \mu_j, 1\leq j\leq \ell(\mu)\}$, then for each cell $c\in \mu$ we refer to the \emph{arm}, \emph{leg}, \emph{co-arm} and \emph{co-leg} (denoted respectively as $a_\mu(c), l_\mu(c), a_\mu(c)', l_\mu(c)'$) as the number of cells in $\mu$ that are strictly to the right, above, to the left and below $c$ in $\mu$, respectively (see Figure~\ref{fig:notation}).

\begin{figure}[h]
	\centering
	\begin{tikzpicture}[scale=.4]
	\draw[gray,opacity=.4](0,0) grid (15,10);
	\fill[white] (1,10)|-(3,9)|- (5,7)|-(9,5)|-(13,2)--(15.2,2)|-(1,10.2);
	\draw[gray]  (1,10)|-(3,9)|- (5,7)|-(9,5)|-(13,2)--(15,2)--(15,0)-|(0,10)--(1,10);
	\fill[blue, opacity=.2] (0,3) rectangle (9,4) (4,0) rectangle (5,7); 
	\fill[blue, opacity=.5] (4,3) rectangle (5,4);
	\draw (7,4.5) node {\tiny{Arm}} (3.25,5.5) node {\tiny{Leg}} (6.25, 1.5) node {\tiny{Co-leg}} (2,2.5) node {\tiny{Co-arm}} ;
	\end{tikzpicture}
	\caption{}
	\label{fig:notation}
\end{figure}

We set
\begin{equation}
M \coloneqq (1-q)(1-t),
\end{equation}
and we define for every partition $\mu$
\begin{align}
B_{\mu} &  \coloneqq B_{\mu}(q,t)=\sum_{c\in \mu}q^{a_{\mu}'(c)}t^{l_{\mu}'(c)}\\
D_{\mu} &  \coloneqq MB_{\mu}(q,t)-1\\
T_{\mu} &  \coloneqq T_{\mu}(q,t)=\prod_{c\in \mu}q^{a_{\mu}'(c)}t^{l_{\mu}'(c)}\\
\Pi_{\mu} &  \coloneqq \Pi_{\mu}(q,t)=\prod_{c\in \mu/(1)}(1-q^{a_{\mu}'(c)}t^{l_{\mu}'(c)})\\
w_{\mu} &  \coloneqq w_{\mu}(q,t)=\prod_{c\in \mu}(q^{a_{\mu}(c)}-t^{l_{\mu}(c)+1})(t^{l_{\mu}(c)}-q^{a_{\mu}(c)+1}).
\end{align}

Notice that
\begin{equation} \label{eq:Bmu_Tmu}
B_{\mu}=e_1[B_{\mu}]\quad \text{ and } \quad T_{\mu}=e_{|\mu|}[B_{\mu}].
\end{equation}

For every symmetric function $f[X]$ we set
\begin{equation}
f^*=f^*[X] \coloneqq f\left[\frac{X}{M}\right].
\end{equation}

The following linear operators were introduced in \cites{Bergeron-Garsia-ScienceFiction-1999,Bergeron-Garsia-Haiman-Tesler-Positivity-1999}, and they are at the basis of the conjectures relating symmetric function coefficients and $q,t$-combinatorics in this area. 

We define the \emph{nabla} operator on $\Lambda$ by
\begin{equation}
\nabla  \widetilde{H}_{\mu}=T_{\mu} \widetilde{H}_{\mu}\quad \text{ for all }\mu,
\end{equation}
and we define the \emph{Delta} operators $\Delta_f$ and $\Delta_f'$ on $\Lambda$ by
\begin{equation}
\Delta_f \widetilde{H}_{\mu}=f[B_{\mu}(q,t)]\widetilde{H}_{\mu}\quad \text{ and } \quad 
\Delta_f' \widetilde{H}_{\mu}=f[B_{\mu}(q,t)-1]\widetilde{H}_{\mu},\quad \text{ for all }\mu.
\end{equation}

Observe that on the vector space of symmetric functions homogeneous of degree $n$, denoted by $\Lambda^{(n)}$, the operator $\nabla$ equals $\Delta_{e_n}$. Moreover, for every $1\leq k\leq n$,
\begin{equation} \label{eq:deltaprime}
\Delta_{e_k}=\Delta_{e_k}'+\Delta_{e_{k-1}}'\quad \text{ on }\Lambda^{(n)},
\end{equation}
and for any $k>n$, $\Delta_{e_k}=\Delta_{e_{k-1}}'=0$ on $\Lambda^{(n)}$, so that $\Delta_{e_n}=\Delta_{e_{n-1}}'$ on $\Lambda^{(n)}$.

Recall the standard notation for $q$-analogues: $n\in \mathbb{N}$
\begin{equation}
[0]_q \coloneqq 0,\quad \text{ and }\quad [n]_q \coloneqq \frac{1-q^n}{1-q}=1+q+q^2+\cdots+q^{n-1} \quad \text{ for } n\geq 1,
\end{equation}
\begin{equation}
[0]_q! \coloneqq 1\quad \text{ and }\quad [n]_q! \coloneqq [n]_q[n-1]_q\cdots [2]_q[1]_q\quad \text{ for }n\geq 1,
\end{equation}
and
\begin{equation}
\qbinom{n}{k}_q \coloneqq \frac{[n]_q!}{[k]_q![n-k]_q!} \quad \text{ for } n \geq k \geq 0, \quad \text{ and }\quad  
\qbinom{n}{k}_q \coloneqq 0 \quad \text{ for } n < k,
\end{equation}
and also the standard notation for the $q$-\emph{rising factorial}
\begin{equation}
(a;q)_n \coloneqq (1-a)(1-qa)(1-q^2a)\cdots (1-q^{n-1}a).
\end{equation}

%\subsection{$C_\alpha$ and $E_{n,k}$ symmetric functions}

\medskip

In \cite{Haglund-Morse-Zabrocki-2012} the following operators were introduced: for any $m\geq 0$ and any $F[X]\in \Lambda$
\begin{equation}
\mathbb{C}_mF[X] \coloneqq (-1/q)^{m-1}\sum_{r\geq 0}q^{-r}h_{m+r}[X]h_{r}[X(1-q)]^\perp F[X],
\end{equation}
and for any composition $\alpha=(\alpha_1,\alpha_2,\dots,\alpha_l)$ of $n$, denoted $\alpha\vDash n$, we set
\begin{equation}
C_\alpha=C_\alpha[X;q] \coloneqq \mathbb{C}_{\alpha_1}\mathbb{C}_{\alpha_2}\cdots \mathbb{C}_{\alpha_l}(1).
\end{equation}

The symmetric functions $E_{n,k}$ were introduced in \cite{Garsia-Haglund-qtCatalan-2002} by means of the following expansion:
\begin{equation} \label{eq:def_Enk}
e_n\left[X\frac{1-z}{1-q}\right]=\sum_{k=1}^n \frac{(z;q)_k}{(q;q)_k}E_{n,k}.
\end{equation}

Notice that setting $z=q^j$ in \eqref{eq:def_Enk} we get
\begin{equation} \label{eq:en_q_sum_Enk}
e_n\left[X[j]_q\right]=e_n\left[X\frac{1-q^j}{1-q}\right]=\sum_{k=1}^n \frac{(q^j;q)_k}{(q;q)_k}E_{n,k} =\sum_{k=1}^n \begin{bmatrix}
k+j-1\\
k
\end{bmatrix}_qE_{n,k} .
\end{equation}
%In particular, for $j=1$, we get
In particular, for $z=q$ we get
\begin{equation} \label{eq:en_sum_Enk}
e_n=E_{n,1}+E_{n,2}+\cdots +E_{n,n}.
\end{equation}

The following identity is proved in \cite{Haglund-Morse-Zabrocki-2012}*{Section~5}:
\begin{equation} \label{eq:Enk=sumCalpha}
E_{n,r}=\mathop{\sum_{\alpha\vDash n}}_{\ell(\alpha)=r}C_\alpha\quad \text{ for all }r=1,2,\dots, n,
\end{equation}
where $\ell(\alpha)$ denotes the length of the composition $\alpha$.

Together with \eqref{eq:en_sum_Enk} it gives immediately
\begin{equation}
e_n=\sum_{\alpha\vDash n} C_\alpha.
\end{equation}

The following identity is proved in \cite{Can-Loehr-2006}*{Theorem~4}:
\begin{equation} \label{eq:pn_Enk}
\omega(p_n)=\sum_{k=1}^n\frac{[n]_q}{[k]_q}E_{n,k}.
\end{equation}

\section{The Theta operators} \label{sec:Delkm_ops}
%\section{The operators $\Del{k}{m}$} \label{sec:Delkm_ops}

In this section we introduce the family of operators $\Theta_f$ and we state a few results about them: \emph{we will give the missing proofs in Section~\ref{sec:proofs_sec_Delkm}}. See also Section~\ref{sec:conj_Delk0} for some conjectures about them.

\medskip

Recall the definition of the invertible linear operator $\mathbf{\Pi}$ on $\Lambda=\oplus_{n\geq 0}\Lambda^{(n)}$ defined by $\mathbf{\Pi} 1=1$, and for any non-empty partition $\mu$
\begin{equation}
\mathbf{\Pi} \widetilde{H}_\mu \coloneqq \Pi_\mu \widetilde{H}_\mu .
\end{equation}

For any symmetric function $f\in \Lambda$ we introduce the following \emph{Theta operators} on $\Lambda$: for every $F[X]\in \Lambda$ we set
\begin{equation} \label{eq:def_Deltaf}
\Theta_fF[X] \coloneqq \mathbf{\Pi}f^*\mathbf{\Pi}^{-1}F[X].
\end{equation}
It is clear that $\Theta_f$ is linear, and moreover, if $f$ is homogenous of degree $k$, then so is $\Theta_f$, i.e. \[\Theta_f\Lambda^{(n)}\subseteq \Lambda^{(n+k)} \qquad \text{ for }f\in \Lambda^{(k)}. \]

Since in the present article we will use mostly a special case of these operators, we introduced the following shorter notation: for $k\geq 0$, we set
%
%For $k\geq 0$ we introduce the following linear operators on $\Lambda$: for every $F[X]\in \Lambda$ we set
\begin{equation} \label{eq:def_Deltakm}
\Theta_k \coloneqq \Theta_{e_k}.
\end{equation}
Notice that $\Theta_0$ is the identity operator on $\Lambda$.

The following theorems, which we prove in Section~\ref{sec:proofs_sec_Delkm}, led to the definition of the Theta operators.
\begin{theorem} \label{thm:DeltakmGD}
	For $n\geq 1$ and $k\geq 0$,
	\begin{equation} \label{eq:thm:DeltakmGD}
	\Theta_k\nabla e_{n-k}=\Delta_{e_{n-k-1}}'e_n.
	\end{equation}
\end{theorem}
\begin{corollary} \label{cor:sumTouchingGD}
	For $n\geq 1$ and $k\geq 0$,
	\begin{equation}
	\sum_{r=1}^{n-k}\Theta_k\nabla E_{n-k,r} = \Delta_{e_{n-k-1}}'e_n. 
	\end{equation}
\end{corollary}
\begin{proof}
	Using \eqref{eq:en_sum_Enk}, we have
	\begin{align*}
	\sum_{r=1}^{n-k} \Theta_k\nabla E_{n-k,r}  & = \Theta_k\nabla e_{n-k}\\
	\text{(using \eqref{eq:thm:DeltakmGD})}& = \Delta_{e_{n-k-1}}'e_n .
	\end{align*}
\end{proof}

\begin{theorem} \label{thm:DeltakmDSq}
	For $n\geq 1$ and $k\geq 0$,
	\begin{equation} \label{eq:thm:DeltakmDSq}
	\Theta_k\nabla \frac{[n]_q}{[n-k]_q}\omega(p_{n-k})=\frac{[n-k]_t}{[n]_t}\Delta_{e_{n-k}}\omega(p_n).
	\end{equation}
\end{theorem}
\begin{corollary} \label{cor:DeltaSqSum}
	For $n\geq 1$ and $k\geq 0$,
	\begin{equation}
	\sum_{r=1}^{n-k}\frac{[n]_q}{[r]_q} \Theta_k\nabla E_{n-k,r} =\frac{[n-k]_t}{[n]_t} \Delta_{e_{n-k}}\omega(p_{n}). 
	\end{equation}
\end{corollary}
\begin{proof}
	We have
	\begin{align*}
	\sum_{r=1}^{n-k}\frac{[n]_q}{[r]_q} \Theta_k\nabla E_{n-k,r}  & =\frac{[n]_q}{[n-k]_q}\sum_{r=1}^{n-k} \Theta_k\nabla \frac{[n-k]_q}{[r]_q}E_{n-k,r}\\
	\text{(using \eqref{eq:pn_Enk})}& = \Theta_k\nabla\frac{[n]_q}{[n-k]_q}\omega(p_{n-k}) \\
	\text{(using \eqref{eq:thm:DeltakmDSq})}& = \frac{[n-k]_t}{[n]_t} \Delta_{e_{n-k}}\omega(p_n).
	\end{align*}
\end{proof}

\section{Combinatorial definitions}
%\section{Combinatorial definitions}

\begin{definition}
	A \emph{square path ending east} of size $n$ is a lattice paths going from $(0,0)$ to $(n,n)$ consisting of east or north unit steps, always ending with an east step. The set of such paths is denoted by $\SQ(n)$. We call \emph{base diagonal} of a square path the diagonal $y=x+l$ with the smallest value of $l$ that is touched by the path (so that $l\leq 0$). The \emph{shift} of the square path is the non-negative value $-l$. We refer to the line $x=y$ as the \emph{main diagonal}. A \emph{Dyck path} is a square path whose shift is $0$, i.e. its base diagonal is the main diagonal. The set of Dyck paths is denoted by $\D(n)$. Of course $\D(n)\subseteq \SQ(n)$.  
\end{definition}

For example, the path in Figure~\ref{fig: labelled square} has shift $3$. 

\begin{figure}[ht]
	\centering
	\begin{tikzpicture}[scale = 0.6]
	\draw[step=1.0, gray!60, thin] (0,0) grid (8,8);
	\draw[gray!60, thin] (3,0) -- (11,8);
	\draw[gray!60, thin] (8,6) -- (9,6) -- (9,8) (8,7) -- (10,7) -- (10,8) (8,8) -- (11,8);
	
	\draw[blue!60, line width=1.6pt] (0,0) -- (0,1) -- (1,1) -- (2,1) -- (3,1) -- (4,1) -- (4,2) -- (5,2) -- (5,3) -- (5,4) -- (6,4) -- (6,5) -- (6,6) -- (6,7) -- (7,7) -- (7,8) -- (8,8);
	
	\node at (5.5,5.5) {$\ast$};
	
	\node at (0.5,0.5) {$2$};
	\draw (0.5,0.5) circle (.4cm); 
	\node at (4.5,1.5) {$0$};
	\draw (4.5,1.5) circle (.4cm); 
	\node at (5.5,2.5) {$2$};
	\draw (5.5,2.5) circle (.4cm); 
	\node at (5.5,3.5) {$4$};
	\draw (5.5,3.5) circle (.4cm); 
	\node at (6.5,4.5) {$0$};
	\draw (6.5,4.5) circle (.4cm); 
	\node at (6.5,5.5) {$1$};
	\draw (6.5,5.5) circle (.4cm); 
	\node at (6.5,6.5) {$3$};
	\draw (6.5,6.5) circle (.4cm); 
	\node at (7.5,7.5) {$1$};
	\draw (7.5,7.5) circle (.4cm);
	\end{tikzpicture}
	\caption{Example of an element in $\LSQ(2,6)^{\ast 1 }$ with reading word $241231$.}
	\label{fig: labelled square}
\end{figure}
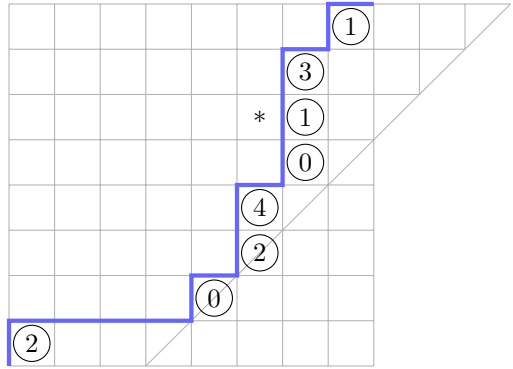

\begin{definition}
	A \emph{labelling} or \emph{word} of a square path $\pi$ of size $m+n$ ending east  is an element $w\in \mathbb N^{m+n}$ such that when we label the $i$-th vertical step of $\pi$ with $w_i$
	\begin{itemize}
		\item the labels appearing in each column of $\pi$ are strictly increasing from bottom to top;
		\item there is at least one nonzero label labelling a vertical step starting from the base diagonal of $\pi$;
		\item if $\pi$ starts with a vertical step, then this first step has a nonzero label.
		%\item exactly $m$ of the $w_i$ are equal to zero. 
	\end{itemize}
	The set of such labellings with $m$ labels equal to $0$ is denoted by $\W(\pi, m)$, and we set $\W(\pi)\coloneqq \W(\pi,0)$. 
	
	A \emph{partially labelled square path ending east} is an element $P=(\pi, w)$ of 
	\begin{align*}
	&\LSQ(m,n)\coloneqq \{(\pi, w) \mid \pi \in \SQ(m+n), w \in \W(\pi,m) \}.
	\end{align*}
	We also define the subset of \emph{labelled Dyck paths} as
	\begin{align*}
	&\LD(m,n)\coloneqq \{(\pi,w)\in \LSQ(m,n) \mid \pi \in \D(m+n)\}\subseteq \LSQ(m,n).	
	\end{align*}
\end{definition}

\begin{definition}
	Let $\pi$ be a square path ending east of size $m+n$. We define its \emph{area word} to be the sequence of integers $a(\pi) = (a_1(\pi),a_2(\pi), \cdots, a_{m+n}(\pi))$ such that the $i$-th vertical step of the path starts from the diagonal $y=x+a_i(\pi)$. For example the path in Figure~\ref{fig: labelled square} has area word $(0, \, -\!3, \, -\!3, \, -\!2, \, -\!2, \, -\!1, \, 0, \, 0)$.
\end{definition}

\begin{definition}\label{def: monomial path}
	We define for each $P\in \LSQ(m,n)$ a monomial in the variables $x_1,x_2,\dots$: we set \[ x^P \coloneqq \prod_{i=1}^{m+n} x_{l_i(P)} \] where $l_i(P)$ is the label of the $i$-th vertical step of $P$ (the first being at the bottom), and where we \underline{conventionally set $x_0 = 1$}. The fact that $x_0$ does not appear in the monomial explains the word \emph{partially}.
\end{definition}

\begin{definition}\label{def: rise}
	The \emph{rises} of a square path ending east $\pi$ are the indices \[ r(\pi) \coloneqq \{2\leq i \leq m+n\mid a_{i}(\pi)>a_{i-1}(\pi)\},\] or the vertical steps that are directly preceded by another vertical step. 
	
	A \emph{decorated square path} (respectively Dyck path) is a pair $P=(\pi, dr)$ where $\pi$ is a square path (respectively Dyck path) and $dr\subseteq r(\pi)$. We set 
	\begin{align*}
	&\SQ(n)^{\ast k}\coloneqq \{(\pi, dr)\mid \pi\in \SQ(n), dr\subseteq r(\pi), \vert dr\vert= k \}\\
	&\D(n)^{\ast k} \coloneqq \{(\pi,dr)\in \SQ(n)^{\ast k}\mid \pi \in \D(n) \}
	\end{align*}
	A \emph{partially labelled decorated square path} (respectively Dyck path) is a triple $(\pi, dr, w)$ where $\pi$ is a square path (respectively Dyck path), $dr\subseteq r(\pi)$ and $w$ is a labelling of $\pi$. We set 
	\begin{align*}
	&\LSQ(m,n)^{\ast k}\coloneqq \{(\pi, dr, w)\mid (\pi,w)\in \LSQ(m,n), dr\subseteq r(\pi), \vert dr\vert= k \}\\
	&\LD(m,n)^{\ast k} \coloneqq \{(\pi,dr,w)\in \LSQ(m,n)^{\ast k}\mid \pi \in \D(m+n) \}\subseteq \LSQ(m,n)^{\ast k}.
	\end{align*} We will also use the following natural identifications
	\begin{align*}
	\LSQ(n,m)^{\ast 0}= \LSQ(n,m) && \LD(m,n)^{\ast 0}= \LD(m,n)
	% \\ \LSQ(n)^{\ast k}\coloneqq \LSQ(0,n)^{\ast k} && \LD(n)^{\ast k}\coloneqq \LD(0,n)^{\ast k} .
	\end{align*}
\end{definition}

\begin{definition}
	Given a partially labelled square path, we call \emph{zero valleys} its vertical steps with label $0$ (which are necessarily preceded by a horizontal step, hence the name valleys).
\end{definition}

\begin{definition}\label{def: reading word}
	Given a partially labelled square path $P$ of shift $s$, we define its \emph{reading word} $\sigma(P)$ as the sequence of \emph{nonzero} labels, read starting from the base diagonal $y=x-s$ going bottom left to top right, then moving to the next diagonal, $y=x-s+1$ again going bottom left to top right, and so on.
	
	If the reading word of $P$ is $r_1\dots r_n$ then the  \emph{reverse reading word} of $P$ is $r_n\dots r_1$. 
\end{definition}
See Figure~\ref{fig: labelled square} and Figure~\ref{fig:pldExample1} for an example.

\begin{figure*}[!ht]
	\centering
	\begin{tikzpicture}[scale = .6]
	
	\draw[step=1.0, gray!60, thin] (0,0) grid (8,8);
	
	\draw[gray!60, thin] (0,0) -- (8,8);
	
	\draw[blue!60, line width=2pt] (0,0) -- (0,1) -- (0,2) -- (1,2) -- (2,2) -- (2,3) -- (2,4) -- (2,5) -- (3,5) -- (4,5) -- (4,6) -- (4,7) -- (4,8) -- (5,8) -- (6,8) -- (7,8) -- (8,8);
	
	\draw (0.5,0.5) circle (0.4 cm) node {$1$};
	\draw (0.5,1.5) circle (0.4 cm) node {$3$};
	\draw (2.5,2.5) circle (0.4 cm) node {$0$};
	\draw (2.5,3.5) circle (0.4 cm) node {$4$};
	\draw (2.5,4.5) circle (0.4 cm) node {$6$};
	\draw (4.5,5.5) circle (0.4 cm) node {$0$};
	\draw (4.5,6.5) circle (0.4 cm) node {$2$};
	\draw (4.5,7.5) circle (0.4 cm) node {$6$};
	
	\node at (1.5,3.5) {$\ast$};
	\node at (3.5,6.5) {$\ast$};
	
	\end{tikzpicture}
	\caption{Example of an element in $\LD(2,6)^{\ast 2}$ with reading word $134626$.}
	\label{fig:pldExample1}
\end{figure*}
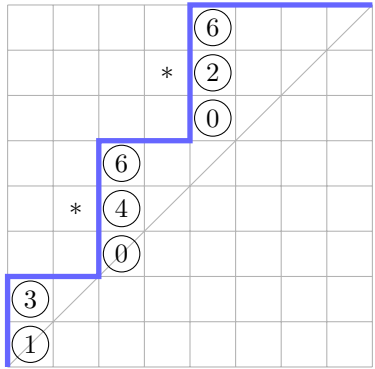

We define two statistics on this set that reduce to the same statistics as defined in \cite{Loehr-Warrington-square-2007} when $m=k=0$. 

\begin{definition}
	\label{def:sqarea}
	Let $P=(\pi, dr) \in \SQ(m,n)^{\ast k}$ and $s$ be its shift. Define 
	\[
	\area(P) \coloneqq \sum_{i\not \in dr} (a_i(\pi) + s).
	\] More visually, the area is the number of whole squares between the path and the base diagonal and not contained in rows containing a decorated rise. 
	
	If $P=(\pi, dr,w)\in \LSQ(m,n)^{\ast k}$ then we set $\area(P)=\area((\pi, dr))$. In other words, the area of a path does not depend on its labelling. 
\end{definition}

For example, the path in Figure~\ref{fig: labelled square} has area $11$. 

\begin{definition} \label{def: dinv SQ}
	Let $P=(\pi, dr, w) \in \LSQ(m,n)^{\ast k}$. For $1 \leq i < j \leq m+n$, we say that the pair $(i,j)$ is an \emph{inversion} if
	\begin{itemize}
		\item either $a_i(\pi) = a_j(\pi)$ and $w_i < w_j$ (\emph{primary inversion}),
		\item or $a_i(\pi) = a_j(\pi) + 1$ and $w_i > w_j$ (\emph{secondary inversion}),
	\end{itemize}
	where $w_i$ denotes the $i$-th letter of $w$, i.e. the label of the vertical step in the $i$-th row.
	
	Then we define 
	\begin{align*}
	\dinv(P) & \coloneqq \# \{ 0\leq i < j \leq m+n \mid (i,j) \; \text{is an inversion}\} \\
	& \quad + \#\{0\leq i\leq m+n \mid a_i(\pi)<0\text{ and } w_i \neq 0 \}.
	\end{align*} 
	This second term is referred to as \emph{bonus dinv}. 
\end{definition}

For example, the path in Figure~\ref{fig: labelled square} has dinv $6$: $2$ primary inversions, i.e. $(1,7)$ and $(2,3)$, $1$ secondary inversion, i.e. $(1,6)$, and $3$ bonus dinv, coming from the rows $3$, $4$ and $6$. Notice that Dyck paths coincide with the square paths with no bonus dinv.

\begin{definition}
	Given $P\in \LD(m,n)^{\ast k}$ we define its \emph{diagonal composition} $\dcomp(P)$ to be the composition of $n-k$ whose $i$-th part is the number of rows of $P$ without a $0$ label or a decoration $\ast$ that lie between the $i$-th and the $(i+1)$-th vertical step of $P$ on the base diagonal not labelled by a $0$ (or from the $i$-th step onwards if it is the last such step). See Figure~\ref{fig:composition} for an example. If $P\in D(n)^{\ast k}$, its diagonal composition is defined identically, without the conditions concerning the labels (since there are none).
\end{definition}

\begin{definition}
	Given $P\in \LSQ(m,n)^{\ast k}$ we define a \emph{touching point} of $P$ to be a starting point of a vertical step of $P$ that lies on the base diagonal, and whose label is not zero. The \emph{touching number} $\touch(P)$ of $P$ is defined as the number of touching points of $P$ or equivalently as the length of its diagonal composition. See Figure~\ref{fig:composition} for an example. For $P\in \D(n)^{\ast}$ these definitions are the same, without the condition concerning the labels (since there are none). 
\end{definition}

\begin{figure}[!ht]
	\centering
	\begin{tikzpicture}[scale = 0.6]
	\draw[gray!60, thin] (0,0) grid (12,12);
	\draw[gray!60, thin] (0,0) -- (12,12);
	
	\draw[blue!60, line width=1.6pt] (0,0) -- (0,1) -- (0,2) -- (1,2) -- (2,2) -- (2,3) -- (3,3) -- (3,4) -- (4,4) -- (4,5) -- (4,6) -- (4,7) -- (5,7) -- (6,7) -- (6,8) -- (7,8) -- (8,8) -- (8,9) -- (9,9) -- (9,10) -- (9,11) -- (10,11) -- (10,12) -- (11,12) -- (12,12);
	
	\draw
	(0.5,0.5) circle(0.4 cm) node {$2$}
	(0.5,1.5) circle(0.4 cm) node {$6$}
	(2.5,2.5) circle(0.4 cm) node {$0$}
	(3.5,3.5) circle(0.4 cm) node {$7$}
	(4.5,4.5) circle(0.4 cm) node {$0$}
	(4.5,5.5) circle(0.4 cm) node {$1$}
	(4.5,6.5) circle(0.4 cm) node {$4$}
	(6.5,7.5) circle(0.4 cm) node {$0$}
	(8.5,8.5) circle(0.4 cm) node {$8$}
	(9.5,9.5) circle(0.4 cm) node {$3$}
	(9.5,10.5) circle(0.4 cm) node {$5$}
	(10.5,11.5) circle(0.4 cm) node {$9$};
	
	\node at (-0.5,1.5) {$\ast$};
	\node at (3.5,5.5) {$\ast$};
	\end{tikzpicture}
	\caption{A partially labelled Dyck path with diagonal composition $\alpha = (1,2,1,3)$ and touching number $\ell(\alpha) = 4$.}
	\label{fig:composition}
\end{figure}
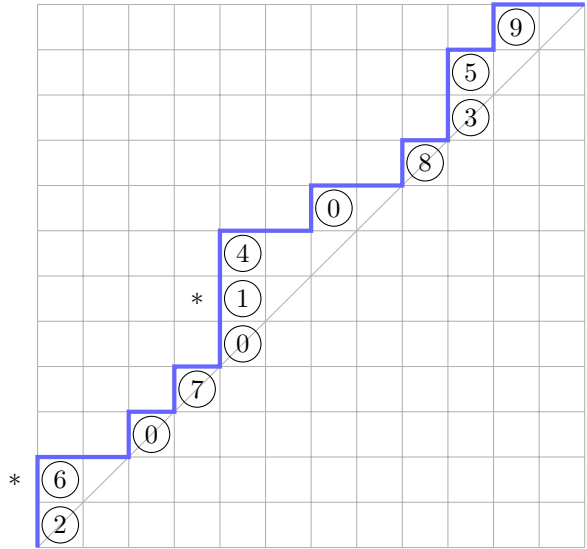

\section{Statements of refined Delta conjectures}
%\section{Touching and compositional Delta conjectures}

In this section we state our refined conjectures. 

\medskip

The following conjecture is due to Haglund, Remmel and Wilson \cite{Haglund-Remmel-Wilson-2015}.

\begin{conjecture}[Generalized Delta] \label{conj:GenDelta}
	Given $n,k,m\in \mathbb{N}$ with $n>k\geq 0$,
	\begin{equation}
	\Delta_{h_m} \Delta_{e_{n-k-1}}'e_n= \sum_{P\in \mathsf{LSQ}(m,n)^{\ast k}} q^{\dinv(P)}t^{\area(P)} x^P.
	\end{equation}	
\end{conjecture}

We state our ``touching'' refinement of this conjecture.
\begin{conjecture}[Touching generalized Delta] \label{conj:touchingGenDelta}
	Given $n,k,m,r\in \mathbb{N}$, $n>k\geq 0$ and $n-k\geq r\geq 1$,
	\begin{equation}
	\Delta_{h_m}\Theta_k\nabla E_{n-k,r}=\mathop{\sum_{P\in \LD(m,n)^{\ast k}}}_{\touch(P)=r}q^{\dinv(P)}t^{\area(P)} x^P.
	\end{equation}	
\end{conjecture}
\begin{remark}
	It follows immediately from Corollary~\ref{cor:sumTouchingGD} that our touching generalized Delta conjecture implies the generalized Delta conjecture.
\end{remark}

We now state our \emph{compositional} refinement of the \emph{Delta conjecture}, i.e.\ of the case $k=0$ of Conjecture~\ref{conj:GenDelta}.
\begin{conjecture}[Compositional Delta] \label{conj:compDelta}
	Given $n,k\in \mathbb{N}$, $n>k\geq 0$ and $\alpha\vDash n-k$,
	\begin{equation}
	\Theta_k\nabla C_{\alpha}=\mathop{\sum_{P\in \LD(0,n)^{\ast k}}}_{\dcomp(P)=\alpha}q^{\dinv(P)}t^{\area(P)} x^P.
	\end{equation}	
\end{conjecture}
\begin{remark}
	We observe immediately that the compositional Delta conjecture implies the touching Delta conjecture: combinatorially we clearly have $\touch(P)=\ell(\dcomp(P))$, while on the symmetric function side we use \eqref{eq:Enk=sumCalpha}.
\end{remark}
\begin{remark}
	Notice that we have a compositional version only of the Delta conjecture, and not of the generalized Delta conjecture, i.e. only for the number $m$ of zero labels equal to $0$. For $m>0$ the dinv seems to be off at the compositional level. This situation should be compared with the final comments of \cite{DAdderio-Iraci-VandenWyngaerd-GenDeltaSchroeder}.
\end{remark}
Still, we can state the following conjecture.

\begin{conjecture}
	Given $n,k,m\in \mathbb{N}$, $m>0$, $n>k\geq 0$ and $\alpha\vDash n-k$,
	\begin{equation}
	\left.\Delta_{h_m}\Theta_k\nabla C_{\alpha}\right|_{q=1}=\mathop{\sum_{P\in \LD(m,n)^{\ast k}}}_{\dcomp(P)=\alpha}  t^{\area(P)} x^P .
	\end{equation}	
\end{conjecture}

\medskip

The following conjecture appeared in our work \cite{DAdderio-Iraci-VandenWyngaerd-Delta-Square}.
\begin{conjecture}[Generalized Delta square] \label{conj:GenDeltaSQ}
	Given $n,k,m\in \mathbb{N}$ with $n>k\geq 0$,
	\begin{equation}
	\frac{[n-k]_q}{[n]_q} \Delta_{h_m} \Delta_{e_{n-k}} \omega(p_n) = \sum_{P\in \mathsf{LSQ}(m,n)^{\ast k}} q^{\dinv(P)}t^{\area(P)} x^P.
	\end{equation}	
\end{conjecture}

We state our ``touching'' refinement of this conjecture.
\begin{conjecture}[Touching generalized Delta square] \label{conj:touchingGenDeltaSQ}
	Given $n,k,m,r\in \mathbb{N}$, $n>k\geq 0$ and $n-k\geq r\geq 1$,
	\begin{equation}
	\frac{[n]_q}{[r]_q} \Delta_{h_m} \Theta_k\nabla E_{n-k,r}=\mathop{\sum_{P\in \mathsf{LSQ}(m,n)^{\ast k}}}_{\touch(P)=r}q^{\dinv(P)}t^{\area(P)} x^P.
	\end{equation}	
\end{conjecture}
\begin{remark}
	It follows immediately from Corollary~\ref{cor:DeltaSqSum} that our touching generalized Delta square conjecture implies the generalized Delta square conjecture.
\end{remark}

\section{About the compositional version}

%\subsection{Relation to the $4$-variable Catalan theorem}
In this section we prove some results about our compositional Delta conjecture.

\subsection{Relation to the $4$-variable Catalan theorem}

In \cite{Zabrocki-4Catalan-2016} Zabrocki showed that for any composition $\alpha\vDash n-\ell$
\[\<\Delta_{h_{\ell}}\nabla C_\alpha,h_ke_{n-\ell-k}\>\]
is the $(\dinv,\area)$ $q,t$-enumerator of Dyck paths $P$ of size $n$ with $\ell$ decorated rises and $k$ decorated peaks with $\dcomp(P)=\alpha$. So this should match the Schr\"{o}der of our compositional Delta conjecture, i.e.\ it should be equal to
\[\<\Theta_\ell\nabla C_\alpha,h_ke_{n-k}\>.\]
Indeed, this follows immediately from the following lemma, which we prove in Section~\ref{sec:proofs_sec_Delkm}.

\begin{lemma} \label{lem:Schroeder_comp}
	For any $f\in \Lambda^{(n-\ell)}$,
	\begin{equation}
	\< \Delta_{h_{\ell}}f,h_ke_{n-\ell-k} \>= \left\<\Theta_\ell f, h_k e_{n-k} \right\>.
	\end{equation}
\end{lemma}

This shows that the $4$-variable Catalan result of Zarbrocki is really the Schr\"{o}der case of our compositional Delta conjecture.

%\subsection{Relation to the Dyck path algebra}
\subsection{Relation to the Dyck path algebra}

In \cite{Carlsson-Mellit-ShuffleConj-2015} the authors construct the Dyck path algebra by defining combinatorial operators that can be used to compute the $q,t,\underline{x}$-enumerators of specific subsets of labelled Dyck paths. The purpose of this subsection is to use these operators to extend their results to decorated labelled Dyck paths.

\subsubsection{Combinatorial translation}

The goal is to relate our compositional Delta conjecture to the operators of the Dyck path algebra from \cite{Carlsson-Mellit-ShuffleConj-2015}. Following \cite{Carlsson-Mellit-ShuffleConj-2015}, first we need to translate our $q,t,\underline{x}$-enumerator of $(\dinv,\area)$ into one of another bistatistic $(\inv,\bounce)$.

In order to do this we need to extend some definitions and bijections in \cite{Haglund_Loehr_conj_Hilbert} to the decorated setting (cf. also \cite[Chapters~2~and~3]{Haglund-Xin_Lecture-Notes}).

\begin{definition}
	Let $\pi \in \D(n)$. We define its \emph{bounce path} as a lattice path from $(0,0)$ to $(n,n)$ computed in the following way: it starts in $(0,0)$ and travels north until it encounters the beginning of an east step of $\pi$, then it turns east until it hits the main diagonal, then it turns north again, and so on; thus it continues until it reaches $(n,n)$. 
	
	We label the vertical steps of the bounce path starting from $0$ and increasing the labels by $1$ every time the path hits the main diagonal (so the steps in the first vertical segment of the path are labelled with $0$, the ones in the next vertical segment are labelled with $1$, and so on). We define the \emph{bounce word} of $\pi$ to be the string $b(\pi) = b_1(\pi) \cdots b_n(\pi)$ where $b_i(\pi)$ is the label attached to the $i$-th vertical step of the bounce path.
	
	We define the statistic \emph{bounce} on $\D(n)$ and $\LD(n)$ as \[ \bounce(\pi) \coloneqq \sum_{i=1}^n b_i(\pi). \]
\end{definition}

Let us start by describing Haglund's classical bijection (Theorem 3.15 in \cite{Haglund-Book-2008}) \[\zeta_0: \D(n)\rightarrow \D(n)\] that transforms $(\dinv, \area)$ into $(\area, \bounce)$. 

Take $\pi\in\D(n)$ and rearrange its area word in ascending order. This new word, call it $u$, will be the bounce word of $\zeta_0(\pi)$. We construct $\zeta_0(\pi)$ as follows. First draw the bounce path corresponding to $u$. The first vertical stretch and last horizontal stretch of $\zeta_0(\pi)$ are fixed by this path. For the section of the path  between consecutive peaks of the bounce path we apply the following procedure: place a pen on the top of the $i$-th peak of the bounce path and scan the area word of $D$ from left to right. Every time we encounter a letter equal to $i-1$ we draw an east step and when we encounter a letter equal to $i$ we draw a north step. By construction of the bounce path, we end up with our pen on top of the $(i+1)$-th peak of the bounce path. Note that in an area word a letter equal to $i\neq 0$ cannot appear unless it is preceded somewhere by a letter equal to $i-1$. This means that starting from the $i$-th peak, we always start with a horizontal step which explains why $u$ is indeed the bounce word of $\zeta_0(\pi)$. 

We want to extend $\zeta_0$ to $\LD(n)^{\ast k}$. For an element $(\pi,dr,w)\in \LD(n)^{\ast k}$ we apply $\zeta_0$ to $\pi$. We must now specify what happens to the decorated rises and the labelling of $\pi$. 

\begin{definition}
	A \emph{corner} (or \emph{valley}) of a Dyck path $\pi$ is one of the indices \[ c(\pi) \coloneqq \{2\leq i \leq n\mid a_{i}(\pi)\leq a_{i-1}(\pi)\}.\] Corners will often be identified with the vertical steps of $\pi$ that are directly preceded by a horizontal step.
\end{definition}

\begin{proposition}\label{prop: rises to corners}
	If $\pi\in \D(n)$ then there exists a bijection between $r(\pi)$ and $c(\zeta_0(\pi))$.
\end{proposition}
\begin{proof}
	Let $j\in r(\pi) $. It follows that  $a_{j}(\pi) = a_{j-1}(\pi)+1$. Take $i$ such that $a_{j-1}(\pi)=i-1$. While scanning the area word to construct the path between the $i$-th and $(i+1)$-th peak of the bounce path, we will encounter $a_{j-1}(\pi)=i-1$, directly followed by $a_{j}(\pi)=i$. This will correspond to a horizontal step followed by a vertical step in $\zeta_0(\pi)$ and thus to an element of $c(\zeta_0(\pi))$. Arguing backwards, it is easy to see that this gives the desired bijection.
\end{proof}

\begin{definition}
	Given $\pi\in \D(n)$, we define $\W'(\pi)$ to be the elements $(w_1,\dots, w_n) \in \mathbb N^n$ such that for every $i\in c(\pi)$, $w_i>w_{j(i)}$, where $j(i)$ is the index of the column containing the horizontal step preceding the $i$-th vertical step of $\pi$.  
\end{definition}

\begin{definition}
	The set of pairs $(\pi, dc)$  with  $\pi \in \D(n)$, $dc\subseteq c(\pi)$ and $\vert dc \vert = k $ will be denoted $\D'(n)^{\bullet k}$. The indices in $dc$ will be referred to as \emph{decorated corners}: we decorate the $i$-th vertical step of $\pi$ with a $\bullet$ for all $i\in dc$.
	
	The set of triples $(\pi, dc, w')$ with $(\pi,dc)\in \D'(n)^{\bullet k}$ and $w'\in \W'(\pi)$ will be denoted by $\mathsf{LD'} (n)^{\bullet k}$.  In this set, we represent $w'$ inside the squares containing the main diagonal $x=y$, starting from the bottom. See Figure~\ref{fig: zeta} on the right for an example.
\end{definition}

The following result (without decorations) first appeared in \cite{Haglund_Loehr_conj_Hilbert} (see also \cite[Chapter~5]{Haglund-Book-2008}). We sketch its proof for completeness.
\begin{proposition}\label{prop: zeta}
	There exists a bijection \[\zeta \colon \LD(n)^{\ast k} \rightarrow \LD'(n)^{\bullet k}.\]
\end{proposition}
\begin{proof}[Sketch of the proof]
	Take $(\pi, dr, w)\in \LD(n)^{\ast k}$. We want to define $\zeta(\pi, dr, w) \coloneqq (\pi',dc,w')\in \LD'(n)^{\bullet k}$. Naturally, we set $\pi'=\zeta_0(\pi)$. The decorated corners $dc$ are given by the bijection described in Proposition~\ref{prop: rises to corners}. Lastly, we set $w'$ to be the dinv reading word of $w$. We claim that  $w'\in \W'(\pi)$. Indeed,  $w\in \mathbb N^n$ is an element of $\W(\pi)$ if and only if for every $i\in r(\pi)$, $w_{i-1} < w_i$.  These rises correspond to the corners of $\zeta_0(\pi)$ by the bijection of Proposition~\ref{prop: rises to corners}, and the condition on the labels translates into the condition that the dinv reading word of $\pi$ is in $\W'(\pi)$: indeed, when representing $w'$ in the squares of the main diagonal, we ensure that the label contained in the column (respectively row) of a horizontal step (respectively vertical step) $s$ is the label of the step of $\pi$ encountered in the construction of $\zeta_0(\pi)$ at the moment of drawing $s$. It follows that a corner $c$ of $\zeta(\pi, dr, w)=(\pi',dc,w')$ is constructed by reading the two consecutive vertical steps of $\pi$ whose labels are the ones contained in the column and the row of the horizontal and vertical step of $c$, respectively. 
	
	The previous argument works also backwards, ensuring the bijectivity of $\zeta$.
	
	See Example~\ref{ex: zeta} for an illustration of this map. 
\end{proof}

We now extend the statistics on (labelled) Dyck paths to the decorated setting.

\begin{definition}
	Take $(\pi, dc)\in \D'(n)^{\bullet k}$ or $(\pi, dc,w')\in \LD'(n)^{\bullet k}$  and let $b_1(\pi)\cdots b_n(\pi)$ be the bounce word of $\pi$. Then we set
	\[\bounce(\pi, dc)=\bounce(\pi,dc, w')\coloneqq \sum_{i\not\in dc }b_i(\pi).\]
\end{definition}

\begin{definition}
	Consider $(\pi,dc,w')\in \LD'(n)^{\bullet k}$. An \emph{inversion} of $(\pi,w')$ is a pair $(i,j)$ with $i<j$ and $w'_i<w'_j$ such that the square whose upper right corner is the intersection of the lines $x=i$ and $y=j$ (i.e.\ the square lying in the same column as the label $w'_i$ and the same row as $w'_j$) lies under the path $\pi$. We set $\inv(\pi, w')=\inv(\pi, dc, w')$ to be the total number of inversions of $(\pi, w')$ (the $\inv$ statistic is also called $\area '$ in the literature).
\end{definition}

\begin{proposition}\label{prop:zetastats}
	The $\zeta$ map of Proposition~\ref{prop: zeta} is such that for $(\pi,dr,w)\in \LD(n)^{\ast k}$
	\begin{align*}
	&\dinv((\pi,dr,w))=\inv(\zeta(\pi,dr,w))\\
	&\area((\pi,dr,w))=\bounce(\zeta(\pi,dr,w)).
	\end{align*}
\end{proposition}

\begin{proof}
	For $k=0$, this result is exactly Remark~2.3 in \cite{Carlsson-Mellit-ShuffleConj-2015}. For $k>0$, it is enough to show that the number of squares in a row containing a decorated rise is equal to the label attached to the step of the bounce path that is in the same row of the vertical step that is part of the corresponding valley. But this is clear as it holds by construction for any vertical step of the preimage (not only for decorated rises).
\end{proof}

\begin{example} \label{ex: zeta}
	The left path of Figure~\ref{fig: zeta}, is an element of $\LD(8)^{\ast 3}$ with area word $(0,1,2,2,3,1,0,1)$, $dr=\{2,5,8\}$ and $w=(2,4,5,5,6,5,1,3)$. It follows that its area is $5$. Its primary inversions $\{(2,6)\}$, its secondary inversions $\{(2,7),(6,7),(3,8),(4,8)\}$, so its dinv equals $5$.
	
	The path on the right is an element of $\LD'(8)^{\bullet 3}$ and the image by $\zeta$ of the path on the left. It has $w'=(2,1,4,5,3,5,5,6)$. Its bounce word is $00111223$ and $dc=\{3,5,8\}$, so its bounce is $5$. Its inversions are $\{(2,3),(2,4),(3,4),(5,6), (5,7)\}$ so its inv equals $5$.

	\begin{figure*}[!ht]
		\begin{minipage}{.5\textwidth}
			\centering
			\begin{tikzpicture}[scale=.6]			
			\draw[gray!60, thin](0,0) grid (8,8) (0,0)--(8,8);
			\draw[blue!60, line width=1.6pt](0,0)--(0,3)--(1,3)--(1,5)--(4,5)|-(6,6)|-(8,8);
			
			\draw
			(.5,0.5) circle(0.4 cm) node {2}
			(.5,1.5) circle(0.4 cm) node {4}
			(.5,2.5) circle(0.4 cm) node {5}
			(1.5,3.5) circle(0.4 cm) node {5}
			(1.5,4.5) circle(0.4 cm) node {6}
			(4.5,5.5) circle(0.4 cm) node {5}
			(6.5,6.5) circle(0.4 cm) node {1}
			(6.5,7.5) circle(0.4 cm) node {3};
			
			\draw
			(-0.5,1.5) node {$\ast$}
			(0.5,4.5) node {$\ast$}
			(5.5,7.5) node {$\ast$};			
			\end{tikzpicture}
			
		\end{minipage}%
		\begin{minipage}{.5 \textwidth}
			\centering
			\begin{tikzpicture}[scale=0.6]
			\draw[gray!60, thin](0,0) grid (8,8);
			
			\draw[blue!60, line width=1.6pt] (0,0)|-(1,2)|-(2,4)|-(3,5)|-(7,7)|-(8,8);
			
			\draw[dashed, opacity=0.6, ultra thick] (0,0)|-(2,2)|-(5,5)|-(7,7)|-(8,8);
			
			\draw
			(.5,0.5) circle(0.4 cm) node {2}
			(1.5,1.5) circle(0.4 cm) node {1}
			(2.5,2.5) circle(0.4 cm) node {4}
			(3.5,3.5) circle(0.4 cm) node {5}
			(4.5,4.5) circle(0.4 cm) node {3}
			(5.5,5.5) circle(0.4 cm) node {5}
			(6.5,6.5) circle(0.4 cm) node {5}
			(7.5,7.5) circle(0.4 cm) node {6};
			
			\filldraw
			(.5,2.5) node {$\bullet$}
			(1.5,4.5) node {$\bullet$}
			(6.5,7.5) node{$\bullet$};			
			\end{tikzpicture}
		\end{minipage}
		\caption{An element in $\LD(n)^{\ast k}$ (left) and its image by the zeta map (right).}\label{fig: zeta}
	\end{figure*}
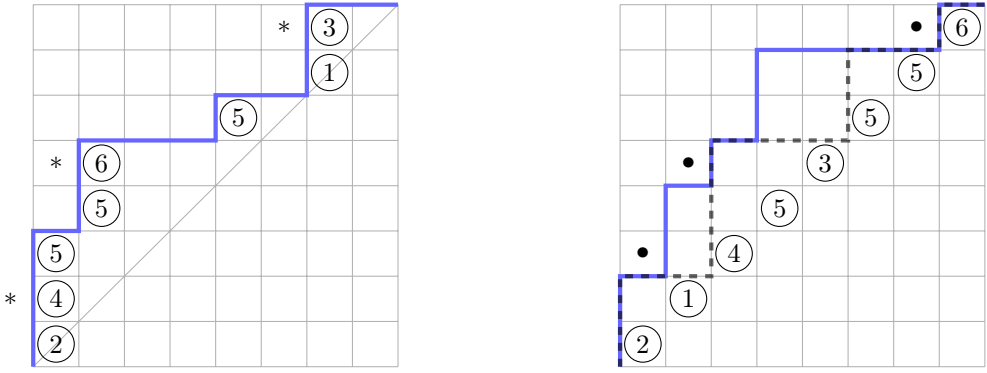
\end{example}

In order to complete our translation, we need to refine our sets according to the compositions.

\begin{definition}
	For $P \in \D'(n)^{\bullet k}\cup \LD'(n)^{\bullet k}$, let $\dcomp'(P) \coloneqq \dcomp(\zeta^{-1}(P))$.	We will see in Lemma~\ref{lem:decomp'} a way to compute $\dcomp'(P)$ directly on $P$. 
\end{definition}

Given a composition $\alpha\vDash n-k$, we set 
\begin{align*}
&\D(\alpha)^{\ast k} \coloneqq \{ P \in \D(n)^{\ast k} \mid \dcomp(P) = \alpha \}  
&&\D'(\alpha)^{\bullet k} \coloneqq \{ P\in \D'(n)^{\bullet k} \mid \dcomp'(P) = \alpha \}  \\
&\LD(\alpha)^{\ast k} \coloneqq \{ P \in \LD(n)^{\ast k} \mid \dcomp(P)=\alpha \} 
&&\LD'(\alpha)^{\bullet k} \coloneqq \{P\in \LD'(n)^{\bullet k} \mid \dcomp'(P)=\alpha \}. 
\end{align*}

\begin{remark} \label{rem:northsteps}
	Observe that by definition of $\dcomp'$, if $(\pi,dc)\in \D'(\alpha)^{\bullet k}$, then $\pi$ starts with $\ell \coloneqq \ell(\alpha)$ many north steps followed by an east step, i.e. $\pi$ is of the form $\pi=N^\ell\tilde{\pi}$ for a unique path $\tilde{\pi}$ from $(0,\ell)$ to $(|\alpha|+k,|\alpha|+k)$ that starts with an east step and stays weakly above the diagonal $x=y$.
\end{remark}

We define the $q,t,\underline{x}$-enumerator of this set as \[ \LD'_{q,t,\underline{x}}(\alpha)^{\bullet k} \coloneqq \sum_{(\pi, dc, w)\in \LD'(\alpha)^{\bullet k}} t^{\bounce(\pi,dc)} q^{\inv(\pi,w)} x^w , \]
where $x^w=\prod_ix_{w_i}$.

From the preceding discussion, the following identity is now clear:
\begin{equation} \label{eq:translation}
\LD'_{q,t,\underline{x}}(\alpha)^{\bullet k} = \mathop{\sum_{P\in \LD(0,n)^{\ast k}}}_{\dcomp(P)=\alpha}q^{\dinv(P)}t^{\area(P)} x^P,
\end{equation}
where the right hand side is precisely what appears in our compositional Delta conjecture. This completes our translation. Now we indicate a decomposition of these enumerators.

The following definitions and statements are from \cite{Carlsson-Mellit-ShuffleConj-2015}.

\begin{definition}
	For $\pi \in \D(n)$, we define the \emph{zero-weight characteristic function} of $\pi$ as \[ \overline{\chi}(\pi) \coloneqq \sum_{w \in \W'(\pi)} q^{\inv(\pi, w)} x^w. \]
\end{definition}

The zero-weight characteristic function is a special case of a more general weighted characteristic function $\chi(\pi, wt)$ (see \cite{Carlsson-Mellit-ShuffleConj-2015}*{Subsection~3.2}) where $wt \colon c(\pi) \rightarrow \mathbb{Q}(q,t)$ is any function: it corresponds to the case $wt = 0$. It is proved in \cite{Carlsson-Mellit-ShuffleConj-2015}*{Proposition~3.7} that the $\chi(\pi, wt)$'s are symmetric functions.

The following proposition follows immediately from the definitions.
\begin{proposition}
	We have
	\begin{equation}\label{eq:prop_decomp} \LD'_{q,t,\underline{x}}(\alpha)^{\bullet k} =\sum_{(\pi,dc,w)\in \LD'(\alpha)^{\bullet k}} t^{\bounce(\pi,dc)} q^{\inv(\pi,w)} x^w = \sum_{(\pi,dc)\in \D'(\alpha)^{\bullet k}} t^{\bounce(\pi,dc)} \overline{\chi}(\pi). 
	\end{equation}
\end{proposition}

\subsubsection{Dyck path algebra operators}

Again following \cite{Carlsson-Mellit-ShuffleConj-2015}, we now introduce the operators of the \emph{Dyck path algebra} in order to give an expression of $\LD'_{q,t,\underline{x}}(\alpha)^{\bullet k}$ in terms of them.

Given a polynomial $P$ depending on variables $u,v$, define the operator $\Upsilon_{uv}$ as

\begin{align*}
(\Upsilon_{uv} P)(u,v) & \coloneqq \frac{(q-1)vP(u,v) + (v-qu)P(v,u)}{v-u}%, \\ (\Upsilon^*_{uv} P)(u,v) & \coloneqq \frac{(q-1)uP(u,v) + (v-qu)P(v,u)}{v-u}.
\end{align*}

In \cite{Carlsson-Mellit-ShuffleConj-2015} this operator is called $\Delta_{uv}$, but we changed the notation in order to avoid confusion with the $\Delta_f$ operator defined on $\Lambda$.

\begin{definition}[\cite{Carlsson-Mellit-ShuffleConj-2015}*{Definition~4.2}]
	For $k \in \mathbb{N}$, define $V_k \coloneqq \Lambda[y_1, \dots, y_k]=\Lambda\otimes \mathbb{Q}[y_1,\dots,y_k]$. Let 
	\[T_i \coloneqq \Upsilon_{y_i y_{i+1}} \colon V_k \rightarrow V_k\text{ for }1 \leq i \leq k-1. \]
	
	We define the operators $d_+ \colon V_k \rightarrow V_{k+1}$ and $d_- \colon V_k \rightarrow V_{k-1}$: for $F[X]\in V_k$
	\begin{align*}
	(d_+ F)[X] & \coloneqq T_1 T_2 \cdots T_k (F[X + (q-1) y_{k+1}]) \\
	(d_- F)[X] & \coloneqq -F[X - (q-1)y_k] \sum_{i\geq 0} \left.  (-1/y_k )^{i}e_i[X]  \right|_{{y_k}^{-1}}.
	\end{align*}
	%	
	%Equivalently, we can define $d_-$ as $d_- ({y_k}^i F) \coloneqq -B_{i+1}F$ for $F \in V_k$ that does not depend on $y_k$.
\end{definition}

The idea is to use these operators to get the zero-weight characteristic functions $\bar{\chi}(\pi)$. In fact for our purposes it will be convenient to define the corresponding operators for \emph{partial Dyck paths}.
\begin{definition}
	Let $\ED^\ell(n)$ be the set of paths from $(0,\ell)$ to $(n,n)$ consisting of east or north unit steps, starting with an east step, and staying weakly above the diagonal $x=y$, together with their translates by vectors of the form $(v,v)$ with $v\in \mathbb{N}$. 
	
	Set $\ED^\ell = \bigsqcup_{n} \ED^\ell(n)$ and $\ED = \bigsqcup_{\ell } \ED^\ell $, and let $d \colon \ED \rightarrow V$ be defined as \[ d(\varnothing) = 1, \qquad d(E\tilde{\pi}) = d_+ d(\tilde{\pi}), \qquad d(EN^i \tilde{\pi}) = \frac{1}{q-1} [d_-, d_+] d_-^{i-1} d(\tilde{\pi}) \] for $\tilde{\pi} \in \ED$, where $[d_-,d_+]=d_-d_+-d_+d_-$ is the usual Lie bracket.
\end{definition}

\begin{definition}
	Given a Dyck path $\pi \in D(n)$, let $\tilde{\pi} \in \ED$ be the unique element such that $\pi = N^\ell \tilde{\pi}$ for some $\ell$. Then we define $d(\pi) \coloneqq d_-^\ell d(\tilde{\pi})$.
\end{definition}

We have the following fundamental theorem.
\begin{theorem}[\cite{Carlsson-Mellit-ShuffleConj-2015}*{Corollary~4.6}]
	For any Dyck path $\pi$,
	\begin{equation} \label{e:thm_CM}
	\overline{\chi}(\pi) = d(\pi).
	\end{equation}
\end{theorem}

The following corollary is now immediate.
\begin{corollary}
	We have
	\begin{equation} \label{eq:cor_decomp}
	\LD'_{q,t,\underline{x}}(\alpha)^{\bullet k} = \sum_{(\pi,dc)\in \D'(\alpha)^{\bullet k}} t^{\bounce(\pi,dc)} d(\pi). 
	\end{equation}
\end{corollary}
\begin{proof}
	Just combine \eqref{eq:prop_decomp} with \eqref{e:thm_CM}.
\end{proof}
\begin{remark} \label{rem:useful}
	Observe that, from \eqref{eq:cor_decomp}, the Remark~\ref{rem:northsteps}, and the definition of the function $d$, for $\ell \coloneqq \ell(\alpha)$ we have
	\begin{align}
	\notag \LD'_{q,t,\underline{x}}(\alpha)^{\bullet k} & = \sum_{(\pi,dc)\in \D'(\alpha)^{\bullet k}} t^{\bounce(\pi,dc)} d(\pi) \\
	\label{eq:reduction}& = \mathop{\sum_{(\pi,dc)\in \D'(\alpha)^{\bullet k}}}_{\pi=N^\ell \tilde{\pi}:\,  \tilde{\pi}\in \ED^\ell} t^{\bounce(\pi,dc)} d(N^\ell \tilde{\pi})\\
	\notag & = d_-^\ell \mathop{\sum_{(\pi,dc)\in \D'(\alpha)^{\bullet k}}}_{\pi=N^\ell \tilde{\pi}:\,  \tilde{\pi}\in \ED^\ell} t^{\bounce(\pi,dc)} d(\tilde{\pi})
	\end{align}
\end{remark}

%\subsubsection{Combinatorial recursion}

We need one more definition: given two compositions $\alpha=(\alpha_1,\dots,\alpha_r)$ and $\beta=(\beta_1,\dots,\beta_s)$, we define their \emph{concatenation} as $\alpha\beta \coloneqq (\alpha_1,\dots,\alpha_r,\beta_1,\dots,\beta_s)$, which is obviously also a composition.

We are now ready to express $\LD'_{q,t,\underline{x}}(\alpha)^{\bullet k}$ in terms of the operators of the Dyck path algebra. This is the main theorem of this section, and it extends \cite[Theorem~4.1]{Carlsson-Mellit-ShuffleConj-2015} to the decorated setting.

\begin{theorem}\label{thm: comp-recursion}
	If $\alpha$ is a composition of length $\ell$, then we have 
	\begin{equation} \label{eq:thm_Dyck_path_rel}
	\LD'_{q,t,\underline{x}}(\alpha)^{\bullet k} = d_-^\ell M_\alpha^{\ast k}
	\end{equation}
	where $M_\alpha^{\ast k} \in V_\ell$ is defined by the recursive relations 
	\begin{equation} \label{eq:recursion_1}
	M_{(1)\alpha}^{\ast k} = d_+ M_\alpha^{\ast k} + \frac{1}{q-1} [d_-, d_+] M_{\alpha(1)}^{\ast k-1},
	\end{equation}
	and for $a > 1$
	\begin{equation} \label{eq:recursion_a}
	M_{(a)\alpha}^{\ast k} = \frac{t^{a-1}}{q-1} [d_-, d_+] \left( \sum_{\beta \vDash a-1} d_-^{\ell(\beta)-1} M_{\alpha \beta}^{\ast k} + \sum_{\beta \vDash a} d_-^{\ell(\beta)-1} M_{\alpha \beta}^{\ast k-1} \right) ,
	\end{equation} 
	with initial conditions $M_\varnothing^{\ast k} = \delta_{k,0}$.
\end{theorem}
In order to prove this theorem, we need to define some combinatorial maps, and prove some lemmas concerning them.

\subsubsection{Combinatorial recursion}

\begin{definition}
	We define $\psi \colon \D(n)^{\ast k} \rightarrow \D(n-1)^{\ast k} \sqcup \D(n-1)^{\ast k-1}$ as follows: given $(\pi, dr)\in \D(n)^{\ast k}$ take the portion of $\pi$ between the first two touching points (or the whole path if there is only one touching point), remove its first (north) step and its last (east) step, and attach it to the end of the path. If the first rise was decorated we remove the decoration since it is no longer a rise. See Figure~\ref{fig:psi}. 
\end{definition}
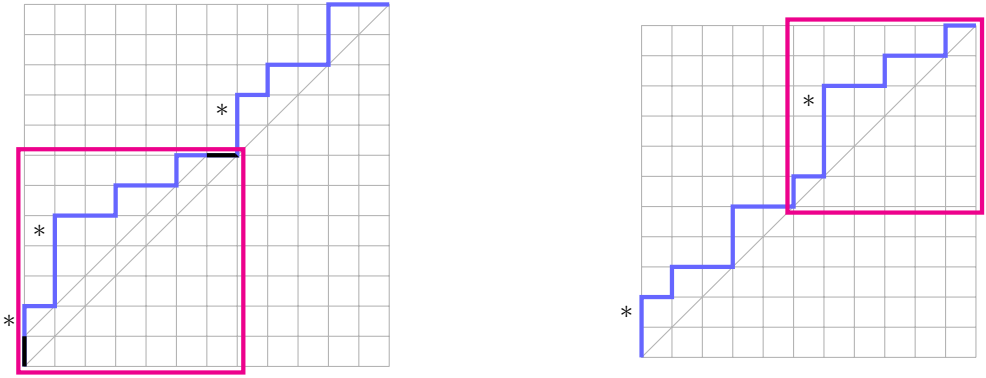
\begin{figure}[!ht]
	\centering
	\begin{minipage}{.5\textwidth}
		\centering
		\begin{tikzpicture}[scale=.4]
		\draw[step=1.0, gray!60, thin] (0,0) grid (12,12);	
		\draw[gray!60, thin] (0,0) -- (12,12);	
		\draw[gray!60, thin] (0,1) -- (6,7);
		
		\draw[blue!60, line width=1.6pt] (0,0) -- (0,1) -- (0,2) -- (1,2) -- (1,3) -- (1,4) -- (1,5) -- (2,5) -- (3,5) -- (3,6) -- (4,6) -- (5,6) -- (5,7) -- (6,7) -- (7,7) -- (7,8) -- (7,9) -- (8,9) -- (8,10) -- (9,10) -- (10,10) -- (10,11) -- (10,12) -- (11,12) -- (12,12);
		\draw[black, line width=1.6pt] (0,0) -- (0,1);
		\draw[black, line width=1.6pt, sharp <-sharp >, sharp > angle = 45] (6,7) -- (7,7);
		\draw[magenta, line width=1.6pt] (-0.2,-0.2) rectangle (7.2,7.2);
		
		\node at (-0.5,1.5) {$\ast$};
		\node at (0.5,4.5) {$\ast$};
		\node at (6.5,8.5) {$\ast$};	
		\end{tikzpicture}
	\end{minipage}%
	\begin{minipage}{.5\textwidth}
		\centering
		\begin{tikzpicture}[scale=.4]
		\draw[step=1.0, gray!60, thin] (0,0) grid (11,11);
		\draw[gray!60, thin] (0,0) -- (11,11);
		
		\draw[blue!60, line width=1.6pt] (0,0) -- (0,1) -- (0,2) -- (1,2) -- (1,3) -- (2,3) -- (3,3) -- (3,4) -- (3,5) -- (4,5) -- (5,5) -- (5,6) -- (6,6) -- (6,7) -- (6,8) -- (6,9) -- (7,9) -- (8,9) -- (8,10) -- (9,10) -- (10,10) -- (10,11) -- (11,11);
		\draw[magenta, line width=1.6pt] (4.8,4.8) rectangle (11.2,11.2);
		
		\node at (-0.5,1.5) {$\ast$};
		\node at (5.5,8.5) {$\ast$};
		\end{tikzpicture}
	\end{minipage}
	
	\caption{A path $(\pi,dr) \in \D((5,2,2))^{\ast 3}$ and its image $\psi(\pi, dr) \in \D((2,2,1,3,1))^{\ast 2}$. The first (north) and the last (east) steps (in black) in the highlighted section of the path are removed, and then the whole section is moved to the end. Since the first step of the section does not form a rise in the image, the corresponding decoration is removed.}
	\label{fig:psi}
\end{figure}

This map is linked to a family of maps, all of which are essentially its right inverses. 

For a composition $\alpha=(\alpha_1,\alpha_2,\dots,\alpha_\ell)$ and an integer $0\leq r\leq \ell$ we set    
\[ \alpha^r \coloneqq \left(\left( 1 + \sum_{i > r} \alpha_i \right) , \alpha_1, \alpha_2, \dots, \alpha_r \right), \qquad \alpha^{r,\ast}=\alpha^{r,\bullet} \coloneqq \left(\left( \sum_{i > r} \alpha_i \right) , \alpha_1, \alpha_2, \dots, \alpha_r \right). \]
We define two similar maps:
\begin{align*} 
&\psi_{r} \colon \D(\alpha)^{\ast k} \rightarrow \D(\alpha^r)^{\ast k} \\
&\psi^{\ast}_{r} \colon \D(\alpha)^{\ast k} \rightarrow \D(\alpha^{r,\ast})^{\ast k+1}. 
\end{align*} 
Given $(\pi,dr)\in \D(\alpha)^{\ast k}$ and $0\leq r \leq \ell(\alpha)$, call $\pi_1$ and $\pi_2$ the portions of $\pi$ below and above its $(r+1)$-th touching point, respectively, if $r\neq \ell(\alpha)$, while $\pi_1=\pi$ and $\pi_2=\emptyset$ if $r=\ell(\alpha)$. Notice that if $\pi_2\neq \emptyset$ then it necessarily starts with a north step. To define $\psi_{r}(\pi,dr)=(\pi',dr')$ we set \[ \pi'\coloneqq N\pi_2E\pi_1\] i.e. the path that starts at $(0,0)$ with a north step, followed by $\pi_2$, an east step and finally $\pi_1$. We use the same definition for $\psi^\ast_{r}(\pi,dr)$. For the decorations, we keep the decorations on the rises in the same place, relative to $\pi_1$ and $\pi_2$. When $\pi_1\neq \emptyset$, i.e. $r\neq \ell(\alpha)$, $\pi_1$ starts with a north step, and so $\pi'$ must start with two north steps, so the second step of $\pi'$ is a newly created rise, which we can choose to decorate or not. This choice is the difference between $\psi_{r}$ and $\psi^\ast_{r}$: for the former we do not decorate the new rise while for the latter we do.  It is clear from the definitions that $\dcomp(\psi_{r}(\pi,dr))=\alpha^r$ and $\dcomp(\psi^\ast_{r}(\pi,dr))=\alpha^{r,\ast}$.

\begin{definition}
	We define $\gamma: \D'(n)^{\bullet k} \rightarrow \D'(n-1)^{\bullet k} \sqcup \D'(n-1)^{\bullet k-1}$ which takes $(\pi, dc)\in \D'(n)^{\bullet k}$ and deletes the first $NE$ sequence of the path and if this $E$ step was part of a decorated corner, it removes its decoration since it is no longer a corner. See Figure~\ref{fig:gamma}. 
\end{definition}
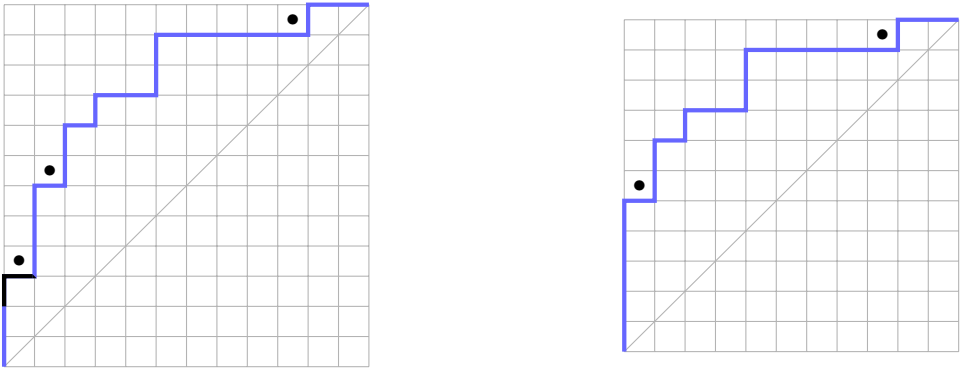
\begin{figure}[!ht]
	\centering
	\begin{minipage}{.5\textwidth}		
		\centering
		\begin{tikzpicture}[scale=.4]
		\draw[gray!60, thin] (0,0) grid (12,12);		
		\draw[gray!60, thin] (0,0) -- (12,12);
		
		\draw[blue!60, line width=1.6pt] (0,0) -- (0,1) -- (0,2) -- (0,3) -- (1,3) -- (1,4) -- (1,5) -- (1,6) -- (2,6) -- (2,7) -- (2,8) -- (3,8) -- (3,9) -- (4,9) -- (5,9) -- (5,10) -- (5,11) -- (6,11) -- (7,11) -- (8,11) -- (9,11) -- (10,11) -- (10,12) -- (11,12) -- (12,12);
		\draw[black, line width=1.6pt, sharp <-sharp >, sharp > angle = 45] (0,2) -- (0,3) -- (1,3);
		
		\node at (0.5,3.5) {$\bullet$};
		\node at (1.5,6.5) {$\bullet$};
		\node at (9.5,11.5) {$\bullet$};	
		\end{tikzpicture}
	\end{minipage}%
	\begin{minipage}{.5\textwidth}		
		\centering
		\begin{tikzpicture}[scale=.4]
		\draw[gray!60, thin] (0,0) grid (11,11);
		
		\draw[gray!60, thin] (0,0) -- (11,11);
		
		\draw[blue!60, line width=1.6pt] (0,0) -- (0,1) -- (0,2) -- (0,3) -- (0,4) -- (0,5) -- (1,5) -- (1,6) -- (1,7) -- (2,7) -- (2,8) -- (3,8) -- (4,8) -- (4,9) -- (4,10) -- (5,10) -- (6,10) -- (7,10) -- (8,10) -- (9,10) -- (9,11) -- (10,11) -- (11,11);
		
		\node at (0.5,5.5) {$\bullet$};
		\node at (8.5,10.5) {$\bullet$};
		\end{tikzpicture}
	\end{minipage}
	
	\caption{A path $(\pi,dc) \in \D'((5,2,2))^{\bullet 3}$ and its image $\gamma(\pi,dc) \in \D'((2,2,1,4))^{\bullet 2}$. The first NE pair (in black) is removed. Since the removed east step was part of a decorated valley, the corresponding decoration is removed. These two paths are actually the images of the paths in Figure~\ref{fig:psi} via $\zeta$.}
	\label{fig:gamma}
\end{figure}

Again, we have a family of right inverses. 
\begin{align*}
&\gamma_{r} \colon \D'(\alpha)^{\bullet k} \rightarrow \D'(\alpha^r)^{\bullet k}\\
&\gamma^{\bullet}_{r} \colon \D'(\alpha)^{\bullet k} \rightarrow \D'(\alpha^{r,\bullet})^{\bullet k+1}
\end{align*}

Take $(\pi, dc)\in \D'(\alpha)^{\bullet k}$. Set $\ell$ to be the number such that $\pi$ starts from the bottom with $\ell$ north steps followed by an east step. Define $\tilde \pi$ to be the portion of $\pi$ following its $\ell$ first vertical steps. Notice that by definition of the map $\zeta$ and $\dcomp'$, we have $\ell=\ell(\alpha)$.  Set $dc^{+j} \coloneqq \{i+j \mid i \in dc\}$. For $0\leq r\leq \ell$, we define
\[\gamma_{r}(\pi,dc) \coloneqq \left( N^{r+1}EN^{\ell-r}\tilde{\pi}, dc^{+1} \right), \]
i.e.\ we add one $NE$ sequence after the first $r$ North steps and we keep the decorated corners as they are, relative to $\pi$.

If $r \neq \ell$ we also have a map $\gamma^\bullet_{r} \colon \D'(\alpha)^{\bullet k} \rightarrow \D'(\alpha^{r,\bullet})^{\bullet k+1}$ defined as \[ \gamma^\bullet_{r}(\pi,dc) \coloneqq \left( N^{r+1}EN^{\ell-r}\tilde{\pi}, \{r+2\} \cup dc^{+1} \right), \] i.e. the path is defined in the same way as before, and we decorate the only new corner. 

It is not immediately clear that for $(\pi, dc)\in \D'(\alpha)^{\bullet k}$ we have $\gamma_{r}(\pi, dc)\in \D'(\alpha^r)^{\bullet k}$ and $\gamma^\bullet_{r}(\pi, dc)\in \D'(\alpha^{r,\bullet})^{\bullet k}$, but it follows from the following lemma.

\begin{lemma}
	\label{lem:mapproperties}
	\hfill
	\begin{enumerate}[(i)]
		\item	For all $\alpha\vDash n$ and $0\leq r \leq \ell(\alpha)$ we have 
		\begin{align*}
		\psi\circ \psi_{r}=\mathop{Id}\rvert_{\D(\alpha)^{\ast k}} && 
		\text{ and if } r\neq\ell(\alpha) \text{ then } && 
		\psi\circ \psi^\ast_{r}=\mathop{Id}\rvert_{\D(\alpha)^{\ast k}}\\ 
		\gamma\circ \gamma_{r}=\mathop{Id}\rvert_{\D'(\alpha)^{\bullet k}} &&&& \gamma\circ \gamma^\bullet_{r}=\mathop{Id}\rvert_{\D'(\alpha)^{\bullet k}}
		\end{align*}
		\item We have $\zeta \circ \psi= \gamma \circ \zeta: \D(n)^{\ast k} \rightarrow \D'(n-1)^{\bullet k} \sqcup \D'(n-1)^{\bullet k-1} $. 
		\item The following diagrams commute \[
		\begin{tikzcd}
		\D(\alpha)^{\ast k} \rar["\zeta"]\dar["\psi_{r}",swap] & \D'(\alpha)^{\bullet k}\dar["\gamma_{r}"] &&&& \D(\alpha)^{\ast k} \rar["\zeta"]\dar["\psi_{r}^\ast",swap] & \D'(\alpha)^{\bullet k}\dar["\gamma_{r}^\bullet"] \\
		\D(\alpha^r)^{\ast k} \rar["\zeta"] &\D'(\alpha^r)^{\bullet k} &&&& \D(\alpha^{r,\ast})^{\ast k+1} \rar["\zeta"] &\D'(\alpha^{r,\bullet})^{\bullet k+1} .
		\end{tikzcd}
		\]
	\end{enumerate}
\end{lemma}

\begin{proof}
	For no decorations, this is exactly \cite[Corollary~2.7]{Haglund-Xin_Lecture-Notes}. The exact same argument generalizes to our case, keeping track of the decorations, so we omit it.
\end{proof}

\begin{lemma}\label{lem:decomp'}
	Given $(\pi',dc)\in \D'(\alpha)^{\bullet k}$ we have
	\begin{align*}
	\alpha_i
	&=\bounce(\gamma_{i-1}(\pi, dc)))-\bounce(\gamma_{i}(\pi, dc))),\\
	&=\bounce(\gamma^\bullet_{i-1}(\pi, dc)))-\bounce(\gamma^\bullet_{i}(\pi, dc))).
	\end{align*}
\end{lemma}
\begin{proof}
	Consider $(\pi, dr)=\zeta^{-1}(\pi', dc)\in \D(\alpha)^{\ast k}$. It is easy to see from the definition of $\psi_{r}$ that \[\area(\psi_{r} (\pi, dr))=\area(\pi, dr)+\sum_{j> r} \alpha_j.\] It follows  that
	\begin{align*}
	\alpha_i &=\area(\psi_{i-1}(\pi, dr)))-\area(\psi_{i}(\pi, dr))).
	\end{align*} The result now follows from Lemma~\ref{lem:mapproperties}~(iii) and Proposition~\ref{prop:zetastats}. This proves the first equality. The second one follows from the first one together with the obvious
	\[ \bounce(\gamma_{r}(\pi,dc)) = \bounce(\gamma^\bullet_{r}(\pi,dc)) + 1 .\]
\end{proof}

\begin{proposition}
	\label{prop:disjoint-union}
	Let $a\in \mathbb{N}\setminus \{0\}$, and consider $(\pi, dc) \in \D'((a)\alpha)^{\bullet k}$. Then either $\gamma(\pi,dc) \in \D'(\alpha \beta)^{\bullet k}$ for some $\beta \vDash a-1$ or  $\gamma(\pi,dc) \in \D'(\alpha \beta)^{\bullet k-1}$ for some $\beta \vDash a$.
\end{proposition}

\begin{proof}
	Let $\dcomp'(\gamma(\pi,dc)) = \alpha'$. We have to prove that ${\alpha}_i' = \alpha_{i+1}$ for $1 \leq i < \ell(\alpha)$; in fact, if this is true, then necessarily $\alpha' = \alpha \beta$ for some $\beta \vDash a-1$ if the first corner in $\pi$ is decorated, or $\beta \vDash a$ if it is not (as the total size is the size of the composition plus the number of decorations, and applying $\gamma$ decreases the size by exactly one unit).
	
	Recall that \[ \alpha_i = \bounce(\gamma_{i-1}(\pi,dc)) - \bounce(\gamma_{i}(\pi,dc)). \] 
	So it will be sufficient to show that \[ \bounce(\gamma_{i}(\pi,dc)) = \bounce(\gamma_{i-1}(\gamma((\pi,dc))) \] for $1 \leq i \leq \ell(\alpha)$, as it implies our thesis by simply taking the relevant differences. But notice that the two bounce paths are identical from the first bouncing point onwards by construction (see Figure~\ref{fig:gamma-i}), as after the first bouncing point they belong to two identical regions (and whatever happens before is irrelevant, because the labels are all $0$'s). The thesis follows.
\end{proof}

\begin{figure}[!ht]
	\centering
	\begin{minipage}{.5\textwidth}		
		\centering
		\begin{tikzpicture}[scale=.5]
		\draw[gray!60, thin] (0,0) grid (13,13);		
		\draw[gray!60, thin] (0,0) -- (13,13);
		
		\draw[blue!60, line width=1.6pt] (0,0) -- (0,1) -- (0,2) -- (0,3) -- (1,3) -- (1,4) -- (2,4) -- (2,5) -- (2,6) -- (2,7) -- (3,7) -- (3,8) -- (3,9) -- (4,9) -- (4,10) -- (5,10) -- (6,10) -- (6,11) -- (6,12) -- (7,12) -- (8,12) -- (9,12) -- (10,12) -- (11,12) -- (11,13) -- (12,13) -- (13,13);
		\draw[green, line width=1.6pt, sharp <-sharp >, sharp > angle = 45] (0,2) -- (0,3) -- (1,3);
		\draw[black, line width=1.6pt, sharp <-sharp >, sharp > angle = 45] (1,3) -- (1,4) -- (2,4);
		\draw[dashed, opacity=0.6, line width=1.6pt] (0,0) -- (0,1) -- (0,2) -- (0,3) -- (3,3) -- (3,4) -- (3,5) -- (3,6) -- (3,7) -- (3,8) -- (3,9) -- (9,9) -- (9,10) -- (9,11) -- (9,12) -- (12,12) -- (12,13) -- (13,13);
		\draw[magenta, line width=1.6pt] (1.8,1.8) rectangle (13.2,13.2);
		
		\node at (1.5,4.5) {$\bullet$};
		\node at (2.5,7.5) {$\bullet$};
		\node at (10.5,12.5) {$\bullet$};
		\end{tikzpicture}
	\end{minipage}%
	\begin{minipage}{.5\textwidth}		
		\centering
		\begin{tikzpicture}[scale=.5]
		\draw[gray!60, thin] (0,0) grid (12,12);		
		\draw[gray!60, thin] (0,0) -- (12,12);
		
		\draw[blue!60, line width=1.6pt] (0,0) -- (0,1) -- (0,2) -- (1,2) -- (1,3) -- (1,4) -- (1,5) -- (1,6) -- (2,6) -- (2,7) -- (2,8) -- (3,8) -- (3,9) -- (4,9) -- (5,9) -- (5,10) -- (5,11) -- (6,11) -- (7,11) -- (8,11) -- (9,11) -- (10,11) -- (10,12) -- (11,12) -- (12,12);
		\draw[green, line width=1.6pt, sharp <-sharp >, sharp > angle = 45] (0,1) -- (0,2) -- (1,2);
		\draw[dashed, opacity=0.6, line width=1.6pt] (0,0) -- (0,1) -- (0,2) -- (2,2) -- (2,3) -- (2,4) -- (2,5) -- (2,6) -- (2,7) -- (2,8) -- (8,8) -- (8,9) -- (8,10) -- (8,11) -- (11,11) -- (11,12) -- (12,12);			
		\draw[magenta, line width=1.6pt] (0.8,0.8) rectangle (12.2,12.2);
		
		\node at (1.5,6.5) {$\bullet$};
		\node at (9.5,11.5) {$\bullet$};
		\end{tikzpicture}
	\end{minipage}
	
	\caption{The bounce paths of $\gamma_{2}(\pi,dc)$ and $\gamma_{1}(\gamma(\pi,dc))$, where $(\pi, dc)$ is the decorated Dyck path in Figure~\ref{fig:gamma}. The extra NE pairs are highlighted in green. Notice that the highlighted regions trivially coincide (since they do not contain the black NE pair removed by $\gamma$), and so does the bounce path within it.}
	\label{fig:gamma-i}
\end{figure}

\begin{remark}
	The strategy for the recursion in terms of $(\dinv, \area)$ is as follows: if $(\pi, dr)\in \D((a)\alpha)^{\ast k}$ then $\psi(\pi, dr)$ is a path whose diagonal composition is $\alpha\beta$ with either $\beta\vDash a-1$ (no decoration removed) or $\beta \vDash a$ (decoration removed). During this procedure the area always decreases by $a-1$. Except for the contribution to the dinv of the first step of $\pi$, which gets deleted, the dinv does not change overall: indeed for the steps that are moved, the primary dinv becomes secondary and vice versa.	
\end{remark}

\subsubsection{Proof of Theorem~\ref{thm: comp-recursion}}

We are now in a position to prove Theorem~\ref{thm: comp-recursion}: this subsection is dedicated to its proof.

\medskip

Thanks to Remark~\ref{rem:useful}, it will be enough to prove that
\begin{equation} \label{eq:the_key}
M_{\alpha}^{\ast k} = \mathop{\sum_{(\pi,dc)\in \D'(\alpha)^{\bullet k}}}_{\pi=N^\ell \tilde{\pi}:\,  \tilde{\pi}\in \ED^\ell} t^{\bounce(\pi,dc)} d(\tilde{\pi}).
\end{equation}

Let $\alpha$ be a composition with $\ell(\alpha) = \ell$, and let $k\in \mathbb{N}$. We want to prove  \eqref{eq:the_key} by induction on $\lvert \alpha \rvert + k$ (i.e.\ the size of the paths).

If $\alpha = \varnothing$, then $\ell=0$, and the only path with the empty composition is the empty path, which has no decorations and only admits the empty labelling. It follows that the sum is nonempty if and only if $k=0$. Also, $\bounce(\varnothing) = 0$, so that \[ \sum_{(\pi,dc)\in \D'(\varnothing)^{\bullet k}} t^{\bounce(\pi,dc)} d(\tilde{\pi}) = \delta_{k,0} d(\varnothing) = \delta_{k,0} = M_{\varnothing}^{\ast k}, \] hence the initial conditions match.

Consider now a nonempty composition, say $(a)\alpha$ with $a\in \mathbb{N}\setminus\{0\}$ and $\alpha$ a (possibly empty) composition of length $r$. Combining Lemma~\ref{lem:mapproperties} and Proposition~\ref{prop:disjoint-union}, we get the decomposition
\begin{equation}\label{eq:decompD'} \D'((a)\alpha)^{\bullet k} = \bigsqcup_{\beta \vDash a-1} \gamma_{r}(\D'(\alpha\beta)^{\bullet k}) \sqcup \bigsqcup_{\beta \vDash a} \gamma_{r}^\bullet (\D'(\alpha\beta)^{\bullet k-1}).
\end{equation}

So by the induction hypothesis we know that 
\begin{align}
M_{\alpha\beta}^{\ast k} & = \mathop{\sum_{(\pi,dc)\in \D'(\alpha\beta)^{\bullet k}}}_{\pi=N^{\ell(\alpha\beta)} \tilde{\pi}:\, \tilde{\pi}\in \ED^{\ell(\alpha\beta)}} t^{\bounce(\pi,dc)} d(\tilde{\pi})\quad \text{ for }\beta \vDash a-1,\text{ and}
\label{eq:recstep1}
\\ 
M_{\alpha\beta}^{\ast k-1} & = \mathop{\sum_{(\pi,dc)\in \D'(\alpha\beta)^{\bullet k-1}}}_{\pi=N^{\ell(\alpha\beta)} \tilde{\pi}:\,  \tilde{\pi}\in \ED^{\ell(\alpha\beta)}} t^{\bounce(\pi,dc)} d(\tilde{\pi})\quad \text{ for }\beta \vDash a \, .
\label{eq:recstep2}
\end{align}

Consider $(\pi, dc)\in \D'(\alpha\beta)^{\bullet k}$. Provided that $a>1$, we get 
\begin{equation} \label{eq:pitilde}
\gamma_{r}(\pi, dc)=(N^{\ell(\alpha)+1}EN N^{\ell(\beta)-1}\tilde \pi, dc^{+1}) 
\end{equation}
with $EN N^{\ell(\beta)-1}\tilde \pi\in \ED^{\ell(\alpha)+1}$ and  $\tilde \pi \in \ED^{\ell(\alpha \beta)}$, so that 
\begin{equation}\label{eq:chibargamma}
d(EN N^{\ell(\beta)-1}\tilde \pi)= \frac{[d_-,d_+]}{q-1} d_-^{\ell(\beta)-1}d (\tilde \pi).
\end{equation}
Notice that if $(\pi, dc)\in \D'(\alpha\beta)^{\bullet k-1}$ and we apply $\gamma_{r}^\bullet$, then we obtain the exact same identity.

If $a=1$ instead, in the first case, i.e.\ when we apply $\gamma_{r}$, we have $\beta \vDash 0$, so
\begin{equation} \label{eq:pitilde1}
\gamma_{r}(\pi, dc)=(N^{\ell(\alpha)+1} E \tilde \pi, dc^{+1}),
\end{equation}
with $E \tilde \pi\in \ED^{\ell(\alpha)+1}$ and  $\tilde \pi \in \ED^{\ell(\alpha)}$, from which we get 
\begin{equation}
\label{eq:chibara=1} d( E \tilde \pi)= d_+ d(\tilde \pi) 
\end{equation}

In the other case, i.e.\ when we apply $\gamma_{r}^\bullet$, we have $\beta \vDash 1$ instead, and it works as before.

Now we look at what happens to the bounce. By Lemma~\ref{lem:decomp'}, we have
\begin{align*}
\bounce(\gamma_{r}(\pi, dc))&= \bounce(\gamma_{r-1}(\pi, dc))-\alpha_r  \\
& = \bounce(\gamma_{r-2}(\pi, dc)) -\alpha_{r-1} -\alpha_r  \\
& \ \, \vdots \\
& = \bounce(\gamma_{0}(\pi, dc))  -\sum_{i=1}^r \alpha_i \\
& = \bounce(\gamma_{0}(\pi, dc)) - \vert \alpha \vert 
\end{align*} The same identity holds when replacing $\gamma_{r}$ with $\gamma^\bullet_{r}$ and $\gamma_{0}$ with $\gamma^\bullet_{0}$.  

Since $\bounce(\gamma_{0}(\pi, dc))= \bounce(\pi, dc)+\vert\alpha\beta\vert$, we get for $(\pi, dc)\in \D'(\alpha\beta)^{\bullet k}$
\begin{align}\label{eq:bounce}
\bounce(\gamma_{r}(\pi, dc))=\bounce(\pi, dc)+\vert \beta \vert=\bounce(\pi, dc)+a-1
\end{align}

Similarly, we have $\bounce(\gamma_{0}^\bullet(\pi, dc)) = \bounce(\pi, dc)+\vert\alpha\beta\vert-1$ and so for $(\pi, dc)\in \D'(\alpha\beta)^{\bullet k-1}$ we get 
\begin{align}\label{eq:bouncebullet}
\bounce(\gamma_{r}^\bullet(\pi, dc))= \bounce(\pi, dc) -1 +\vert \beta \vert=\bounce(\pi, dc) +a-1. 
\end{align}

We can conclude that when $a>1$, by applying \eqref{eq:recursion_a}, we get 	
\begin{align*}
M_{(a)\alpha}^{\ast k} & = \frac{t^{a-1}}{q-1} [d_-, d_+] \left( \sum_{\beta \vDash a-1} d_-^{\ell(\beta)-1} M_{\alpha \beta}^{\ast k} + \sum_{\beta \vDash a} d_-^{\ell(\beta)-1} M_{\alpha \beta}^{\ast k-1} \right) \\
\text{(using \eqref{eq:recstep1}, \eqref{eq:recstep2})}& =\sum_{\beta\vDash a-1}\mathop{\sum_{(\pi,dc)\in \D'(\alpha\beta)^{\bullet k}}}_{\pi=N^{\ell(\alpha\beta)} \tilde{\pi}:\, \tilde{\pi}\in \ED^{\ell(\alpha\beta)}} t^{a-1}t^{\bounce(\pi,dc)}\frac{[d_-, d_+]}{q-1}   d_-^{\ell(\beta)-1} d(\tilde{\pi})\\
& \quad + \sum_{\beta\vDash a}\mathop{\sum_{(\pi,dc)\in \D'(\alpha\beta)^{\bullet k-1}}}_{\pi=N^{\ell(\alpha\beta)} \tilde{\pi}:\,  \tilde{\pi}\in \ED^{\ell(\alpha\beta)}} t^{a-1}t^{\bounce(\pi,dc)}\frac{[d_-, d_+]}{q-1}  d_-^{\ell(\beta)-1} d(\tilde{\pi})\\
\text{(using \eqref{eq:bounce}, \eqref{eq:bouncebullet} and \eqref{eq:chibargamma})}& =\sum_{\beta\vDash a-1}\mathop{\sum_{(\pi,dc)\in \D'(\alpha\beta)^{\bullet k}}}_{\pi=N^{\ell(\alpha\beta)} \tilde{\pi}:\, \tilde{\pi}\in \ED^{\ell(\alpha\beta)}} t^{\bounce(\gamma_{r} (\pi, dc))} d(ENN^{\ell(\beta)-1}\tilde{\pi})\\
& \quad + \sum_{\beta\vDash a}\mathop{\sum_{(\pi,dc)\in \D'(\alpha\beta)^{\bullet k-1}}}_{\pi=N^{\ell(\alpha\beta)} \tilde{\pi}:\,  \tilde{\pi}\in \ED^{\ell(\alpha\beta)}} t^{\bounce(\gamma_{r}^\bullet(\pi, dc))}  d(ENN^{\ell(\beta)-1}\tilde{\pi})\\
\text{(using \eqref{eq:decompD'} and \eqref{eq:pitilde})}& =\mathop{\sum_{(\pi,dc)\in \D'((a)\alpha )^{\bullet k}}}_{\pi=N^{\ell(\alpha)+1} \tilde{\pi}:\, \tilde{\pi}\in \ED^{\ell(\alpha)+1}} t^{\bounce(\pi, dc)} d(\tilde{\pi})
\end{align*}
which is precisely \eqref{eq:reduction}.

The argument for $a=1$ is very similar but it uses \eqref{eq:recursion_1}, \eqref{eq:pitilde1} and \eqref{eq:chibara=1} instead of \eqref{eq:recursion_a}, \eqref{eq:pitilde} and \eqref{eq:chibargamma}. This completes the proof of Theorem~\ref{thm: comp-recursion}.

\subsubsection{Operator version of the compositional Delta conjecture}

We can state our \emph{operator version} of the compositional Delta conjecture.
\begin{conjecture}[Operator Delta conjecture] \label{conj:operatDelta}
	If $\alpha$ is a composition of length $\ell$, then
	\begin{equation}
	\Theta_k\nabla C_\alpha = d_-^\ell M_\alpha^{\ast k},
	\end{equation}
	with $M_\alpha^{\ast k}$ defined as in Theorem~\ref{thm: comp-recursion}.
\end{conjecture}

The following proposition is an immediate consequence of Theorem~\ref{thm: comp-recursion} and \eqref{eq:translation}.
\begin{proposition}
	The compositional Delta conjecture, i.e.\ Conjecture~\ref{conj:compDelta}, is equivalent to the operator Delta conjecture, i.e.\ to Conjecture~\ref{conj:operatDelta}.
\end{proposition}

\section{About the touching version}

\subsection{Proof of the generalized shuffle conjecture}
%\section{Proof of generalized shuffle conjecture}

In this section we prove the \emph{generalized shuffle conjecture}, i.e.\ the special case $k=0$ of Conjecture~\ref{conj:touchingGenDelta}, which is also Conjecture~7.5 in \cite{Haglund-Remmel-Wilson-2015}.

\begin{theorem}[Touching generalized shuffle] \label{thm:GenShuffle}
	Given $n,m,r\in \mathbb{N}$ with $n\geq r\geq 1$,
	\begin{equation} \label{eq:GenShuffle}
	\Delta_{h_m}\nabla E_{n,r}=\mathop{\sum_{P\in \LD(m,n)^{\ast 0}}}_{\touch(P)=r}q^{\dinv(P)}t^{\area(P)} x^P.
	\end{equation}
\end{theorem}

This theorem extends the shuffle theorem of Carlsson and Mellit \cite{Carlsson-Mellit-ShuffleConj-2015}.

\medskip

The rest of this subsection is devoted to the proof of this result. We start by introducing some convenient combinatorial objects.

\begin{definition}
	An \emph{augmented Dyck path} is a labelled Dyck path with labels in $\mathbb{N} \cup \{\infty\}$, such that all the labels $\infty$ lie on the main diagonal and they are the only labels in their own columns. See Figure~\ref{fig:pushing} for two examples.
\end{definition}

We denote by $\LD(m,n,s)$ the set of augmented Dyck paths of size $m+n+s$ with $m$ labels $0$ and $s$ labels $\infty$. We compute the statistics as follows: the area is not influenced by the labels. The definition of the dinv stays the same, using the following conventions: $\ell < \infty$ for each positive label $\ell$, but $0 \not < \infty$ and $0\not > \infty$. Finally, to compute the relevant monomials, we set $x_\infty = 1$. For example, the augmented Dyck path to the right of Figure~\ref{fig:pushing} has area $5$, dinv $16$, and the corresponding monomial equals $x_1 x_2 x_3 x_4 x_5 x_6 x_7$.

We have the following combinatorial identity.
\begin{proposition}
	Given $n,s\in \mathbb{N}$ with $n \geq 1$
	\begin{equation} \label{eq:prop_combId}
	\sum_{P \in \LD(0,n,s)} q^{\dinv(P)} t^{\area(P)} x^P = \sum_{r=1}^n \qbinom{r+s}{r}_q \mathop{\sum_{P\in \LD(0,n)}}_{\touch(P) = r} q^{\dinv(P)} t^{\area(P)} x^P.
	\end{equation}
\end{proposition}

\begin{proof}
	Given any path $P \in \LD(0,n)$ with $\touch(P) = r$, we can insert $s$ labels $\infty$ on the diagonal, and the contribution to the dinv given by those labels is exactly $\qbinom{r+s}{r}_q$: choose an interlacing between the $s$ labels $\infty$ and the $r$ labels on the diagonal; each time one of the latters precedes one of the formers, a unit of dinv is created. Summing over all the possible values of $r$ we get the desired identity.
\end{proof}

Applying nabla to \eqref{eq:en_q_sum_Enk}, we get
\begin{equation}
\label{eq:shuffle_syst}
\nabla e_n\left[X[s+1]_q\right] =\sum_{r=1}^n \qbinom{r+s}{r}_q \nabla E_{n,r}.
\end{equation}

Recall that Carlsson and Mellit in \cite{Carlsson-Mellit-ShuffleConj-2015} showed in particular that
\begin{equation}
\nabla E_{n,r} = \mathop{\sum_{P\in \LD(0,n)}}_{\touch(P) = r} q^{\dinv(P)} t^{\area(P)} x^P.
\end{equation}

Combining all this with \eqref{eq:prop_combId} we get immediately the following corollary.
\begin{corollary}
	For $n, s \in \mathbb{N}$ with $n\geq 1$, we have
	\begin{equation}
	\label{eq:s_shuffle_thm}
	\nabla e_n\left[X[s+1]_q\right]=\sum_{P \in \LD(0,n,s)} q^{\dinv(P)} t^{\area(P)}x^P.
	\end{equation}
\end{corollary}

Our goal is to show the following more general theorem.
\begin{theorem}
	\label{thm:comb_GenShuffle}
	For $m, n, s \in \mathbb{N}$ with $n\geq 1$, we have
	\begin{equation}
	\label{eq:gen_shuffle_aux}
	\Delta_{h_m} \nabla e_n\left[X[s+1]_q\right] = \sum_{P \in \LD(m,n,s)} q^{\dinv(P)} t^{\area(P)} x^P.
	\end{equation}
\end{theorem}

On the symmetric function side, the key identity is provided by the following theorem, that we prove in Section~\ref{sec:proofs_sec_Delkm}.
\begin{theorem}
	\label{thm:SF_GenShuffle}
	For $n, j, s \in \mathbb{N}$ with $n\geq 1$, we have
	\begin{equation} \label{eq:SF_GenShuffle}
	h_j^\perp \nabla e_n[X[s+1]_q] = \sum_{p=0}^j t^{j-p} \qbinom{s+p}{p}_q \Delta_{h_{j-p}} \nabla e_{n-j}[X[s+p+1]_q].
	\end{equation}
\end{theorem}
\begin{proof}[Proof of Theorem~\ref{thm:comb_GenShuffle}]
	We proceed by induction on $m$. The base case $m=0$ is exactly \eqref{eq:s_shuffle_thm}, which is proved. Now suppose that the thesis is true for $m < j$. To show that the thesis holds for $j$, we need a combinatorial interpretation of \eqref{eq:SF_GenShuffle}. Recall that, as by definition $\< h_j^\perp f, h_\mu \> = \< f, h_j h_\mu \>$, a combinatorial interpretation of $h_j^\perp \nabla e_n[X[s+1]_q]$ is given by the $q,t,\underline{x}$-enumerator of the elements in $P \in \LD(0,n,s)$ whose reading word is a $(j, \mu_1, \dots, \mu_{\ell(\mu)})$-shuffle, where the $j$ biggest labels (which we call \emph{big labels}) do not contribute to the monomial (see \cite[Chapter~6]{Haglund-Book-2008} for the theory of shuffle).
	
	Notice that each occurrence of a big label must be a peak, since all the $\infty$ labels lie on the diagonal. Suppose that exactly $p$ of these labels lie on the diagonal, while $j-p$ do not. First we replace the $p$ big labels on the diagonal by $\infty$ labels: this gives a contribution of $\qbinom{s+p}{p}_q$ to the dinv depending on the interlacing of these $p$ labels with the $s$ many $\infty$ labels that were already on the diagonal. Then we ``push in the peaks'', namely we move the vertical step in the rows containing a big label right after the following horizontal step, and we replace the label with a $0$: see Figure~\ref{fig:pushing} for an example.
	
	\begin{figure}
		\centering
		\begin{minipage}{.5\textwidth}
			\centering
			\begin{tikzpicture}[scale = 0.55]
			\draw[gray!60, thin] (0,0) grid (12,12);
			\draw[gray!60, thin] (0,0) -- (12,12);
			
			\draw[blue!60, line width=1.6pt] (0,0) -- (0,1) -- (0,2) -- (1,2) -- (2,2) -- (2,3) -- (3,3) -- (3,4) -- (3,5) -- (4,5) -- (4,6) -- (5,6) -- (5,7) -- (5,8) -- (6,8) -- (7,8) -- (8,8) -- (8,9) -- (9,9) -- (9,10) -- (9,11) -- (10,11) -- (11,11) -- (11,12) -- (12,12);
			
			\draw
			(0.5,0.5) circle(0.4 cm) node {$2$}
			(0.5,1.5) circle(0.4 cm) node {$6$}
			(2.5,2.5) circle(0.4 cm) node {$\infty$}
			(3.5,3.5) circle(0.4 cm) node {$7$}
			(3.5,4.5) circle(0.4 cm) node[red] {$9$}
			(4.5,5.5) circle(0.4 cm) node {$1$}
			(5.5,6.5) circle(0.4 cm) node {$4$}
			(5.5,7.5) circle(0.4 cm) node[red] {$8$}
			(8.5,8.5) circle(0.4 cm) node[red] {$10$}
			(9.5,9.5) circle(0.4 cm) node {$3$}
			(9.5,10.5) circle(0.4 cm) node {$5$}
			(11.5,11.5) circle(0.4 cm) node {$\infty$};
			\end{tikzpicture}
		\end{minipage}%
		\begin{minipage}{.5\textwidth}
			\centering
			\begin{tikzpicture}[scale = 0.55]
			\draw[gray!60, thin] (0,0) grid (12,12);
			\draw[gray!60, thin] (0,0) -- (12,12);
			
			\draw[blue!60, line width=1.6pt] (0,0) -- (0,1) -- (0,2) -- (1,2) -- (2,2) -- (2,3) -- (3,3) -- (3,4) -- (4,4) -- (4,5) -- (4,6) -- (5,6) -- (5,7) -- (6,7) -- (6,8) -- (7,8) -- (8,8) -- (8,9) -- (9,9) -- (9,10) -- (9,11) -- (10,11) -- (11,11) -- (11,12) -- (12,12);
			
			\draw
			(0.5,0.5) circle(0.4 cm) node {$2$}
			(0.5,1.5) circle(0.4 cm) node {$6$}
			(2.5,2.5) circle(0.4 cm) node {$\infty$}
			(3.5,3.5) circle(0.4 cm) node {$7$}
			(4.5,4.5) circle(0.4 cm) node[red] {$0$}
			(4.5,5.5) circle(0.4 cm) node {$1$}
			(5.5,6.5) circle(0.4 cm) node {$4$}
			(6.5,7.5) circle(0.4 cm) node[red] {$0$}
			(8.5,8.5) circle(0.4 cm) node[red] {$\infty$}
			(9.5,9.5) circle(0.4 cm) node {$3$}
			(9.5,10.5) circle(0.4 cm) node {$5$}
			(11.5,11.5) circle(0.4 cm) node {$\infty$}
			(3.5,4.5) node {$\rightarrow$}
			(5.5,7.5) node {$\rightarrow$};
			\end{tikzpicture}
		\end{minipage}
		\caption{The pushing algorithm for a path of size $12$ with $3$ big labels and $2$ $\infty$ labels. On the left, the big labels appear in red.}
		\label{fig:pushing}
	\end{figure}
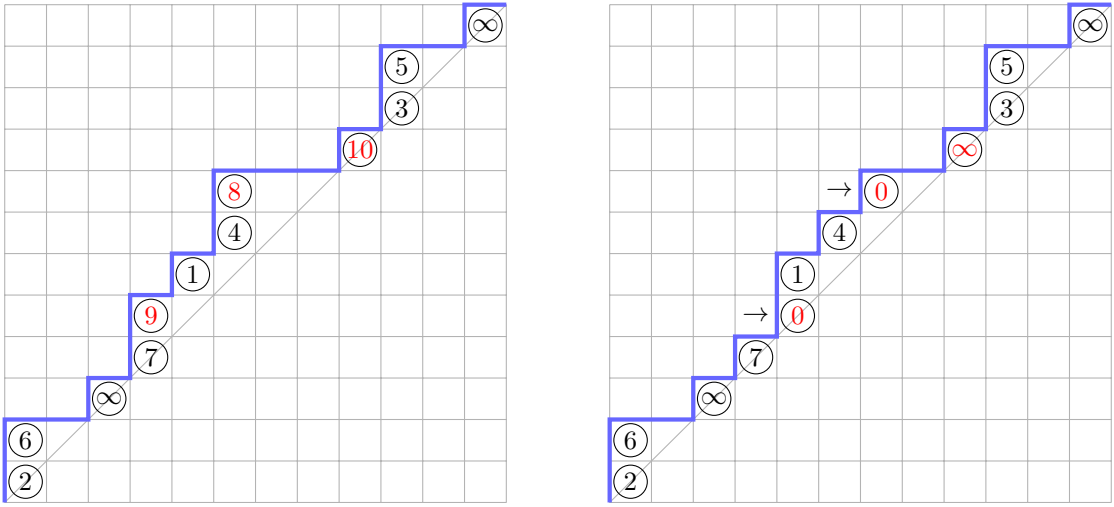
	
	It is easy to check that this operation decreases the area by $j-p$ (taken into account by the factor $t^{j-p}$) and does not change the dinv (as it sends the primary inversions involving these labels to secondary inversions and vice versa). This way we get a path in $\LD(j-p, n, s+p)$. By inductive hypothesis, for $p>0$ we have that $\Delta_{h_{j-p}} \nabla e_{n-j}[X[s+p+1]_q]$ is the $q,t,\underline{x}$-enumerator  of this set. Replacing this in \eqref{eq:SF_GenShuffle}, and taking into account the combinatorial interpretation of the left hand side, we get
	
	\begin{align*}
	& \hspace{-0.5cm}\sum_{p=0}^j t^{j-p} \qbinom{s+p}{p}_q \LD_{q,t,\underline{x}}(j-p, n, s+p)= \\
	& = t^j  \Delta_{h_j} \nabla e_n[X[s+1]_q] + \sum_{p=1}^j t^{j-p} \qbinom{s+p}{p}_q \LD_{q,t,\underline{x}}(j-p, n, s+p)
	\end{align*}
	
	which simplifies to 
	
	\[ \LD_{q,t,\underline{x}}(j, n, s) =  \Delta_{h_j} \nabla e_n[X[s+1]_q] \]
	
	which is what we wanted to prove.
\end{proof}

Applying $\Delta_{h_m}$ to \eqref{eq:shuffle_syst} we get the symmetric function counter part of \eqref{eq:prop_combId}, i.e.
\begin{equation} \label{eq:gen_shuffle_syst}
\Delta_{h_m}\nabla e_n\left[X[s+1]_q\right] =\sum_{r=1}^n \qbinom{r+s}{r}_q \Delta_{h_m}\nabla E_{n,r} .
\end{equation}

Now to see that  \eqref{eq:gen_shuffle_aux} implies our touching generalized shuffle conjecture, observe that \eqref{eq:gen_shuffle_syst} for $s=0,1,\dots,n-1$ gives a system of linear equations relating the symmetric functions \[\{\Delta_{h_m}\nabla e_{n}[X[s+1]_q]\}_{0\leq s\leq n-1}\] with $\{\Delta_{h_m}\nabla E_{n,r}\}_{1\leq r\leq n}$, whose coefficient matrix is $\left\|\qbinom{r+s}{r}_q \right\|_{0< r\leq n,0\leq s<n}\in M_{n\times n}(\mathbb{Q}(q))$. 
Clearly the determinant of this matrix is in $\mathbb{Z}[q]$, and in fact setting $q=1$ in it we get $\det \left\|\binom{r+s}{r}\right\|_{0< r\leq n,0\leq s<n}$. 
The following lemma shows that this determinant is $1$.
\begin{lemma}
	For $n\geq 1$, \[\det \left\|\binom{r+s}{r}\right\|_{0< r\leq n,0\leq s<n}=1.\]
\end{lemma}
\begin{proof}
	Set $A_n \coloneqq \left\|\binom{r+s}{r}\right\|_{0< r\leq n,0\leq s<n}$, and call $R(i,j)$ (resp. $C(i,j)$) the row (resp. column) operator that replaces the $i$-th row (resp. column) of a matrix with itself minus the $j$-th one. Notice that these operators do not change the determinant. 
	
	Let us consider the matrix
	\[B_n \coloneqq C(2,1)C(3,2)\cdots C(n-1,n-2)C(n,n-1)R(2,1)R(3,2)\cdots R(n-1,n-2)R(n,n-1) A_n.\]
	By construction, it is immediate to check that in the first row of $B_n$ all coefficients are $0$, except the leftmost one, which is equal to $1$.
	
	On the other hand, using the well-known identity $\binom{n}{k}=\binom{n-1}{k}+\binom{n-1}{k-1}$, we can check that first for $r'>0$ the row operators changed the coefficient $\binom{r'+s}{r'}$ into $\binom{r'+s-1}{r'}$, and then for $s>0$ the column operators changed the latter into $\binom{r'-1+s-1}{r'-1}$. So the submatrix of $B_n$ corresponding to the indices $1\leq r',s < n$ is actually equal to $A_{n-1}$. By induction on $n$ (the case $n=1$ being trivial), we conclude that $\det A_n=\det B_n=\det A_{n-1}=1$, as we wanted.
\end{proof}

From this it follows that the system \eqref{eq:gen_shuffle_syst} for $s=0,1,\dots,n-1$ has a unique solution, so that the $\{\Delta_{h_m}\nabla E_{n,r}\}_{1\leq r\leq n}$ are uniquely determined by the $\{\Delta_{h_m}\nabla e_{n}[X[s+1]_q]\}_{0\leq s\leq n-1}$ and by the system.

Now the combinatorial counterpart of \eqref{eq:gen_shuffle_syst} is \eqref{eq:prop_combId}, so that also the combinatorial interpretations of $\{\Delta_{h_m}\nabla e_{n}[X[s+1]_q]\}_{0\leq s\leq n-1}$ and $\{\Delta_{h_m}\nabla E_{n,r}\}_{1\leq r\leq n}$ are in the same relation. From this we can conclude that the proof of the combinatorial interpretation for $\Delta_{h_m}\nabla e_{n}[X[s+1]_q]$ implies the one for $\Delta_{h_m}\nabla E_{n,r}$, which is what we wanted. This completes the proof of Theorem~\ref{thm:GenShuffle}.

\subsection{Proof of the generalized square conjecture}
In this section we show that the touching generalized shuffle conjecture (Theorem~\ref{thm:GenShuffle}) implies the touching generalized square conjecture: 
\begin{theorem}[Touching generalized square] \label{thm: gensquare}
	Given $n,m,r\in \mathbb{N}$ with $n\geq r\geq 1$,
	\begin{equation} 
	\frac{[n]_q}{[r]_q}\Delta_{h_m}\nabla E_{n,r}=\mathop{\sum_{P\in \LSQ(m,n)^{\ast 0} }}_{\touch(P)=r}q^{\dinv(P)}t^{\area(P)} x^P.
	\end{equation}
\end{theorem}
The rest of this subsection is dedicated to the proof of this theorem.

\medskip

We begin by giving equivalent well-known formulations of Theorem~\ref{thm:GenShuffle} and Theorem~\ref{thm: gensquare}. In order to do this, we need some definitions. 
\begin{definition}
	A \emph{partial preference function} is a  partially labelled square path in $\LSQ(m,n)^{\ast 0}$ whose nonzero labels are exactly the numbers $1,2,\dots, n$. The subset of these paths is denoted by $\Pref(m,n)$.
	
	A \emph{partial parking function} is a partial preference function whose shift is $0$. The set of such paths is denoted by $\Park(m,n)$.
\end{definition}
\begin{definition}
	For every $S\subseteq \{1,2,\dots,n-1\}$, let $Q_{S,n}$ denote the \emph{Gessel fundamental quasisymmetric function} of degree $n$ indexed by $S$ , i.e.
	\begin{equation}
	Q_{S,n} \coloneqq \mathop{\sum_{i_1\leq i_2\leq \cdots \leq i_n}}_{i_j<i_{j+1}\text{ if }j\in S}x_{i_1} x_{i_2}  \cdots   x_{i_n}.
	\end{equation}
\end{definition}

\begin{definition}
	\label{def: ides}
	The \emph{descent set} of a given permutation $\tau$ is the set of indices $i$ such that $\tau_i>\tau_{i+1}$, denoted by $\mathsf{Des}(\tau)$. We define the \emph{inverse descent set  of a permutation} $\tau$ as \[\ides(\tau)\coloneqq \mathsf{Des}(\tau^{-1}). \]
	
	Take $P\in \Pref(m,n)$. We set $r(P)$ to be the reverse reading word of $P$ (see Definition~\ref{def: reading word}). The \emph{inverse descent set of a preference function} $P$ is defined by \[\ides(P)\coloneqq \ides(r(P)) \]  
\end{definition}

Theorem~\ref{thm:GenShuffle} and Theorem~\ref{thm: gensquare} are equivalent to the following equations, respectively 
\begin{align}
\Delta_{h_m} \nabla E_{n,r} &=\sum_{\substack{P\in \Park(m,n)\\ \touch(P)=r}}q^{\dinv(D)} t^{\area(P)} Q_{\ides(P), n} \\
\label{eq:touchgensquare}	\frac{[n]_q}{[r]_q}\Delta_{h_m} \nabla E_{n,r} &= \sum_{\substack{P\in \Pref(m,n)\\ \touch(P)=r}} q^{\dinv(P)} t^{\area(P)} Q_{\ides(P), n}.
\end{align}

The argument is essentially identical to the one given for the shuffle conjecture in \cite[Chapter~6]{Haglund-Book-2008}, so we omit it.

We will show here that the first of these equations implies the second one. We will use the same strategy that Sergel used in \cite{Leven-2016}, where she proved that the shuffle conjecture implies the square conjecture. Notice that our notation will be slightly different from the one in \cite{Leven-2016}.

\begin{definition}
	Take $P\in \Pref(m,n)$ with shift $s$ and set $l\coloneqq \max \{a_i(P) \mid 1\leq i \leq m+n \}$. 
	For $i=l,l-1,\dots, -s$ define $\rho_i$ to be the labels of $P$ contained in the diagonal $y=x+i$, in increasing order. Define the \emph{diagonal word} of $P$ as
	\[\diagword(P)\coloneqq \rho_l\rho_{l-1}\cdots \rho_0 \rho_{-1} \cdots \rho_{-s}.\] It is clear from the definition of a preference function that the $\rho_i$ are the maximal substrings of consecutive weakly increasing numbers of $\diagword(P)$: we will call them \emph{runs} of $\diagword(P)$. We set $\vert \rho_i\vert$ to be the number of elements in $\rho_i$. A run $\rho_i$ will be referred to as \emph{positive, zero} or \emph{negative}, referring to the sign of its index $i$.
\end{definition}

\begin{remark} \label{rem:zeros}
	Notice that a run of a diagonal word never consists of only zeros. 
\end{remark}

\begin{definition}
	Let $P\in \Pref(m,n)$ be a preference function with shift $s$ and diagonal word $\tau=\diagword(P)= \rho_l\rho_{l-1}\cdots \rho_0 \rho_{-1} \cdots \rho_{-s}$, where the $\rho_i$'s are the runs of $\tau$. Let $c$ be a nonzero element of $\tau$. We define the \emph{schedule number} $w^s(c)$ of $c$ as follows:
	\begin{enumerate}
		\item if $c$ is in a positive run, then $w^s(c)$ equals the number of elements bigger than $c$ in its run plus the number of elements smaller than $c$ in the next run;
		\item if $c$ is the zero run, then $w^s(c)$ equals $1$ plus the number of elements bigger than $c$ in its own run; 
		\item if $c$ is in a negative run, then $w^s(c)$ equals the number of elements smaller that $c$ in its run plus the number of elements bigger that $c$ in the previous run.
	\end{enumerate} Note that these schedule numbers do depend on the shift $s$ to determine which is the zero run. 
\end{definition}

\begin{definition}
	For a $P\in \Pref(m,n)$ whose diagonal word is $\tau$ we define its \emph{reduced diagonal word}, $\rdiagword(P)$ or $\tilde \tau$ to be the word obtained from $\tau$ by deleting all the zeros.  	
\end{definition}

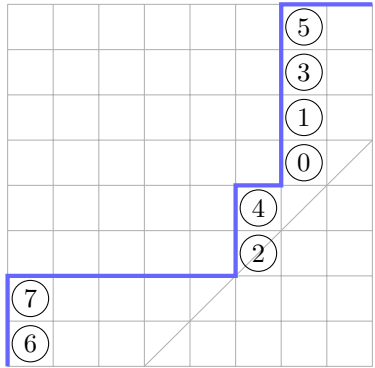
\begin{figure}[!ht]
	\centering
	\begin{tikzpicture}[scale = 0.6]
	\draw[step=1.0, gray!60, thin] (0,0) grid (8,8);
	
	\draw[gray!60, thin] (3,0) -- (8,5);
	
	\draw[blue!60, line width=1.6pt] (0,0) -- (0,2) -- (4,2) -- (4,2) -- (5,2) -- (5,3) -- (5,4) -- (6,4) -- (6,5) -- (6,6) -- (6,7) -- (6,7) -- (6,8) -- (8,8);
	
	\node at (0.5,0.5) {$6$};
	\draw (0.5,0.5) circle (.4cm); 
	\node at (.5,1.5) {$7$};
	\draw (.5,1.5) circle (.4cm); 
	\node at (5.5,2.5) {$2$};
	\draw (5.5,2.5) circle (.4cm); 
	\node at (5.5,3.5) {$4$};
	\draw (5.5,3.5) circle (.4cm); 
	\node at (6.5,4.5) {$0$};
	\draw (6.5,4.5) circle (.4cm); 
	\node at (6.5,5.5) {$1$};
	\draw (6.5,5.5) circle (.4cm); 
	\node at (6.5,6.5) {$3$};
	\draw (6.5,6.5) circle (.4cm); 
	\node at (6.5,7.5) {$5$};
	\draw (6.5,7.5) circle (.4cm);
	\end{tikzpicture}
	\caption{Example of an element in $\Pref(1,7)$.}
	\label{fig: labelled pref}
\end{figure}

\begin{example}
	Consider the path $P$ represented in Figure~\ref{fig: labelled pref}. Its shift is $3$ and its diagonal word is $57\;36\;1\;04\;2$. There is one positive run which is $57$, the zero run is $36$ and there are three negative runs: $1$, $04$ and $2$.  The schedule numbers are 
	\begin{multicols}{5}
		\begin{center}
			\noindent
			$w^3(5)=2 $\\ 
			$w^3(7)=2 $\\
			$w^3(3)=2 $\\
			$w^3(6)=1 $\\
			$w^3(1)=2 $\\
			\columnbreak
			$w^3(4)=1 $\\	
			\columnbreak
			$w^3(2)=1 $\\	
		\end{center}
	\end{multicols}
	Its reduced diagonal word is $57\;36\;1\;4\;2$

\end{example}

The following theorem extends \cite[Theorem~2.5]{Leven-2016}.

\begin{theorem}\label{prop: zeroformula} Let $S\subseteq \Pref(m,n)$ be the set of preference functions $P$ such that 
	\begin{align*}
	\shift(P)=s \qquad \text{ and }\qquad  \diagword(P)=\tau= \rho_l \cdots \rho_0 \cdots \rho_{-s},
	\end{align*} where the $\rho_i$'s are the runs of $\tau$. Let $\rdiagword(P)=\tilde \tau =  \pi_l \cdots \pi_0 \cdots \pi_{-s}$ where we obtain $\pi_i$ by deleting the zeros from $\rho_i$. Set $r_i=\vert\rho_i\vert$ and $p_i=\vert \pi_i \vert$. Then 
	\begin{align*}
	\sum_{P\in S} q^{\dinv(P)} t^{\area(P)} = 
	t^{\maj(\tau)} q^{p_{-1}+\cdots +p_{-s}}
	&\left( \prod_{c\in \tilde \tau} [w^s(c)]_q \right)
	\left( \prod_{i=0}^l \qbinom{r_i -1}{r_i-p_i}_q \right)\times \\
	&\times \left( \prod_{i=-s}^{-1} \qbinom{p_{i+1}+r_i-p_i-1}{r_i-p_i}_q \right).
	\end{align*}
\end{theorem}
\begin{remark}
	Notice that the $\pi_i$ are not necessarily the runs of $\tilde \tau$. 
\end{remark}
\begin{proof} First of all, observe that the area of a path in $S$ is given by $\maj(\tau)$. Indeed, its area is clearly given by \[0\cdot  r_{-s} +1 \cdot r_{-s+1} + \cdots +(l+s)\cdot  r_l . \] 
	
	Let us now consider the dinv. We obtain the formula by constructing every possible path in $S$, while keeping track of the dinv. See Figure \ref{fig: tree} for an example. 
	
	We start with an empty path. For $i=0,\dots, l$ we do the following.
	\begin{itemize}
		\item  Reading $\rho_i$ from \emph{right to left} , we insert first its nonzero elements one by one, into the diagonal $y=x+i$ of the grid, in a way that each step defines a parking function. For each of these elements $c\neq 0$ there are exactly $w^s(c)$ ways to do this. Furthermore, the dinv added to the resulting path by each of these choices gives all the values from $0$ to $w^s(c)-1$, hence the factor $[w^s(c)]_q$. 
		\item Next, we insert the $r_i-p_i$ zeros of $\rho_i$, also into the diagonal $y=x+i$. A zero label can occur directly after another label in its diagonal, provided that the first label in the diagonal is not zero. Each time a nonzero label precedes a zero label, one unit of dinv is created: indeed, since we are always inserting into the highest diagonal, the zeros only create primary dinv. It follows that the dinv that is created is $q$-counted by the factor $ \qbinom{r_i-1}{r_i-p_i}_q $.
	\end{itemize}
	Then, for $i=-1,-2,\dots, -s$ we proceed as follows. 
	\begin{itemize}
		\item We insert $r_i-p_i$ zeros into the diagonal $y=x+i$, so as to obtain a labelled square path. We must insert every such zero label directly underneath a nonzero label of the diagonal $y=x+i+1$, of which there are $p_{i+1}$, or directly before another zero. To give a preference function, the last zero in the diagonal must always be of the first kind. Each time a nonzero label in the diagonal $y=x+i+1$ precedes a zero label in the diagonal $y=x+i$, one unit of secondary dinv is created. This explains the factor $\qbinom{p_{i+1}+r_i-p_i-1}{r_i-p_i}_q$. 
		\item Next, we insert the nonzero labels of $\rho_i$, one by one, from \emph{left to right}. It is not hard to see that for such a $c\neq 0$ there are exactly $w^s(c)$ ways to insert it, and the dinv of the different options is $q$-counted by $[w^s(c)]_q$. 
	\end{itemize}
	
	Finally, the factor $q^{p_{-1}+\cdots +p_{-s}}$ accounts for the bonus dinv, i.e.\ for the number of nonzero labels in negative runs. 
\end{proof}

\begin{figure}
	\begin{center}
		\includegraphics[scale=.9]{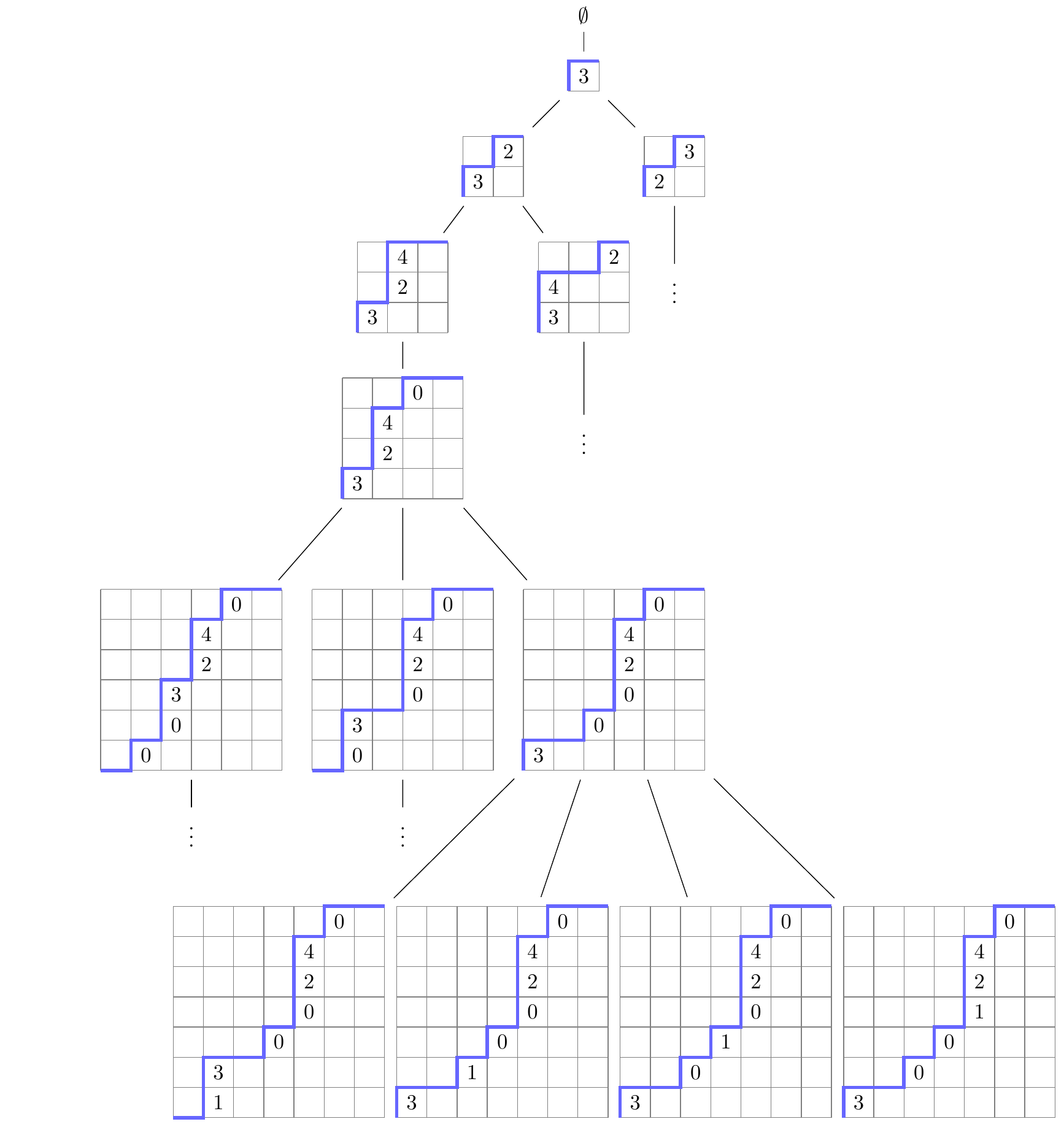}
	\end{center}
	\caption{Partial tree of construction of paths with diagonal word $0423001$ and shift $1$.}
	\label{fig: tree} 
\end{figure}

\begin{definition} 
	Set $\mathfrak S_{m,n}$ to be the set of permutations of $\mathcal O \cup \{1,\dots, n\}$ where $\mathcal O$ is the multiset containing $m$ zeros. 
\end{definition}
\begin{proposition}\label{prop: dycktosquare-k=0}
	Consider $\tau\in \mathfrak S_{m,n}$. Let $\tau=\rho_l\dots \rho_0$ be its runs and  set $ \tilde \tau =\pi_l \cdots \pi_0\in \mathfrak{S}_n$, where $\pi_i$ is obtained from $\rho_i$ by deleting its zeros. Set $p_i=\vert \pi_i \vert$. We have
	\begin{align*}
	\frac{[p_s]_q}{[p_0]_q} q^{p_{s-1}+\cdots +p_{0}}\sum_{\substack{D\in \Park(m,n)\\ \diagword(D)=\tau}} q^{\dinv(D)}t^{\area(D)}= 
	\sum_{\substack{ P\in \Pref(m,n) \\ \shift(P)=s \\ \diagword(P)=\tau}} q^{\dinv(P)}t^{\area(P)}
	\end{align*}
\end{proposition}
\begin{proof}
	In fact we will show that 
	
	\begin{align} \label{eq:shift_id}
	\frac{[p_s]_q}{[p_{s-1}]_q} q^{p_{s-1}}\sum_{\substack{P\in \Pref(m,n)\\ \shift(P)=s-1 \\ \diagword(P)=\tau}} q^{\dinv(P)}t^{\area(P)}= 
	\sum_{\substack{ P\in \Pref(m,n) \\  \shift(P)=s\\ \diagword(P)=\tau}} q^{\dinv(P)}t^{\area(P)}.
	\end{align} 
	This easily implies the thesis: just divide by the sum on the left hand side, and multiply these identities starting from $s= 1$, in order to get the identity in the statement. 
	
	We will make use of Theorem~\ref{prop: zeroformula}.
	
	Notice that, with the exception of $\rho_s$ and $\rho_{s-1}$, the sign of any other run $\rho_i$ (i.e.\ it being positive, negative or zero) is the same for shift $s$ and shift $s-1$. It follows that, after replacing in \eqref{eq:shift_id} the corresponding formulae from Theorem~\ref{prop: zeroformula} (notice  the difference in the numbering of the runs in the statement), the terms concerning these runs are the same for the right hand side and left hand side. Therefore, after the obvious cancellations, we are left to prove that 
	\begin{align*}
	\frac{[p_s]_q}{[p_{s-1}]_q} q^{p_{s-1}}
	\left(	\prod_{c\in \pi_{s}\cup\pi_{s-1}}[w^{s-1}(c)]_q \qbinom{r_{s-1}-1}{r_{s-1}-p_{s-1}}_q\qbinom{r_s-1}{r_s-p_s}_qq^{p_{s-2}+\cdots +p_0}  \right) \\
	= \prod_{c\in \pi_{s}\cup\pi_{s-1}}[w^{s}(c)]_q \qbinom{r_s-1}{r_s-p_s}_q \qbinom{p+r_{s-1}-p_{s-1}-1}{r_{s-1}-p_{s-1}}_q q^{p_{s-1}+\cdots +p_0} ,
	\end{align*}
	which is equivalent, after further cancellations, to
	\begin{align*}
	\frac{[p_s]_q}{[p_{s-1}]_q}
	\left(	\prod_{c\in \pi_{s}\cup\pi_{s-1}}[w^{s-1}(c)]_q \qbinom{r_{s-1}-1}{r_{s-1}-p_{s-1}}_q\right)
	= \prod_{c\in \pi_{s}\cup\pi_{s-1}}[w^{s}(c)]_q  \qbinom{p_s+r_{s-1}-p_{s-1}-1}{r_{s-1}-p_{s-1}}_q .
	\end{align*}
	
	By the definition of the schedule numbers we know that 
	\begin{align*}
	\prod_{c\in \pi_{s-1}}[w^{s-1}(c)]_q = [p_{s-1}]_q! \qquad \text{ and }\qquad   \prod_{c\in \pi_{s}}[w^{s}(c)]_q = [p_{s}]_q!.
	\end{align*}
	So the above condition reduces to showing that 
	\begin{align*}
	\frac{[p_s]_q}{[p_{s-1}]_q} \prod_{c\in \rho_s}[w^{s-1}(c)]_q[p_{s-1}]_q! \frac{[r_{s-1}-1]_q!}{[p_{s-1}-1]_q!}
	&= \prod_{c\in \rho_{s-1}}[w^{s}(c)]_q[p_{s}]_q! \frac{[p_s+r_{s-1}-p_{s-1}-1]_q!}{[p_s-1]_q!} ,
	\end{align*}
	which is equivalent, after trivial cancellations, to
	\begin{align*}
	\prod_{c\in \rho_s}[w^{s-1}(c)]_q[r_{s-1}-1]_q!
	&= \prod_{c\in \rho_{s-1}}[w^{s}(c)]_q[p_s+r_{s-1}-p_{s-1}-1]_q!\, . 
	\end{align*}
	We will prove this identity in Lemma \ref{lem: one-deviation}, concluding the proof of this proposition.
\end{proof}

\begin{lemma}\label{lem: one-deviation}
	Let $\tau= \alpha\beta$ and $\tilde \tau=\tilde \alpha \tilde \beta$ be the diagonal word and reduced diagonal word of a given preference function, respectively, where $\alpha$ and $\beta$ are the runs of $\tau$ and $\tilde\alpha$ and $\tilde \beta$ are obtained by deleting the zeros of $\alpha$ and $\beta$, respectively. Set $a=\vert \alpha \vert$, $b= \vert \beta \vert$, $\tilde a = \vert \tilde \alpha \vert$ and $\tilde b = \vert \tilde \beta \vert$. Then
	\begin{align*}
	[b-1]_q!\prod_{c\in \tilde \alpha}[w^{0}(c)]_q
	&= [\tilde a+b-\tilde b-1]_q!\prod_{c\in \tilde \beta}[w^{1}(c)]_q
	\end{align*}
\end{lemma}
\begin{proof}
	The general argument is better understood with the help of a specific example. 
	
	Consider $\tau=01468 \,  00023579$, so that
	\begin{align*}
	&\alpha= 01468 && \tilde \alpha= 1468 && a=5 && \tilde a = 4\\
	&\beta= 00023579 &&\tilde \beta= 23579 && b=8 && \tilde b = 5.
	\end{align*}
	Let us define a partition $\lambda$ by setting 
	\[\lambda_i= \vert \{ c\in \tilde \beta \mid c< \text{the $(\tilde a+1-i)$-th element of  } \tilde \alpha \}\vert.\]
	In Figure~\ref{fig: lambda} we construct the Young  diagram of $\lambda$ as follows: draw a $b\times a$ grid. Label its rows, bottom to top with the elements of $\tilde \alpha$ and its columns, left to right with the elements of $\tilde \beta$. Then $\lambda_i$ is the number of cells in the $i$-th row from the top such that the label of its column is smaller then the label of its row. We coloured all such cells blue.

	It is now clear from Figure~\ref{fig: lambda} that the conjugate partition $\lambda'$ of $\lambda$  is such that 
	\[\lambda'_i= \vert \{ c\in \tilde \alpha \mid c> \text{the $i$-th element of  } \tilde \beta \}\vert.\]
	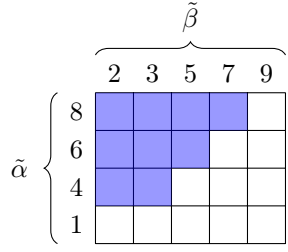
\begin{figure*}
		\centering
		\begin{tikzpicture}[scale=.25,rotate=0]
		\draw (0,0) grid[step=2] (10,8);
		\node at (-1,1){1};
		\node at (-1,3){4};
		\node at (-1,5){6};
		\node at (-1,7){8};
		\node at (1,9){2};
		\node at (3,9){3};
		\node at (5,9){5};
		\node at (7,9){7};
		\node at (9,9){9};
		%\node at (5,11){$\tilde \beta$};
		%\draw (0,10)-- (10,10);
		\draw [decorate,decoration={brace,amplitude=5pt}] (0,10)--(10,10) node [midway,yshift=.5cm]{$\tilde \beta$} ;
		\draw [decorate,decoration={brace,amplitude=5pt}] (-2,0)--(-2,8) node [midway,xshift=-.5cm]{$\tilde \alpha$} ;
		
		\fill[blue, opacity=.4](0,2)-|(4,4)-|(6,6)-|(8,8)-|(0,2);
		\end{tikzpicture}
		\caption{Construction of the Young diagram of $\lambda$.}
		\label{fig: lambda}
	\end{figure*}	
	These partitions are useful because they encode essential information about the schedule numbers of elements that appear in the thesis. Recall that $\alpha$ is a positive run for shift $0$ and so $w^0(c)$ equals the number of elements bigger than $c$ in $\alpha$ plus the number of elements smaller than $c$ in $\beta$. Similarly, $\beta$ is a negative run for shift $1$, so $w^1(c)$ equals the number of elements smaller than $c$ in $\beta$ plus the number of elements bigger than $c$ in $\alpha$.

	So if $c$ is the $(\tilde a +1-i)$-th element of $\tilde \alpha$, we must have \[ w^0(c)= (i-1)+\lambda_i  + (b-\tilde b)\] where the first term is the number of elements of $\alpha$ that are bigger than $c$, the second term accounts for the number of elements in $\beta$ that are smaller then $c$ and different from $0$ and the last term counts the number of zeros in the next run (that must necessarily be smaller then $c\neq 0$). 
	
	Similarly, if $c$ is the $i$-th element of $\tilde \beta$, we have \[w^1(c)= (i-1)+ (b-\tilde b) +\lambda'_i\] where the first term is the number of elements in $\beta$ smaller than $c$ and different from $0$, the second term the number of zeros in $\beta$ (which are smaller than $c$) and the third term the number of elements in $\alpha$ bigger than $c$.

	So for our running example, the schedule numbers are computed in Figure~\ref{fig:schedule_numbers}.
	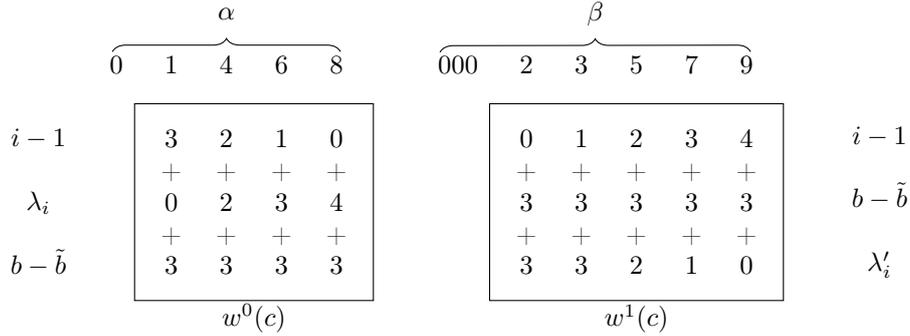
\begin{figure}[!h]
		\begin{center}
			\begin{tikzpicture}[scale=.5]
			\node (1) {0};
			\node (2) [right=.3cm of 1] {1};
			\node (3) [right=.3cm of 2] {4};
			\node (4) [right=.3cm of 3] {6};
			\node (5) [right=.3cm of 4] {8};
			\path (1) ++(135:.5cm) node (dec1) {};
			\path (5) ++(45:.5cm) node (dec2) {};
			\draw [decorate,decoration={brace,amplitude=5pt}] (dec1)--(dec2) node [midway, yshift=.5cm]{$\alpha$};
			\node (6) [right=1cm of 5] {000};
			\node (7) [right=.3cm of 6] {2};
			\node (8) [right=.3cm of 7] {3};
			\node (9) [right=.3cm of 8] {5};
			\node (10) [right=.3cm of 9] {7};
			\node (11) [right=.3cm of 10] {9};
			\path (6) ++(155:.8cm) node (dec3) {};
			\path (11) ++(45:.5cm) node (dec4) {};
			\draw [decorate,decoration={brace,amplitude=5pt}] (dec3)--(dec4) node [midway, yshift=.5cm]{$\beta$};
			\node (A1) [below=.5cm of 2, align =center] {3\\+\\0\\+\\3};
			\node (A2) [below=.5cm of 3, align =center] {2\\+\\2\\+\\3};
			\node (A3) [below=.5cm of 4,align =center] {1\\+\\3\\+\\3};
			\node (A4) [below=.5cm of 5,align =center] {0\\+\\4\\+\\3};
			\node (een) [below left=.1cm and .1 cm of A1]{};
			\node (twee) [above right=.1cm and .1 cm of A4]{};
			\draw (een) rectangle (twee);
			
			\node (lab1) [left=1cm of A1, align =center] {$i-1$\\ \\$\lambda_i$ \\ \\ $b-\tilde b$};
			
			\node (aux) at ($(A2)!0.5!(A3)$) {};
			\node (name) [below=1.1cm of aux] {$w^0(c)$};
			\node (B1) [below=.5cm of 7, align =center] {0\\+\\3\\+\\3};
			\node (B2) [below=.5cm of 8, align =center] {1\\+\\3\\+\\3};
			\node (B3) [below=.5cm of 9, align =center] {2\\+\\3\\+\\2};
			\node (B4) [below=.5cm of 10, align =center] {3\\+\\3\\+\\1};
			\node (B5) [below=.5cm of 11, align =center] {4\\+\\3\\+\\0};
			\node (een) [below left=.1cm and .1 cm of B1]{};
			\node (twee) [above right=.1cm and .1 cm of B5]{};
			\draw (een) rectangle (twee);
			\node (name) [below=.15 of B3] {$w^1(c)$};
			
			\node (lab2) [right=1cm of B5, align =center] {$i-1$\\ \\$b-\tilde b$ \\ \\ $\lambda'_i$};
			\end{tikzpicture}
		\end{center}
		\caption{Schedule numbers $w^0(c)$ for $c\in \tilde \alpha$ and $w^1(c)$ for $c\in \tilde \beta$. }
		\label{fig:schedule_numbers}
	\end{figure}
	
	Using this decomposition of the schedule numbers, we can write
	\begin{align}
	&[b-1]_q!\prod_{c\in \tilde \alpha}[w^{0}(c)]_q=[b-1]_q! \prod_{i=1}^{\tilde a} [\lambda_i +(i-1)+b-\tilde b ]_q  \label{eq: rows} \\
	&[\tilde a+b-\tilde b-1]_q!\prod_{c\in \tilde \beta}[w^{1}(c)]_q=
	[\tilde a+b-\tilde b-1]_q!\prod_{i=1}^{\tilde b}[\lambda'_i+(i-1)+b-\tilde b ]_q. \label{eq: columns}
	\end{align} 
	The fact that these two equations are equal turns out to be a consequence of a general fact about partitions.

	Consider the Ferrers diagram of $\delta_{b+\tilde a -1}\coloneqq (b+\tilde a -1, b+\tilde a -2, \dots, 2,1)$ whose parts are justified to the right (see Figure~\ref{fig:genshuffle_argument}). Next, delete all the cells in the bottom right $\tilde b\times \tilde a $ rectangle that are not elements of $\lambda$, when $\lambda$ is placed in the top left corner of this rectangle: see Figure~\ref{fig:genshuffle_argument}. Call the resulting skew diagram $\Gamma$. Next, label the bottom $\tilde a$ rows with the elements of $\tilde \alpha$, starting from the bottom; and the $\tilde b$ rightmost columns with the elements of $\tilde \beta$, starting from the left. 
	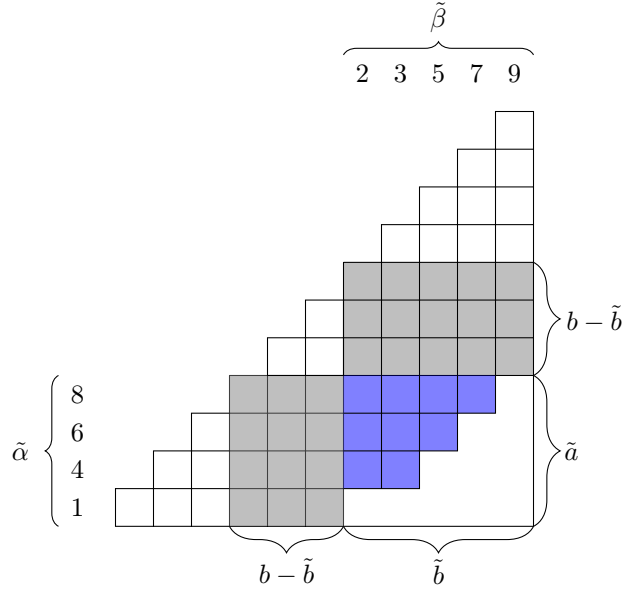
\begin{figure}[!ht]
		\begin{center}
			\begin{tikzpicture}[scale=.5,rotate=0]
			%\draw[gray] (0,0) grid (5,4);
			\draw(0,0) grid (-3,4);
			\fill[gray, opacity=.5] (0,0) rectangle (-3,4);
			\fill[gray, opacity=.5] (0,4) rectangle (5,7);
			\draw (0,4) grid (5,7);
			\fill[blue, opacity=.5](0,1)-|(2,2)-|(3,3)-|(4,4)-|(0,1);
			\draw (1,7)|-(5,8) (2,7)|-(5,9) (3,7) |-(5,10) (4,7)|-(5,11)--(5,7); 
			\draw (0,1)-|(2,2)-|(3,3)-|(4,4) (1,1)--(1,4) (2,2)--(2,4) (3,3)--(3,4) (0,2)--(2,2) (0,3)--(3,3);
			\draw (-3,3) -|(-4,0) (-3,2)-|(-5,0) (-3,1)-|(-6,0)--(-3,0);
			\draw (-2,4)|-(-1,5)|-(0,6)|-(-2,4) (-1,4)|-(0,5);
			\draw (0,0)-|(5,4);
			\draw [decorate,decoration={brace,amplitude=10pt,mirror}] (0,0)--(5,0) node [midway,yshift=-.6cm]{$\tilde b$} ;
			\draw [decorate,decoration={brace,amplitude=10pt,mirror}] (-3,0)--(0,0) node [midway,yshift=-.6cm]{$b- \tilde b$} ;
			\draw [decorate,decoration={brace,amplitude=10pt,mirror}] (5,0)--(5,4) node [midway,xshift=.5cm]{$\tilde a$} ;
			\draw [decorate,decoration={brace,amplitude=10pt,mirror}] (5,4)--(5,7) node [midway,xshift=.8cm]{$ b-\tilde b$} ;
			\node at (-7,.5) {1};
			\node at (-7,1.5) {4};
			\node at (-7,2.5) {6};
			\node at (-7,3.5) {8};
			\draw [decorate,decoration={brace,amplitude=5pt}] (-7.5,0)--(-7.5,4) node [midway,xshift=-.5cm]{$\tilde \alpha$} ;
			\node at (.5,12) {2};
			\node at (1.5,12) {3};
			\node at (2.5,12) {5};
			\node at (3.5,12) {7};
			\node at (4.5,12) {9}; 
			\draw [decorate,decoration={brace,amplitude=5pt}] (0,12.5)--(5,12.5) node [midway,yshift=.5cm]{$\tilde \beta$} ;
			\end{tikzpicture}
			\caption{Construction and interpretation of $\Gamma$. }
			\label{fig:genshuffle_argument}
		\end{center}
	\end{figure}
	\begin{itemize}
		\item \emph{The rows of $\Gamma$.} It follows from the discussion above and the construction of $\Gamma$, that if $1\leq i\leq \tilde a$, and $c$ is the label of the $i$-th row of $\Gamma$, then the number of cells in this $i$-th row equals $w^0(c)$. The remaining $(b-\tilde{b})+(\tilde{b}-1)=b-1$ rows form a staircase. 
		\item \emph{The columns of $\Gamma$.} Similarly, if $1\leq i\leq \tilde b$, and $c$ is the label of the $i$-th row from the right of $\Gamma$, then the number of cells in this row equals $w^1(c)$.
		The remaining $(b-\tilde b) +(\tilde a -1)$ columns form a staircase. 
	\end{itemize} It follows that taking the product of the $q$-analogues of the multiset recording the length of the rows (respectively columns) yields (\ref{eq: rows}) (respectively (\ref{eq: columns})). 
	
	To conclude, we only need the following general observation.
	\begin{lemma}
		Let $\delta=(n,n-1,\dots,2,1)$ be the staircase partition, and let $\lambda \subseteq \delta$. Then the multiset of the lengths of the rows of $\delta/\lambda$ equals the multiset of the lengths of its columns.
	\end{lemma}
	\begin{proof}[Proof of the Lemma]
		It follows immediately from the observation that any internal corner of a skew shape of the form $\delta/\mu$  with $\mu\subseteq \delta$ (i.e.\ a cell in $\delta/\mu$ adjacent to $\mu$ that once removed leaves a skew shape) belongs to exactly one row and exactly one column of $\delta/\mu$ and these must have the same length. Then removing an internal corner does not change the equality between the multisets that we are considering, so that we can remove the cells from $\lambda\subseteq \delta$ one at the time to get $\delta/\lambda$. 
	\end{proof}
	
	This concludes the proof of the theorem.
\end{proof}

\begin{corollary}
	\label{cor: dycktosquare-k=O}
	Take $\tau\in \mathfrak S_{m,n}$. Let $r$ be the number of nonzero elements in the rightmost run of $\tau$. Then 
	\begin{align*}
	\sum_{\substack{P\in \Pref(m,n) \\ \diagword(P)=\tau }} 
	q^{\dinv(P)}t^{\area(P)}
	= 
	\frac{[n]_q}{[r]_q} 
	\sum_{\substack{D\in \Park(m,n) \\ \diagword(D)=\tau} }
	q^{\dinv(D)}t^{\area(D)}.
	\end{align*}
\end{corollary}
\begin{proof}
	Let $\tau=\rho_l\dots \rho_0$ with the $\rho_i$'s its runs and $\tilde \tau=\pi_l \cdots \pi_0$, where $\pi_i$ is obtained from $\rho_i$ by deleting its zeros. Set $p_i=\vert \pi_i \vert$. Applying Proposition~\ref{prop: dycktosquare-k=0}, we get
	\begin{align*}
	\sum_{\substack{P\in \Pref(m,n)\\ \diagword(P)=\tau }} 
	q^{\dinv(P)}t^{\area(P)}
	&= \sum_{s=0}^{l} \sum_{\substack{P\in \Pref(m,n)\\ \diagword(P)=\tau \\ \shift(P)=s }} 
	q^{\dinv(P)}t^{\area(P)} \\
	&=\sum_{s=0}^l \frac{[p_s]_q}{[p_0]_q} q^{p_{s-1}+\cdots +p_{0}}\sum_{\substack{D\in \Park(m,n)\\ \diagword(D)=\tau}} q^{\dinv(D)}t^{\area(D)}. 
	\end{align*}
	
	But
	\begin{align*}
	\sum_{s=0}^l \frac{[p_s]_q}{[p_0]_q} q^{p_{0}+\cdots +p_{s-1}} 
	&= \frac{1}{[p_0]_q} \left( [p_0]_q + [p_1]_q q^{p_0}+\cdots + [p_l]_q q^{p_0+\cdots+ p_{l-1}} \right)\\
	&=\frac{[p_0+\cdots +p_l]_q}{[p_0]_q}=\frac{[n]_q}{[r]_q}
	\end{align*}
	concluding the proof.
\end{proof}

This result gives a link between the $q,t$-enumerators of preference functions and parking functions. Next, we need to deal with the Gessel quasisymmetric functions in \eqref{eq:touchgensquare}. 

In her thesis \cite[Corollary~73]{Hicks-Thesis}, Hicks found a way to factor the $q,t$-enumerator \[\sum_{P\in S} q^{\dinv(P)}t^{\area(P)}\] out of the expression \[ \sum_{P\in S} q^{\dinv(P)}t^{\area(P)} Q_{\ides(P),n},\] where $S$ is the set of parking functions of size $n$ with a fixed diagonal word. In \cite[Lemma~4.2]{Leven-2016}, Sergel showed that Hicks' argument generalizes in a straightforward way to the case where $S$ is the set of preference functions of a given diagonal word and shift. We notice here that in fact again the same argument works in the case where $S$ is the set of partially labelled preference functions of a given diagonal word and shift. The proof will be exactly the same as the ones given in the two aforementioned references, so we omit it. 

\begin{definition}
	Consider $\tau \in \mathfrak S_{m,n} $.  A \emph{consecutive block} of $\tau$ is a substring of $\tau$ of the form $i,i+1,\dots, i+k$ with $i\neq 0$. Define $\Yconsec(\tau)$ to be the Young subgroup of $\mathfrak S_n$ which permutes only elements within the same consecutive block of $\tau$. 
\end{definition}
This definition coincides with the one in \cite{Hicks-Thesis} when $m=0$.

\begin{example}
	If $\tau=00412506703$ then $\Yconsec(\tau)= \mathfrak{S}_{\{1,2\}}\times \mathfrak{S}_{\{3\}}\times \mathfrak{S}_{\{4\}}\times \mathfrak{S}_{\{5\}} \times \mathfrak{S}_{\{6,7\}}$. 
\end{example}

\begin{definition}
	If $\tau\in \mathfrak S_{m,n}$ and $\tilde \tau\in \mathfrak S_n$ is obtained from $\tau$ by deleting its zeros, we set $\ides(\tau) \coloneqq \ides (\tilde \tau )$. 
\end{definition}
For an argument to prove the following proposition, see \cite{Hicks-Thesis}*{Corollary~73} or \cite[|lemma~4.2]{Leven-2016}.

\begin{proposition}\label{lem: qsym factors}
	Given $\tau \in \mathfrak S_{m,n}$, if 
	\begin{align*}
	S \coloneqq \left\{ P\in \Pref(m,n) \mid \shift(P)=s, \diagword(P)=\tau \right\}, 
	\end{align*}
	then 
	\begin{align*}
	&\sum_{P\in S}
	q^{\dinv(P)}t^{\area(P)} Q_{\ides(P),n}
	=\left( 
	\sum_{P \in S}
	q^{\dinv(P)}t^{\area(P)} 
	\right) 
	\times 
	\left( 
	\frac{\sum\limits_{\pi\in \Yconsec(\tau)}q^{\inv(\pi)}Q_{\ides( \tau)\cup \ides(\pi),n}}
	{\sum\limits_{\pi\in \Yconsec(\tau)}q^{\inv(\pi)}} 
	\right).
	\end{align*}
\end{proposition}

We combine this proposition with Corollary~\ref{cor: dycktosquare-k=O}.
\begin{corollary} \label{cor:7.29}
	Take $\tau\in \mathfrak S_{m,n}$. Let $r$ be the number of nonzero elements in the rightmost run of $\tau$. Then  
	\begin{align*}
	\sum_{\substack{P\in \Pref(m,n)\\\diagword(P)=\tau }} 
	q^{\dinv(P)}t^{\area(P)}Q_{\ides(P),n}
	= 
	\frac{[n]_q}{[r]_q} 
	\sum_{\substack{D\in \Park(m,n)\\ \diagword(D)=\tau } }
	q^{\dinv(D)}t^{\area(D)}Q_{\ides(D),n}
	\end{align*}
\end{corollary}
\begin{proof}
	Take $l+1$ to be the number of runs of $\tau$. Then 
	
	\begin{align*}
	\sum_{\substack{P\in \Pref(m,n)\\\diagword(P)=\tau }} &
	q^{\dinv(P)}t^{\area(P)}Q_{\ides(P),n}= \\
	& = \sum_{s=0}^l \sum_{\substack{P\in \Pref(m,n)\\ \shift(P)=s \\ \diagword(P)=\tau}} 
	q^{\dinv(P)}t^{\area(P)}Q_{\ides(P),n}   \\
	\text{(by Proposition~\ref{lem: qsym factors})} 
	&=
	\sum_{s=0}^l 
	\left( 
	\sum_{\substack{P\in \Pref(m,n)\\ \shift(P)=s \\ \diagword(P)=\tau }}
	q^{\dinv(P)}t^{\area(P)} 
	\right) 
	\times 
	\left( 
	\frac{\sum\limits_{\pi\in \Yconsec(\tau)}q^{\inv(\pi)}Q_{\ides(\tau)\cup \ides(\pi),n}}
	{\sum\limits_{\pi\in \Yconsec(\tau)}q^{\inv(\pi)}} 
	\right) 
	\end{align*}
	\begin{align*}
	&=
	\left( 
	\sum_{\substack{P\in \Pref(m,n) \\ \diagword(P)=\tau }}
	q^{\dinv(P)}t^{\area(P)} 
	\right) 
	\times 
	\left( 
	\frac{\sum\limits_{\pi\in \Yconsec(\tau)}q^{\inv(\pi)}Q_{\ides(\tau)\cup \ides(\pi),n}}
	{\sum\limits_{\pi\in \Yconsec(\tau)}q^{\inv(\pi)}} 
	\right)  \\
	\text{(by Corollary~\ref{cor: dycktosquare-k=O})}
	&=
	\left( 
	\frac{[n]_q}{[r]_q} 
	\sum_{\substack{D\in \Park(m,n) \\ \diagword(D)=\tau } }
	q^{\dinv(D)}t^{\area(D)}
	\right) 
	\times 
	\left( 
	\frac{\sum\limits_{\pi\in \Yconsec(\tau)}q^{\inv(\pi)}Q_{\ides(\tau)\cup \ides(\pi),n}}
	{\sum\limits_{\pi\in \Yconsec(\tau)}q^{\inv(\pi)}} 
	\right) \\
	\text{(by Proposition~\ref{lem: qsym factors})}
	&= 
	\frac{[n]_q}{[r]_q} 
	\sum_{\substack{D\in \Park(m,n)\\ \diagword(D)=\tau } } q^{\dinv(D)}t^{\area(D)}Q_{\ides(D),n}
	\end{align*}	
\end{proof}

We are now ready to prove the announced result. 
\begin{proof}[Proof of Theorem~\ref{thm: gensquare}]
	Using Theorem~\ref{thm:GenShuffle}, we have
	\begin{align*}
	\frac{[n]_q}{[r]_q} \Delta_{h_m}\nabla E_{n,r} 
	&= \frac{[n]_q}{[r]_q} \sum\limits_{ \substack{P\in \Park(m,n) \\ \touch(P)=r} } q^{\dinv(P)}t^{\area(P)}Q_{\ides(P),n}\\
	&=  
	\sum_{\substack{\tau \in \mathfrak{S}_{m,n}\\ \text{last run of $\tau$ has}\\ \text{$r$ nonzero elements} }} \frac{[n]_q}{[r]_q} 
	\sum\limits_{ \substack{P\in \Park(m,n) \\ \diagword(P)=\tau} } 
	q^{\dinv(P)}t^{\area(P)}Q_{\ides(P),n} \\
	\text{(using Corollary~\ref{cor:7.29})} 
	&=
	\sum_{\substack{\tau \in \mathfrak{S}_{m,n}\\ \text{last run of $\tau$ has}\\ \text{$r$ nonzero elements} }} \;
	\sum\limits_{ \substack{P\in \Pref(m,n) \\ \diagword(P)=\tau} } 
	q^{\dinv(P)}t^{\area(P)}Q_{\ides(P),n} \\
	& = \sum\limits_{ \substack{P\in \Pref(m,n) \\ \touch(P)=r} } 
	q^{\dinv(P)}t^{\area(P)}Q_{\ides(P),n}.
	\end{align*}
\end{proof}

\section{Super-diagonal coinvariants and Theta operators} \label{sec:superdiag}
%\section{Super-diagonal coinvariants and $\Thera_k$}

In this section we extend a conjecture of Zabrocki in \cite{Zabrocki_Delta_Module}. 

\medskip

Let $n,r\in \mathbb{N}$, $n,r\geq 1$, and consider the algebra of polynomials in $2+r$ sets of $n$ variables \[R_{n}^{(r)} \coloneqq \mathbb{C}[x_1,\dots,x_n,y_1,\dots,y_n,\theta_{1}^{(1)},\dots,\theta_{n}^{(1)},\dots,\theta_{1}^{(r)},\dots,\theta_{n}^{(r)}],\]
where the $x_i$ and $y_j$ are commuting variables, while the $\theta_i^{(k)}$ are Grassmannian variables, i.e. $\theta_i^{(k)}\theta_j^{(k)}=-\theta_j^{(k)}\theta_i^{(k)}$ for $1\leq i\neq j\leq n$ and $\theta_i^{(k)}\theta_i^{(k)}=0$ (notice that in $R_{n}^{(r)}$ variables from different sets commute with each others). 

Consider the diagonal action of the symmetric group $\mathfrak{S}_n$, i.e. each element of $\mathfrak{S}_n$ permutes simultaneously the $2+r$ sets of variables acting on their indices, and let $I_n^{(r)}$ be the ideal of $R_{n}^{(r)}$ generated by the homogeneous invariants of positive degree. Following \cite{Zabrocki_Delta_Module}, we call the quotient $M_n^{(r)} \coloneqq R_{n}^{(r)}/I_n^{(r)}$ the space of \emph{super-diagonal coinvariants} (in \cite{Zabrocki_Delta_Module} Zabrocki calls in this way only the case $r=1$).

Clearly $M_n^{(r)}$ is an $\mathfrak{S}_n$-module, naturally $(2+r)$-graded by the degree in the $2+r$ sets of variables. For $a,b\in \mathbb{N}$ and $\alpha \coloneqq (\alpha_1,\dots,\alpha_r)\in \mathbb{N}^r$, denote by $M_n^{(r)}(a,b,\alpha)$ the homogeneous submodule of $M_n^{(r)}$ of multidegree $(a,b,\alpha_1,\dots,\alpha_r)$ in the respective variables $x_i$'s, $y_i$'s, $\theta_i^{(1)}$'s,\dots , $\theta_i^{(r)}$'s.

For a partition $\mu$ of $n$, let $\chi_{M_n^{(r)}(a,b,\alpha)}(\mu)$ be the value of the character of the $\mathfrak{S}_n$-module $M_n^{(r)}(a,b,\alpha)$ on a permutation of cycle type $\mu$. Define the \emph{$q,t,\underline{z}$-Frobenius image} for super-diagonal coinvariants as
\[\mathcal{F}_{q,t,\underline{z}}(M_n^{(r)}) \coloneqq \sum_{a,b\in \mathbb{N}} \sum_{\alpha\in (\mathbb{N})^r}q^at^bz^\alpha \chi_{M_n^{(r)}(a,b,\alpha)}(\mu) \frac{p(\mu)}{z_\mu}\in \Lambda_{\mathbb{Q}(q,t,\underline{z})},\]
where $\underline{z} \coloneqq z_1,\dots,z_r$, $z^\alpha \coloneqq z_1^{\alpha_1}\cdots z_r^{\alpha_r}$, and $z_\mu \coloneqq \prod_{i=1}^{\mu_1}m_i!i^{m_i}$ with $m_i$ equal to the number of parts of size $i$ in $\mu$.

Using Theorem~\ref{thm:DeltakmGD}, Zabrocki's conjecture in \cite{Zabrocki_Delta_Module} can be restaded in the following way.
\begin{conjecture}[Zabrocki]
	For $n\geq 1$
	\begin{equation}
	\mathcal{F}_{q,t,\underline{z}}(M_n^{(1)}) = \sum_{i=0}^{n-1} z_1^i\Theta_{i}\nabla e_{n-i}.
	\end{equation}
\end{conjecture}

Now we state our conjecture for $r=2$.

\begin{conjecture}
	For $n\geq 1$
	\begin{equation}
	\mathcal{F}_{q,t,\underline{z}}(M_n^{(2)}) = \mathop{\sum_{i,j\geq 0}}_{1\leq i+j<n} z_1^iz_2^j\Theta_{i}\Theta_{j}\nabla e_{n-(i+j)}.
	\end{equation}
\end{conjecture}

More generally, we risk the following conjecture.
\begin{conjecture}
	For $n\geq 1$, $r\geq 3$ and $\alpha\in \mathbb{N}^r$ with $\lvert \alpha \rvert \coloneqq \sum_{i=1}^r\alpha_i<n$
	\begin{equation}
	\mathcal{F}_{q,t,\underline{z}}(M_n^{(r)}(\alpha)) = z^\alpha \Theta_{\alpha_1} \cdots \Theta_{\alpha_r} \nabla e_{n- \lvert \alpha \rvert},
	\end{equation}
	where the left hand side is the $q,t,\underline{z}$-Frobenius image of the $\mathfrak{S}_n$-module $M_n^{(r)}(\alpha):=\oplus_{a,b\geq 0}M_n^{(r)}(a,b,\alpha)$.
\end{conjecture}
Notice that this last conjecture does not wholly cover the module $M_n^{(r)}$. Indeed it seems that there are nonzero submodules of total degree $\geq n$ in the $\theta_i^{(k)}$ variables for $r \geq 3$ that are not present for $r=2$: we do not have a formula for those.

\begin{remark}
	It is a classical result that the Frobenius characteristic of the classical harmonics $\mathcal{H}_n$, i.e.\ the coinvariants of the action of $\mathfrak{S}_n$ on $\mathbb{C}[x_1,\dots,x_n]$ is given by
	\[ \mathcal{F}_{q}(\mathcal{H}_n)= \widetilde{H}_{(n)}[X]=\prod_{i=1}^n(1-q^i)h_n\left[\frac{X}{1-q}\right],  \]
	while, thanks to a theorem of Haiman \cite{Haiman-Vanishing-2002}, the Frobenius characteristic of the \emph{diagonal harmonics} $\mathcal{M}_n^{(0)}$, i.e.\ the coinvariants of the diagonal action of $\mathfrak{S}_n$ on $\mathbb{C}[x_1,\dots,x_n,y_1,\dots,y_n]$ is given by
	\[ \mathcal{F}_{q,t}(\mathcal{M}_n^{(0)})= \nabla e_n.  \]
	
	It is probably worth mentioning that Haglund's \cite{Haglund-Schroeder-2004}*{Theorem~2.5} translates into
	\[ \nabla E_{n,r}=t^{n-r}\Theta_{h_{n-r}}\widetilde{H}_{(r)}[X],  \]
	so that
	\[\nabla e_n=\sum_{r=1}^{n} t^{n-r}\Theta_{h_{n-r}}\widetilde{H}_{(r)}[X] ,\]
	and hence
	\[ \mathcal{F}_{q,t}(\mathcal{M}_n^{(0)}) = \sum_{r=1}^{n} t^{n-r}\Theta_{h_{n-r}}\mathcal{F}_{q}(\mathcal{H}_r).\]
	Though the obvious generalization of this formula does not seem to hold, it is at least conceivable that the Theta operators (or possibly some variation/extension of them) will give the universal formula for super-diagonal coinvariants with $r$ sets of $n$ commuting variables and $r'$ sets of $n$ Grassmannian variables conjectured by Bergeron and Zabrocki (see \cite{Bergeron_Multivariate_Coinvariants} for these explicitly computed formulas up to $n=5$).
\end{remark}

\section{A Theta conjecture (?)}

Let $\LD(0,n)^{\ast k, \circ r}$ be the set of partially labelled Dyck paths $P\in \LD(0,n)^{\ast k}$ with $r$ contractible valleys decorated with a $\circ$, where a \emph{contractible valley} (see \cite{Haglund-Remmel-Wilson-2015}) is a vertical step which is either preceded by two horizontal steps or by one horizontal step and the label in its row is bigger than the label in the row immediately below. We define the area of a path $P\in \LD(0,n)^{\ast k, \circ r}$ as the area of the underlying decorated Dyck path in $\D(n)^{\ast k}$, and its monomial $x^P$ in the usual way  (i.e.\ the decorations on the valleys do not affect the area nor the monomial).

Experimentally, we observed the following formula.
\begin{conjecture} 
	For $n,k,r\in \mathbb{N}$ with $n\geq 1$ and $r+k<n$,
	\begin{equation}
	\left.\Theta_r\Theta_k \nabla e_{n-r-k}\right|_{q=1}  = \sum_{P\in \LD(0,n)^{\ast k, \circ r}}t^{\area(P)}x^P.
	\end{equation}
\end{conjecture}

It is a very interesting problem to find a $\dinv$ statistic on $\LD(0,n)^{\ast k, \circ r}$ which would give the full $\Theta_r\Theta_k \nabla e_{n-r-k}$. Notice that this might be a statistic that ``interpolates'' the \emph{rise version} of the Delta conjecture (i.e.\ the one that we stated in this article) and its \emph{valley version}: see \cite{Haglund-Remmel-Wilson-2015} for details. Such a conjecture, which should better be called \emph{Theta conjecture}, would bring the whole framework of the Delta conjecture to this new more general setting.

\section{More on Theta operators} \label{sec:conj_Delk0}
%\section{More on $\Del{k}{0}$} \label{sec:conj_Delk0}

\subsection{Relation to the $D_k$ operators}

Recall the definition of the operators $D_k$ from \cite{Bergeron-Garsia-Haiman-Tesler-Positivity-1999}: for $k\in \mathbb{Z}$ and any $F[X]\in \Lambda$
\begin{equation}
D_kF[X] \coloneqq \left.\left( F\left[X+\frac{M}{z}\right]\sum_{r\geq 0}(-z)^re_r[X]\right)\right|_{z^k}.
\end{equation}
Recall the following well-known identities \cite[Equation~I.12 (i) and (ii)]{Bergeron-Garsia-Haiman-Tesler-Positivity-1999}
\begin{align}
\label{eq:D0}	D_0\widetilde{H}_\mu[X] & =-D_\mu \widetilde{H}_\mu[X]\\
\label{eq:[Dk_e1]}	D_k\underline{e}_1- \underline{e}_1 D_k & = MD_{k+1}
\end{align}
where $\underline{e}_1$ denotes the multiplication by $e_1$.

The next proposition is proved in Section~\ref{sec:proofs_sec_Delkm}.

\begin{proposition} \label{prop:propDel10}
	We have
	\begin{equation} \label{eq:propDel10}
	\Theta_1=\frac{1}{M}(\underline{e}_1+D_1).
	\end{equation}
\end{proposition}

\begin{remark}  \label{rem:e1D1}
	It has been shown in \cite{Garsia-Haiman-Tesler-Explicit-1999} that $\Lambda$ is spanned by the symmetric functions obtained by applying $D_1$ and $e_1$ to the constant $1$.
\end{remark}

We use the standard commutator notation for associative algebras, i.e. $[a,b] \coloneqq ab-ba$.

\begin{conjecture}
	For any $k\geq 0$,
	\begin{equation} \label{eq:conjDelk0_e1}
	[\Theta_k,\underline{e}_1]=\sum_{i=1}^k (-1)^{i+1}D_{i+1} \Theta_{k-i}
	\end{equation}
	and
	\begin{equation} \label{eq:conjDelk0_D1}
	[\Theta_k,D_1]=-\sum_{i=1}^k (-1)^{i+1}D_{i+1} \Theta_{k-i}.
	\end{equation}
\end{conjecture}
\begin{remark}
	It is clear from their definition that the operators $\Theta_k$ commute with each others, i.e. for all $a,b\geq 0$ we have $[\Theta_a,\Theta_b]=0$. Using \eqref{eq:propDel10}, it is now clear that the identities \eqref{eq:conjDelk0_e1} and \eqref{eq:conjDelk0_D1} are equivalent.
\end{remark}
From Remark~\ref{rem:e1D1} it follows that the relations \eqref{eq:conjDelk0_e1} and \eqref{eq:conjDelk0_D1} determine the operators $\Theta_k$ uniquely.

Finally, notice that in terms of the operators of Carlsson and Mellit (cf. \cite{Carlsson-Mellit-ShuffleConj-2015}), the identity \eqref{eq:propDel10} translates in
\begin{equation} \label{eq:conjDel10CM}
\Theta_1=\frac{1}{M}d_-(d_+-d_+^*)\qquad \text{ on }V_0=\Lambda.
\end{equation}
It would be interesting to see if the $\Theta_k$'s have a similar formula, i.e.\ a formula in terms of the operators of the Dyck path algebra. This could be a step towards a proof of the operator Delta conjecture (and hence of the compositional Delta conjecture).

\subsection{Schur positivity conjectures}

As it is natural to do in these situations, after we discovered the Theta operators we looked for Schur positivity. Sure enough, we found out experimentally that anything that comes naturally from the nabla operator and is Schur positive remains Schur positive after applying $\Theta_{s_\lambda}$ for any $\lambda$.

With our limited experimental evidence, we risk the following conjecture.
\begin{conjecture}
	For $n\in \mathbb{N}$, $n\geq 1$, $\mu \vdash n$, $\alpha\vDash n$ and any partition $\lambda$ we have
	\begin{align*}
	(-1)^{|\mu|-\ell(\mu)}\big{\<} \Theta_{s_\lambda}\nabla m_\mu , s_\nu \big{\>}  & \in \mathbb{N}[q,t]\quad \text{ for every }\nu\vdash n+|\lambda | ,\\ 
	(-1)^{\mathrm{spin}(\mu)}\big{\<} \Theta_{s_\lambda}\nabla s_\mu , s_\nu \big{\>}  & \in \mathbb{N}[q,t]\quad \text{ for every }\nu\vdash n+|\lambda | ,\\ 
	\big{\<} \Theta_{s_\lambda}\nabla C_\alpha , s_\nu \big{\>}  & \in \mathbb{N}[q,t]\quad \text{ for every }\nu\vdash n+|\lambda | ,
	\end{align*}
	where $\mathrm{spin}(\mu)$ is a nonnegative integer that depends on $\mu$, but not on $\lambda$ nor on $\nu$ (cf. \cite[Appendix~B]{Haglund-Book-2008}).
\end{conjecture}

\section{Technical proofs} \label{sec:proofs_sec_Delkm}

\subsection{Symmetric functions: tools} \label{sec:SF_tools}
%\section{Symmetric function tools} \label{sec:SF_tools}

In this subsection we introduce more tools from symmetric function theory, that we are going to use in our proofs.

\subsubsection{More symmetric function notation}

It is useful to introduce the so called \emph{star scalar product} on $\Lambda$ given by
$$
\langle p_{\lambda},p_{\mu} \rangle_*=(-1)^{|\mu|-|\lambda|}\prod_{i=1}^{\ell(\mu)}(1-q^{\mu_i})(1-t^{\mu_i}) z_{\mu}\chi(\mu=\lambda),
$$
where $\chi(\mathcal{P})=1$ if the statement $\mathcal{P}$ is true, and $\chi(\mathcal{P})=0$ otherwise.

For every symmetric function $f[X]$ and $g[X]$ we have (see \cite[Proposition~1.8]{Garsia-Haiman-Tesler-Explicit-1999})
\begin{equation} \label{eq:starprod}
\langle f,g\rangle_*= \langle \omega \phi f,g\rangle=\langle \phi \omega f,g\rangle
\end{equation}
where 
\begin{equation}
\phi f[X] \coloneqq f[MX]\qquad \text{ for all } f[X]\in \Lambda.
\end{equation}
Observe that
\begin{equation}
f^*=f^*[X]=f\left[\frac{X}{M}\right]=\phi^{-1}f[X].
\end{equation}
For all symmetric functions $f,g,h$ we have
\begin{equation} \label{eq:hperp_estar_adjoint}
\langle h^\perp f,g\rangle_*=\langle h^\perp f,\omega \phi g\rangle=\langle f,h\omega \phi g\rangle=\langle f,\omega \phi ( (\omega h)^* \cdot g)\rangle =\langle f,  (\omega h)^* \cdot g \rangle_*,
\end{equation}
so the operator $h^\perp$ is the adjoint of the multiplication by $(\omega h)^*$ with respect to the star scalar product.

We record here the addition formulas 
\begin{equation}
p_k[X+Y]=p_k[X]+p_k[Y]\quad \text{ and } \quad p_k[X-Y]=p_k[X]-p_k[Y],
\end{equation}
and
\begin{equation} \label{eq:e_h_sum_alphabets}
e_n[X+Y]=\sum_{i=0}^ne_{n-i}[X]e_i[Y]\quad \text{ and } \quad  h_n[X+Y]=\sum_{i=0}^nh_{n-i}[X]h_i[Y].
\end{equation}
Notice in particular that $p_k[-X]$ equals $-p_k[X]$ and not $(-1)^kp_k[X]$. As the latter sort of negative sign can be also useful, it is customary to use the notation $\epsilon$ to express it: we will have $p_k[\epsilon X] = (-1)^k p_k[X]$, so that, in general, 
\begin{equation} \label{eq:minusepsilon}
f[-\epsilon X] = \omega f[X]
\end{equation} 
for any symmetric function $f$.

We refer to \cite{Haglund-Book-2008} for more informations on this topic.

\subsubsection{Macdonald symmetric functions fundamental identities}

It turns out that the Macdonald polynomials are orthogonal with respect to the star scalar product: more precisely
\begin{equation} \label{eq:H_orthogonality}
\langle \widetilde{H}_{\lambda},\widetilde{H}_{\mu}\rangle_*=w_{\mu}(q,t)\chi(\lambda=\mu).
\end{equation}
These orthogonality relations give the following Cauchy identities
\begin{equation} \label{eq:Mac_Cauchy}
e_n\left[\frac{XY}{M}\right]=\sum_{\mu\vdash n} \frac{\widetilde{H}_\mu[X]\widetilde{H}_\mu[Y]}{w_\mu}\quad \text{ for all }n.
\end{equation}

We will use the following form of \emph{Macdonald-Koornwinder reciprocity} (see \cite{Macdonald-Book-1995} p. 332 or \cite{Garsia-Haiman-Tesler-Explicit-1999}): for all nonempty partitions $\alpha$ and $\beta$
\begin{equation} \label{eq:Macdonald_reciprocity}
\frac{\widetilde{H}_{\alpha}[MB_{\beta}]}{\Pi_{\alpha}}=\frac{\widetilde{H}_{\beta}[MB_{\alpha}]}{\Pi_{\beta}}.
\end{equation}

\subsubsection{Pieri rules and summation formulae}

For a given $k\geq 1$, we define the Pieri coefficients $c_{\mu \nu}^{(k)}$ and $d_{\mu \nu}^{(k)}$ by setting
\begin{align}
\label{eq:def_cmunu} h_{k}^\perp \widetilde{H}_{\mu}[X] & =\sum_{\nu \subset_k \mu} c_{\mu \nu}^{(k)}\widetilde{H}_{\nu}[X],\\
\label{eq:def_dmunu} e_{k}\left[\frac{X}{M}\right] \widetilde{H}_{\nu}[X] & =\sum_{\mu \supset_k \nu} d_{\mu \nu}^{(k)}\widetilde{H}_{\mu}[X].
\end{align}

The following identity is Proposition~5 in \cite{Bergeron-Haiman-2013}, written in the notation of \cite{Garsia-Haglund-Xin-Zabrocki-Pieri-2016}, which is coherent with ours:
\begin{equation} \label{eq:cmunu_recursion}
c_{\mu \nu}^{(k+1)}=\frac{1}{B_{\mu/\nu}}\sum_{\nu\subset_1 \alpha\subset_k\mu}c_{\mu \alpha}^{(k)}c_{\alpha \nu}^{(1)}\frac{T_{\alpha}}{T_{\nu}}\quad \text{ with }\quad B_{\mu/\nu} \coloneqq B_{\mu}-B_{\nu},
\end{equation}
where $\nu\subset_k \mu$ means that $\nu$ is
contained in $\mu$ (as Ferrers diagrams) and $\mu/\nu$ has $k$ lattice cells, while the
symbol $\mu\supset_k \nu$ is analogously defined. It follows from \eqref{eq:H_orthogonality} that
\begin{equation} \label{eq:rel_cmunu_dmunu}
c_{\mu \nu}^{(k)}=\frac{w_{\mu}}{w_{\nu}}d_{\mu \nu}^{(k)}.
\end{equation}

\subsubsection{Other useful identities}

In this section we collect some results from the literature that we are going to use later in the text.

\bigskip

The following identity is Proposition~2.2 in\cite{Garsia-Haglund-qtCatalan-2002}:
\begin{equation} \label{eq:garsia_haglund_eval}
\widetilde{H}_\mu[(1-t)(1-q^j)]=(1-q^j)\Pi_\mu h_j[(1-t)B_\mu].
\end{equation}
So, using \eqref{eq:Mac_Cauchy} with $Y=[j]_q=\frac{1-q^j}{1-q}$, we get
\begin{align} \label{eq:qn_q_Macexp}
e_n\left[X\frac{1-q^j}{1-q}\right] & =\sum_{\mu\vdash n}\frac{\widetilde{H}_\mu[X]\widetilde{H}_\mu[(1-t)(1-q^j)]}{w_\mu}\\
\notag & =(1-q^j)\sum_{\mu\vdash n}\frac{\Pi_\mu \widetilde{H}_\mu[X]h_j[(1-t)B_\mu]}{w_\mu}.
\end{align}

\bigskip

For $\mu\vdash n$, Macdonald proved (see \cite{Macdonald-Book-1995} p. 362) that
\begin{equation} \label{eq:Mac_hook_coeff_ss}
\langle \widetilde{H}_{\mu},s_{(n-r,1^r)}\rangle=e_r[B_{\mu}-1],
\end{equation}
so that, since by Pieri rule $e_rh_{n-r}=s_{(n-r,1^r)}+s_{(n-r+1,1^{r-1})}$,
\begin{equation} \label{eq:Mac_hook_coeff}
\langle \widetilde{H}_{\mu},e_rh_{n-r}\rangle=e_r[B_{\mu}].
\end{equation}

We need the following well-known proposition.
\begin{proposition} 
	For $n\in \mathbb{N}$ we have
	\begin{align}
	\label{eq:en_expansion}
	e_n[X] = e_n \left[ \frac{XM}{M} \right] = \sum_{\mu \vdash n} \frac{M B_\mu \Pi_{\mu} \widetilde{H}_\mu[X]}{w_\mu}.
	\end{align}
	Moreover, for all $k\in \mathbb{N}$ with $0\leq k\leq n$, we have
	\begin{align}
	\label{eq:e_h_expansion}
	h_k \left[ \frac{X}{M} \right] e_{n-k} \left[ \frac{X}{M} \right] = \sum_{\mu \vdash n} \frac{e_k[B_\mu] \widetilde{H}_\mu[X]}{w_\mu},
	\end{align}
	and
	\begin{align}
	\label{eq:p_expansion}
	\omega (p_n[X]) = [n]_q[n]_t\sum_{\mu \vdash n} \frac{M\Pi_\mu\widetilde{H}_\mu[X]}{w_\mu}.
	\end{align}
\end{proposition}

We will need two more identities: Lemma~5.2 in\cite{DAdderio-VandenWyngaerd-2017}, i.e. for $\beta\vdash n>k\geq 1$
\begin{equation} \label{eq:lem52}
e_{n-k-1}[B_\beta-1]B_\beta=\sum_{\gamma\subset_k \beta}c_{\beta \gamma}^{(k)}B_\gamma T_\gamma ,
\end{equation}
and 
\begin{align} \label{eq:mikeomegaid}
\notag e_{n-k}[B_\beta] & =  T_{\beta}e_{k}[B_\beta(1/q,1/t)] \\
\text{(using \cite{Zabrocki-4Catalan-2016}*{Lemma~13})} & = \sum_{\gamma\subset_{k} \beta}c_{\beta \gamma}^{e_{k}^\perp}(1/q,1/t)T_{\beta} \\
\notag \text{(using \cite{Garsia-Haglund-qtCatalan-2002}*{Equation~(3.16)})}& =\sum_{\gamma\subset_k \beta}c_{\beta \gamma}^{(k)} T_\gamma , 
\end{align}
where $c_{\mu\nu}^{e_{n-k}^\perp}$ is the generalized Pieri coefficient defined by
\begin{equation}
\sum_{\nu\subset_{n-k} \mu}c_{\mu\nu}^{e_{n-k}^\perp}\widetilde{H}_\nu=e_{n-k}^\perp \widetilde{H}_\mu.
\end{equation}

We will use also \cite{Haglund-Schroeder-2004}*{Theorem~2.6}, i.e. for any $A,F\in \Lambda$ homogeneous
\begin{equation} \label{eq:HaglundThm}
\sum_{\mu\vdash n}\Pi_\mu F[MB_\mu]d_{\mu\nu}^A=\Pi_\nu\left(\Delta_{A[MX]}F[X]\right)[MB_\nu],
\end{equation}
where $d_{\mu\nu}^A$ is the generalized Pieri coefficient defined by
\begin{equation}
\sum_{\mu\supset \nu}d_{\mu\nu}^A\widetilde{H}_\mu=A\widetilde{H}_\nu.
\end{equation}

Another identity that we will need is \cite[Proposition~2.6]{Garsia-Hicks-Stout-2011}:
\begin{equation} \label{eq:Garsia_Hicks_Stout}
h_k\left[\frac{X}{1-q}\right]e_{n-k}\left[\frac{X}{M}\right] = \sum_{\mu\vdash n}\frac{\widetilde{H}_\mu[X]}{w_\mu}\sum_{r=1}^k\qbinom{k-1}{r-1}_qq^{\binom{r}{2}+r-kr} (-1)^{k-r} h_r[(1-t)B_\mu].
\end{equation}

Finally, we will use the following theorem from \cite{DAdderio-VandenWyngaerd-2017}.
\begin{theorem}[\cite{DAdderio-VandenWyngaerd-2017}*{Theorem~3.1}]
	For $m,k\geq 1$ and $\ell\geq 0$, we have
	\begin{align} \label{eq:mastereq}
	\sum_{\gamma\vdash m}\frac{\widetilde{H}_\gamma[X]}{w_\gamma} h_k[(1-t)B_\gamma]e_\ell[B_\gamma]  & = \sum_{j=0}^{\ell} t^{\ell-j}\sum_{s=0}^{k}q^{\binom{s}{2}}\qbinom{s+j}{s}_q \qbinom{k+j-1}{s+j-1}_q \\
	\notag & \times h_{s+j}\left[\frac{X}{1-q}\right] h_{\ell-j}\left[\frac{X}{M}\right] e_{m-s-\ell}\left[\frac{X}{M}\right].
	\end{align}
\end{theorem}

%\subsection{Proofs}

%\section{Missing proofs from Section~\ref{sec:Delkm_ops} and Section~\ref{sec:conj_Delk0}} \label{sec:proofs_sec_Delkm}

\subsection{Proof of Theorem~\ref{thm:DeltakmGD}}

Using \eqref{eq:en_expansion}, we have
\begin{align*}
\mathbf{\Pi} e_k^*\mathbf{\Pi}^{-1} \nabla e_{n-k}
& = \mathbf{\Pi} e_k^*\mathbf{\Pi}^{-1} \sum_{\mu\vdash n-k} \frac{MB_\mu\Pi_{\mu}T_\mu}{w_\mu}\widetilde{H}_\mu[X]\\
\text{(using \eqref{eq:def_dmunu})} & = \mathbf{\Pi} \sum_{\mu\vdash n-k} \frac{MB_\mu T_\mu}{w_\mu}\sum_{\lambda\supset_k \mu} d_{\lambda\mu}^{(k)}\widetilde{H}_\lambda[X]\\
\text{(using \eqref{eq:rel_cmunu_dmunu})} & = \mathbf{\Pi} \sum_{\lambda\vdash n} M\sum_{\mu\subset_k \lambda}c_{\lambda\mu}^{(k)}B_\mu T_\mu \frac{\widetilde{H}_\lambda[X]}{w_\lambda}\\
\text{(using \eqref{eq:lem52})}& = \mathbf{\Pi} \sum_{\lambda\vdash n} M e_{n-k-1}[B_\lambda-1] B_\lambda \frac{\widetilde{H}_\lambda[X]}{w_\lambda}\\
& =  \sum_{\lambda\vdash n} M e_{n-k-1}[B_\lambda-1] \Pi_\lambda B_\lambda \frac{\widetilde{H}_\lambda[X]}{w_\lambda}\\
\text{(using \eqref{eq:en_expansion})}& =  \Delta_{e_{n-k-1}}'e_n .
\end{align*}

%The proof of Theorem~\ref{thm:DeltakmDSq} is similar.
\subsection{Proof of Theorem~\ref{thm:DeltakmDSq}}

Using \eqref{eq:p_expansion}, we have
\begin{align*}
\mathbf{\Pi} e_k^*\mathbf{\Pi}^{-1} \nabla \frac{[n]_q}{[n-k]_q}\omega(p_{n-k})
& = [n]_q[n-k]_t\mathbf{\Pi} e_k^*\mathbf{\Pi}^{-1} \sum_{\mu\vdash n-k} \frac{M \Pi_{\mu}T_\mu}{w_\mu}\widetilde{H}_\mu[X]\\
\text{(using \eqref{eq:def_dmunu})} & = [n]_q[n-k]_t\mathbf{\Pi} \sum_{\mu\vdash n-k} \frac{M T_\mu}{w_\mu}\sum_{\lambda\supset_k \mu} d_{\lambda\mu}^{(k)}\widetilde{H}_\lambda[X]\\
\text{(using \eqref{eq:rel_cmunu_dmunu})} & = [n]_q[n-k]_t \mathbf{\Pi} \sum_{\lambda\vdash n} M\sum_{\mu\subset_k \lambda}c_{\lambda\mu}^{(k)} T_\mu \frac{\widetilde{H}_\lambda[X]}{w_\lambda}\\
\text{(using \eqref{eq:mikeomegaid})}& = [n]_q[n-k]_t \mathbf{\Pi} \sum_{\lambda\vdash n} Me_{n-k}[B_\lambda] \frac{\widetilde{H}_\lambda[X]}{w_\lambda}\\
& =  [n]_q[n-k]_t \sum_{\lambda\vdash n} Me_{n-k}[B_\lambda]\Pi_\lambda  \frac{\widetilde{H}_\lambda[X]}{w_\lambda}\\
\text{(using \eqref{eq:p_expansion})}& =\frac{[n-k]_t}{[n]_t}  \Delta_{e_{n-k}}\omega(p_n) .
\end{align*}

\subsection{Proof of Proposition~\ref{prop:propDel10}}

It is enough to check \eqref{eq:propDel10} on the Macdonald basis: for $\mu\vdash n$
\begin{align*}
\Theta_1 \widetilde{H}_\mu[X] & = \mathbf{\Pi}e_1^*\mathbf{\Pi}^{-1}\widetilde{H}_\mu[X] \\
\text{(using \eqref{eq:def_dmunu})}& = \mathbf{\Pi}\sum_{\lambda\supset_1 \mu} d_{\lambda\mu}^{(1)}\frac{1}{\Pi_\mu}\widetilde{H}_\lambda[X] \\
& = \sum_{\lambda\supset_1 \mu} d_{\lambda\mu}^{(1)}\frac{\Pi_\lambda}{\Pi_\mu}\widetilde{H}_\lambda[X]\\
& = \sum_{\lambda\supset_1 \mu} d_{\lambda\mu}^{(1)}\left(1-\frac{D_\lambda-D_\mu}{M}\right)\widetilde{H}_\lambda[X]\\
& = \sum_{\lambda\supset_1 \mu} d_{\lambda\mu}^{(1)}\widetilde{H}_\lambda[X]-\frac{1}{M} \sum_{\lambda\supset_1 \mu} d_{\lambda\mu}^{(1)} D_\lambda \widetilde{H}_\lambda[X] +\frac{1}{M} \sum_{\lambda\supset_1 \mu} d_{\lambda\mu}^{(1)}D_\mu \widetilde{H}_\lambda[X]\\
\text{(using \eqref{eq:def_dmunu} and \eqref{eq:D0})} & = \frac{1}{M}\underline{e}_1\widetilde{H}_\mu[X] +\frac{1}{M} D_0\frac{1}{M}\underline{e}_1\widetilde{H}_\lambda[X] -\frac{1}{M} \frac{1}{M}\underline{e}_1D_0 \widetilde{H}_\mu[X]\\
& = \frac{1}{M}\left(\underline{e}_1 +\frac{1}{M} (D_0\underline{e}_1 - \underline{e}_1D_0) \right)\widetilde{H}_\lambda[X]\\
\text{(using \eqref{eq:[Dk_e1]})}& =\frac{1}{M}(\underline{e}_1+D_1)\widetilde{H}_\lambda[X].
\end{align*}

\subsection{Proof of Theorem~\ref{thm:SF_GenShuffle}}

We have

\begin{align*}
& \hspace{-0.5cm} h_j^\perp \nabla e_n[X[s+1]_q]=\\
\text{(using \eqref{eq:qn_q_Macexp})}& =\sum_{\lambda\vdash n}(1-q^{s+1})h_{s+1}[(1-t)B_\lambda] \Pi_\lambda T_\lambda h_j^\perp\frac{\widetilde{H}_\lambda[X]}{w_\lambda}\\
\text{(using \eqref{eq:def_cmunu})}& =\sum_{\mu\vdash n-j} \widetilde{H}_\mu[X](1-q^{s+1}) \sum_{\lambda\vdash n}\frac{c_{\lambda\mu}^{(j)}}{w_\lambda}  h_{s+1}[(1-t)B_\lambda] T_\lambda \Pi_\lambda  \\
\text{(using \eqref{eq:rel_cmunu_dmunu})}& =\sum_{\mu\vdash n-j} \frac{\widetilde{H}_\mu[X]}{w_\mu} (1-q^{s+1}) \sum_{\lambda\vdash n}d_{\lambda\mu}^{(j)} h_{s+1}[(1-t)B_\lambda] T_\lambda \Pi_\lambda \\
\text{(using \eqref{eq:HaglundThm})}& =\sum_{\mu\vdash n-j} \frac{\widetilde{H}_\mu[X]}{w_\mu} (1-q^{s+1}) \Pi_\mu\left.\left(\Delta_{e_j}h_{s+1}[Z/(1-q)] e_n[Z/M]\right)\right|_{Z=MB_\mu}  
\end{align*}
\begin{align*}
\text{(using \eqref{eq:Garsia_Hicks_Stout})}& =\sum_{\mu\vdash n-j} \frac{\widetilde{H}_\mu[X]}{w_\mu} (1-q^{s+1}) \Pi_\mu\sum_{\beta\vdash n+s+1}e_j[B_\beta]\frac{\widetilde{H}_\beta[MB_\mu]}{w_\beta}\\
& \times \sum_{i=1}^{s+1}\qbinom{s}{i-1}_qq^{\binom{i}{2}+i-i(s+1)}(-1)^{s+1-i}h_i[(1-t)B_\beta] \\
& =\sum_{\mu\vdash n-j} \frac{\widetilde{H}_\mu[X]}{w_\mu} (1-q^{s+1}) \Pi_\mu\sum_{i=1}^{s+1}\qbinom{s}{i-1}_qq^{\binom{i}{2}+i-i(s+1)}(-1)^{s+1-i} \\
& \times \sum_{\beta\vdash n+s+1}\frac{\widetilde{H}_\beta[MB_\mu]}{w_\beta}h_i[(1-t)B_\beta]e_j[B_\beta] \\
\text{(using \eqref{eq:mastereq})}& =\sum_{\mu\vdash n-j} \frac{\widetilde{H}_\mu[X]}{w_\mu} (1-q^{s+1}) \Pi_\mu\sum_{i=1}^{s+1}\qbinom{s}{i-1}_qq^{\binom{i}{2}+i-i(s+1)}(-1)^{s+1-i} \\
& \times \sum_{p=0}^jt^{j-p}\sum_{b=0}^iq^{\binom{b}{2}}\qbinom{b+p}{p}_q\qbinom{i+p-1}{b+p-1}_qh_{b+p}[(1-t)B_\mu] h_{j-p}[B_\mu] e_{n+s+1-b-j}[B_\mu]  \\
\text{($e_r[B_\mu]=0$ for $r>|\mu|$)}& =\sum_{\mu\vdash n-j} \frac{\widetilde{H}_\mu[X]}{w_\mu} (1-q^{s+1}) \Pi_\mu q^{\binom{s+1}{2}+s+1-(s+1)^2} \times\\
& \times \sum_{p=0}^jt^{j-p} q^{\binom{s+1}{2}}\qbinom{s+1+p}{p}_qh_{s+1+p}[(1-t)B_\mu] h_{j-p}[B_\mu] e_{n-j}[B_\mu] \\
\text{(using \eqref{eq:Bmu_Tmu})}& =\sum_{\mu\vdash n-j} \frac{\widetilde{H}_\mu[X]}{w_\mu} (1-q^{s+1}) \Pi_\mu \sum_{p=0}^jt^{j-p} \qbinom{s+1+p}{p}_qh_{s+1+p}[(1-t)B_\mu] h_{j-p}[B_\mu] T_\mu \\
& =\sum_{\mu\vdash n-j} \frac{\widetilde{H}_\mu[X]}{w_\mu} \Pi_\mu \sum_{p=0}^jt^{j-p} \qbinom{s+p}{p}_q(1-q^{s+1+p}) h_{s+1+p}[(1-t)B_\mu] h_{j-p}[B_\mu] T_\mu \\
\text{(using \eqref{eq:qn_q_Macexp})}& =\sum_{p=0}^jt^{j-p} \qbinom{s+p}{p}_q  \Delta_{h_{j-p}}\nabla e_{n-j}[X[s+1+p]_q] 
\end{align*}
where we used
\[\binom{s+1}{2}+s+1-(s+1)^2=\frac{s^2+s+2s+2-2s^2-4s-2}{2}=-\binom{s+1}{2}.\]

\subsection{Proof of Lemma~\ref{lem:Schroeder_comp}}

We have

\begin{align*}
& \hspace{-0.5cm}\< \Delta_{h_{\ell}}f,h_ke_{n-\ell-k} \>= \\
\text{(using \eqref{eq:starprod})}& =\< \Delta_{h_{\ell}}f,e_k^*h_{n-\ell-k}^* \>_*\\
& =\< f, \Delta_{h_{\ell}}e_k^*h_{n-\ell-k}^* \>_*\\
\text{(using \eqref{eq:e_h_expansion})}& =\sum_{\beta\vdash n-\ell} h_\ell[B_\beta] e_{n-\ell-k}[B_\beta]\< f, \frac{\widetilde{H}_\beta}{w_\beta}\>_* \\
& =\sum_{\beta\vdash n-\ell} \Pi_{\beta}\left.\left(h_\ell^* e_{n-\ell-k}^*\right)\right|_{X=MB_\beta} \< f, \mathbf{\Pi}^{-1}\frac{\widetilde{H}_\beta}{w_\beta}\>_* \\
\text{(using \eqref{eq:e_h_expansion})}& =\sum_{\beta\vdash n-\ell} \Pi_{\beta}\left.\left(\Delta_{e_{\ell}} e_{n-k}^*\right)\right|_{X=MB_\beta} \< \mathbf{\Pi}^{-1}f, \frac{\widetilde{H}_\beta}{w_\beta}\>_* \\
\text{(using \eqref{eq:HaglundThm})}& =\sum_{\beta\vdash n-\ell} \sum_{\mu\supset_\ell \beta} \Pi_{\mu}d_{\mu\beta}^{e_\ell^*}e_{n-k}[B_\mu] \< \mathbf{\Pi}^{-1}f, \frac{\widetilde{H}_\beta}{w_\beta}\>_* \\
\text{(using \eqref{eq:rel_cmunu_dmunu})}& =\sum_{\mu \vdash n}\frac{\Pi_{\mu}}{w_\mu}e_{n-k}[B_\mu] \left\< \mathbf{\Pi}^{-1}f, \sum_{\beta \subset_\ell \mu}  c_{\mu\beta}^{h_\ell^\perp} \widetilde{H}_\beta \right\>_*
\end{align*}
\begin{align*}
\text{(using \eqref{eq:def_cmunu})}& = \sum_{\mu \vdash n}\frac{\Pi_{\mu}}{w_\mu}e_{n-k}[B_\mu]\left\< \mathbf{\Pi}^{-1}f, h_\ell^\perp\widetilde{H}_\mu \right\>_*\\
\text{(using \eqref{eq:hperp_estar_adjoint})}& = \sum_{\mu \vdash n}\frac{1}{w_\mu}e_{n-k}[B_\mu]\left\<\mathbf{\Pi}e_\ell^*\mathbf{\Pi}^{-1}f,  \widetilde{H}_\mu \right\>_* \\
\text{(using \eqref{eq:e_h_expansion})}& = \left\<\mathbf{\Pi}e_\ell^*\mathbf{\Pi}^{-1}f, e_k^* h_{n-k}^*  \right\>_*\\
\text{(using \eqref{eq:starprod})}& = \left\<\mathbf{\Pi}e_\ell^*\mathbf{\Pi}^{-1}f, h_k e_{n-k} \right\>\\
& =\left\<\Theta_\ell f, h_k e_{n-k} \right\>.
\end{align*}

% Bibliography

\bibliographystyle{amsalpha}
\bibliography{Biblebib}

\end{document}